\documentclass[letter,10pt]{amsart}
\usepackage{amssymb,amsmath,amsfonts}
\usepackage{mathrsfs}
\usepackage[left=1 in, right=1 in,top=1 in, bottom=1 in]{geometry}
\usepackage{mathenv}

\newtheorem{theorem}{Theorem}[section]
\newtheorem{lemma}[theorem]{Lemma}

\newtheorem{remark}[theorem]{Remark}

\makeatletter
\renewcommand \theequation {%
\ifnum \c@section>\z@ \@arabic\c@section.%
\fi\@arabic\c@equation} \@addtoreset{equation}{section}
\makeatother

\providecommand{\abs}[1]{\left\vert#1\right\vert}
\providecommand{\nm}[1]{\left\Vert#1\right\Vert}

\providecommand{\br}[1]{\left\langle #1 \right\rangle}
\providecommand{\tm}[2]{\left\Vert#1\right\Vert_{L^2(#2)}}
\providecommand{\im}[2]{\left\Vert#1\right\Vert_{L^{\infty}(#2)}}

\providecommand{\lnnm}[1]{{\left\Vert#1\right\Vert}_{L^{\infty}L^{\infty}}}
\providecommand{\lnm}[1]{\left\Vert#1\right\Vert_{L^{\infty}}}
\providecommand{\tnm}[1]{\left\Vert#1\right\Vert_{L^{2}}}
\providecommand{\tnnm}[1]{{\left\Vert#1\right\Vert}_{L^{2}L^2}}
\providecommand{\ltnm}[1]{{\left\Vert#1\right\Vert}_{L^{\infty}L^{2}}}

\providecommand{\lss}[2]{\left\Vert#1\right\Vert_{L^{\infty}_{#2}}}
\providecommand{\tss}[2]{\left\Vert#1\right\Vert_{L^{2}_{#2}}}

\def\ud{\mathrm{d}}

\def\p{\partial}
\def\ls{\lesssim}

\def\half{\dfrac{1}{2}}
\def\rt{\rightarrow}
\def\r{\mathbb{R}}
\def\no{\nonumber}
\def\ue{\mathrm{e}}

\def\ds{\displaystyle}

\def\u{U}
\def\ub{\mathscr{U}}
\def\bu{\bar U}
\def\bub{\bar{\mathscr{U}}}
\def\uf{\mathfrak{U}}
\def\buf{\bar{\mathfrak{U}}}
\def\uu{\mathcal{U}}
\def\buu{\bar{\mathcal{U}}}

\def\e{\epsilon}
\def\s{\mathbb{S}}
\def\vx{\vec x}
\def\vw{\vec w}

\def\nx{\nabla_{x}}

\def\l{\lambda}
\def\ll{\mathcal{L}}

\def\q{Q}
\def\qb{\mathscr{Q}}
\def\qf{\mathfrak{Q}}

\def\v{\mathscr{V}}

\def\d{\delta}

\def\vn{\vec\nu}

\def\t{\mathcal{T}}
\def\k{\mathcal{K}}
\def\a{\mathscr{A}}

\def\rk{R_{\kappa}}
\def\id{{\bf{1}}}

\def\rr{\mathscr{R}}

\def\gb{\mathscr{G}}
\def\gf{\mathfrak{G}}

\begin{document}

\title{Boundary Layer of Transport Equation with In-Flow Boundary}

\author[L. Wu]{Lei Wu}
\address[L. Wu]{
   \newline\indent Department of Mathematical Sciences, Carnegie Mellon University
\newline\indent Pittsburgh, PA 15213, USA}
\email{zjkwulei1987@gmail.com}


\begin{abstract}
Consider the steady neutron transport equation in 2D convex domains with in-flow boundary condition. In this paper, we establish the diffusive limit while the boundary layers are present. Our contribution relies on a delicate decomposition of boundary data to separate the regular and singular boundary layers, novel weighted $W^{1,\infty}$ estimates for the Milne problem with geometric correction in convex domains, as well as an $L^{2m}-L^{\infty}$ framework which yields stronger remainder estimates.
\end{abstract}

\maketitle

\tableofcontents

\newpage


\pagestyle{myheadings} \thispagestyle{plain} \markboth{LEI WU}{BOUNDARY LAYER OF TRANSPORT EQUATION}

\section{Introduction}

\subsection{Problem Formulation}

We consider the steady neutron transport equation in a
two-dimensional bounded convex domain with in-flow boundary. In the spacial
domain $\vx=(x_1,x_2)\in\Omega$ where $\p\Omega\in C^3$ and the velocity domain
$\vw=(w_1,w_2)\in\s^1$, the neutron density $u^{\e}(\vx,\vw)$
satisfies
\begin{eqnarray}\label{transport}
\left\{
\begin{array}{rcl}\displaystyle
\e \vw\cdot\nabla_x u^{\e}+u^{\e}-\bar u^{\e}&=&0\ \ \text{in}\ \ \Omega,\\
u^{\e}(\vx_0,\vw)&=&g(\vx_0,\vw)\ \ \text{for}\
\ \vw\cdot\vn<0\ \ \text{and}\ \ \vx_0\in\p\Omega,
\end{array}
\right.
\end{eqnarray}
where
\begin{eqnarray}\label{average}
\bar u^{\e}(\vx)=\frac{1}{2\pi}\int_{\s^1}u^{\e}(\vx,\vw)\ud{\vw},
\end{eqnarray}
$\vn$ is the outward unit normal vector, with the Knudsen number $0<\e<<1$.
We intend to study the behavior of $u^{\e}$ as $\e\rt0$.

Based on the flow direction, we can divide the boundary $\Gamma=\{(\vx,\vw): \vx\in\p\Omega\}$ into
the in-flow boundary $\Gamma^-$, the out-flow boundary $\Gamma^+$
and the grazing set $\Gamma^0$ as
\begin{eqnarray}
\Gamma^{-}&=&\{(\vx,\vw): \vx\in\p\Omega,\ \vw\cdot\vn<0\}\\
\Gamma^{+}&=&\{(\vx,\vw): \vx\in\p\Omega,\ \vw\cdot\vn>0\}\\
\Gamma^{0}&=&\{(\vx,\vw): \vx\in\p\Omega,\ \vw\cdot\vn=0\}
\end{eqnarray}
It is easy to see that $\Gamma=\Gamma^+\cup\Gamma^-\cup\Gamma^0$.
Hence, the boundary condition is only given for $\Gamma^{-}$.

\subsection{Background and Method}

\subsubsection{Asymptotic Analysis}

Diffusive limits, or more general hydrodynamic limits, are central to connecting the kinetic theory and fluid mechanics. The basic idea is to consider the asymptotic behaviors of the solutions to Boltzmann equation, transport equation, or Vlasov systems. Since early 20th century, this type of problems have been extensively studied in many different settings: steady or unsteady, linear or nonlinear, strong solution or weak solution, etc.

Among all these variations, one of the simplest but most important models --- neutron transport equation in bounded domains, has attracted a lot of attention since the dawn of atomic age. The neutron transport equation is usually regarded as a linear prototype of the more complicated nonlinear Boltzmann equation, and thus, is an ideal starting point to develop new theories and techniques.
We refer to
\cite{Larsen1974=}, \cite{Larsen1974}, \cite{Larsen1975}, \cite{Larsen1977}, \cite{Larsen.D'Arruda1976}, \cite{Larsen.Habetler1973}, \cite{Larsen.Keller1974}, \cite{Larsen.Zweifel1974}, \cite{Larsen.Zweifel1976}, \cite{Li.Lu.Sun2015=}, \cite{Li.Lu.Sun2015} for more details.

For steady neutron transport equation, the exact solution can be approximated by the sum of an interior solution $\u$ and a boundary layer $\uu$. The interior solution satisfies certain fluid equations or thermodynamic equations, and the boundary layer satisfies a half-space kinetic equation, which decays rapidly when it is away from the boundary.

The justification of diffusive limit usually involves two steps:
\begin{enumerate}
\item
Expanding $\u=\ds\sum_{k=0}^{\infty}\e^k\u_k$ and $\uu=\ds\sum_{k=0}^{\infty}\e^k\uu_k$ as power series of $\e$ and proving the coefficients $\u_k$ and $\uu_k$ are well-defined. Traditionally, the estimates of interior solutions $\u_k$ are relatively straightforward. On the other hand, boundary layers $\uu_k$ satisfy one-dimensional half-space problems which lose some key structures of the original equations. The well-posedness of boundary layer equations are sometimes extremely difficult and it is possible that they are actually ill-posed (e.g. certain type of Prandtl layers).
\item
Proving that $R=u^{\e}-\u_0-\uu_0=o(1)$ as $\e\rt0$. Ideally, this should be done just by expanding to the leading-order level $\u_0$ and $\uu_0$. However, in singular perturbation problems, the estimates of the remainder $R$ usually involves negative powers of $\e$, which requires expansion to higher order terms $\u_N$ and $\uu_N$ for $N\geq1$ such that we have sufficient power of $\e$. In other words, we define $R=u^{\e}-\ds\sum_{k=0}^{N}\e^k\u_k-\ds\sum_{k=0}^{N}\e^k\uu_k$ for $N\geq1$ instead of $R=u^{\e}-\u_0-\uu_0$ to get better estimate of $R$.
\end{enumerate}

\subsubsection{Classical Approach}

The construction of kinetic boundary layers
has long been believed to be satisfactorily solved since Bensoussan, Lions and Papanicolaou published their remarkable paper \cite{Bensoussan.Lions.Papanicolaou1979} in 1979. Their formulation, based on the flat Milne problem, was later extended to treat the nonlinear Boltzmann equation (see \cite{Sone2002} and \cite{Sone2007}).

In detail, in $\Omega$, let $\eta\in[0,\infty)$ denote the rescaled normal variable with respect to the boundary, $\tau\in[-\pi,\pi)$ the tangential variable, and $\phi\in[-\pi,\pi)$ the velocity variable defined in (\ref{substitution 1}), (\ref{substitution 2}), (\ref{substitution 3}), and (\ref{substitution 4}).
The boundary layer $\uu_0$ satisfies the flat Milne problem,
\begin{eqnarray}
\sin\phi\frac{\p \uu_0}{\p\eta}+\uu_0-\bar\uu_0&=&0.
\end{eqnarray}

Unfortunately, in \cite{AA003}, we demonstrated that both the proof and results of this formulation are invalid due to a lack of regularity in estimating $\dfrac{\p\uu_0}{\p\tau}$. This pulls the whole research back to the starting point, and any later results based on this type of boundary layers should be reexamined.

To be more specific, the remainder estimates require that $\uu_1\in L^{\infty}$ which needs $\dfrac{\p\uu_0}{\p\tau}\in L^{\infty}$. However, though \cite{Bensoussan.Lions.Papanicolaou1979} shows that $\uu_0\in L^{\infty}$, it does not necessarily mean that $\dfrac{\p \uu_0}{\p\eta}\in L^{\infty}$. Furthermore, this singularity $\dfrac{\p \uu_0}{\p\eta}\notin L^{\infty}$ will be transferred to $\dfrac{\p \uu_0}{\p\tau}\notin L^{\infty}$. A careful construction of boundary data justifies this invalidity, i.e. the chain of estimates
\begin{eqnarray}
R=o(1)\ \Leftarrow\ \uu_1\in L^{\infty}\ \Leftarrow\ \dfrac{\p\uu_0}{\p\tau}\in L^{\infty}\ \Leftarrow\ \dfrac{\p \uu_0}{\p\eta}\in L^{\infty},
\end{eqnarray}
is broken since the rightmost estimate is wrong.

\subsubsection{Geometric Correction}

While the classical method breaks down, a new approach with geometric correction to the boundary layer construction has been developed to ensure regularity in the cases of disk and annulus in \cite{AA003} and \cite{AA006}. The new boundary layer $\uu_0$ satisfies the $\e$-Milne problem with geometric correction,
\begin{eqnarray}
\sin\phi\frac{\p \uu_0}{\p\eta}+\frac{\e}{R_{\kappa}-\e\eta}\cos\phi\frac{\p
\uu_0}{\p\phi}+\uu_0-\bar\uu_0&=&0,
\end{eqnarray}
where $R_{\kappa}$ is the radius of curvature at boundary. We proved that the solution recovers the well-posedness and exponential decay as in flat Milne problem, and the regularity in $\tau$ is indeed improved, i.e. $\dfrac{\p\uu_0}{\p\tau}\in L^{\infty}$.

However, this new method fails to treat more general domains. Roughly speaking, we have two contradictory goals to achieve:
\begin{itemize}
\item
To prove diffusive limits, the remainder estimates require higher-order regularity estimates of the boundary layer.
\item
The geometric correction $\dfrac{\e}{R_{\kappa}-\e\eta}\cos\phi\dfrac{\p
\uu_0}{\p\phi}$ in the boundary layer equation is related to the curvature of the boundary curve, which prevents higher-order regularity estimates.
\end{itemize}
In other words, the improvement of regularity is still not enough to close the proof. We may analyze the effects of different domains and formulations as follows:
\begin{itemize}
\item
In the absence of the geometric correction $\dfrac{\e}{R_{\kappa}-\e\eta}\cos\phi\dfrac{\p\uu_0}{\p\phi}$, which is the flat Milne problem as in \cite{Bensoussan.Lions.Papanicolaou1979}, the key tangential derivative $\dfrac{\p\uu_0}{\p\tau}$ is not bounded. Therefore, the expansion breaks down.
\item
In the domain of disk or annulus, when $R_{\kappa}$ is constant, as in \cite{AA003} and \cite{AA006}, $\dfrac{\p\uu_0}{\p\tau}$ is bounded, since the tangential derivative $\dfrac{\p}{\p\tau}$ commutes with the equation, and thus we do not even need to estimate $\dfrac{\p \uu_0}{\p\eta}$.
\item
For general smooth convex domains, when $R_{\kappa}$ is a function of $\tau$, $\dfrac{\p\uu_0}{\p\tau}$ relates to the normal derivative $\dfrac{\p\uu_0}{\p\eta}$, which has been shown possibly unbounded in \cite{AA003}. Therefore, we get stuck again at the regularity estimates.
\end{itemize}

\subsubsection{Diffusive Boundary}

In \cite{AA007} and \cite{AA009}, for the case of diffusive boundary, the above argument is pushed from both sides, i.e. improvements in remainder estimates and boundary layer regularity.

In detail, consider the boundary layer expansion
\begin{eqnarray}
\uu(\eta,\tau,\vw)\sim \uu_0(\eta,\tau,\vw)+\e\uu_1(\eta,\tau,\vw).
\end{eqnarray}
The diffusive boundary condition leads to an important simplification that $\uu_0=0$. As \cite{AA003} stated, the next-order boundary layer $\uu_1$ must formally satisfy
\begin{eqnarray}
\sin\phi\frac{\p \uu_1}{\p\eta}+\frac{\e}{R_{\kappa}-\e\eta}\cos\phi\frac{\p
\uu_1}{\p\phi}+\uu_1-\bar\uu_1&=&0.
\end{eqnarray}
Naturally, the diffusive limit requires the estimate of $\dfrac{\p \uu_1}{\p\tau}$. Here, a key observation is that $W=\dfrac{\p \uu_1}{\p\tau}$ satisfies
\begin{eqnarray}
\sin\phi\frac{\p W}{\p\eta}+\frac{\e}{R_{\kappa}-\e\eta}\cos\phi\frac{\p
W}{\p\phi}+W-\bar W&=&-\frac{\p_{\tau}R_{\kappa}}{R_{\kappa}-\e\eta}\bigg(\frac{\e}{R_{\kappa}-\e\eta}\cos\phi\frac{\p
\uu_1}{\p\phi}\bigg).
\end{eqnarray}
Note that the right-hand side is part of the $\uu_1$ equation and its estimate depends on $\sin\phi\dfrac{\p \uu_1}{\p\eta}$. In other words, the estimate of $\dfrac{\p \uu_1}{\p\tau}$ depends on $\sin\phi\dfrac{\p \uu_1}{\p\eta}$, not just $\dfrac{\p \uu_1}{\p\eta}$ which is possibly unbounded. The $\sin\phi$ is crucial to eliminate the singularity. This forms the major proof in \cite{AA007} and \cite{AA009}, i.e. the weighted regularity of $\uu_1$.

Our main idea is
to delicate track $\uu_1$ along the characteristics in the mild formulation, and prove the weighted $W^{1,\infty}$ estimates of the boundary layer.
In particular, we showed that $\dfrac{\p\uu_1}{\p\tau}$ is bounded even when $R_{\kappa}$ is not constant for general convex domains.

Furthermore, with a novel $L^{2m}-L^{\infty}$ framework, we justified an almost optimal remainder estimate to reduce the further regularity requirement of $\uu_1$.

In summary, in \cite{AA007} and \cite{AA009}, we proved the diffusive limit that $u^{\e}$ converges to the solution of a Laplace's equation with Neumann boundary condition.

\subsubsection{In-Flow Boundary and Basic Ideas}

It is notable that, for the case of in-flow boundary as equation (\ref{transport}), the situation is much worse. The leading-order boundary layer $\uu_0$ is no longer zero.
\begin{eqnarray}
\sin\phi\frac{\p \uu_0}{\p\eta}+\frac{\e}{R_{\kappa}-\e\eta}\cos\phi\frac{\p
\uu_0}{\p\phi}+\uu_0-\bar\uu_0&=&0,\\
\sin\phi\frac{\p \uu_1}{\p\eta}+\frac{\e}{R_{\kappa}-\e\eta}\cos\phi\frac{\p
\uu_1}{\p\phi}+\uu_1-\bar\uu_1&=&-\cos\phi\frac{\p\uu_0}{\p\tau}.
\end{eqnarray}
The remainder contains the term $\dfrac{\p\uu_1}{\p\tau}$, which depends on the estimate of $\dfrac{\p^2\uu_0}{\p\tau^2}$. Then we must prove $W^{2,\infty}$ estimates in the boundary layer equation. In principle, this is impossible for general kinetic equations as \cite{Guo.Kim.Tonon.Trescases2013} pointed out.

On the other hand, we have a key observation that actually the singularity that prevents higher-order regularity concentrates in the neighborhood of the grazing set, so it is natural to isolate the singular part from the whole solution and tackle them in different methods.

Inspired by \cite{Li.Lu.Sun2017}, we introduce a new regularization argument. Instead of trying different weighted norms, we may also modify the boundary data and smoothen the boundary layer in this modified problem.

To be precise, we decompose the boundary data $g=\gb+\gf$, such that
\begin{itemize}
\item
the boundary layer $\ub$ with data $\gb$, which we call regular boundary layer, attains second-order regularity in the tangential direction, i.e. $\dfrac{\p^2\ub}{\p\tau^2}\in L^{\infty}$; $\gb=g$ in most of the region except a small neighborhood of the grazing set in order to strengthen the smoothness of $\ub$;
\item
the boundary layer $\uf$ with data $\gf$, which we call singular boundary layer, attains only first-order regularity in the tangential direction i.e. $\dfrac{\p\uf}{\p\tau}\in L^{\infty}$, but the support of $\gf$ is restricted to a very small neighborhood of the grazing set with diameter $\e^{\alpha}$ for some $0<\alpha<1$.
\end{itemize}
In other words, for the remainder estimates, the extra power of $\e$ comes from two sources: $\ub$ gains power by expanding to the higher order, and $\uf$ gains power through a small support $\e^{\alpha}$.

Definitely, this decomposition comes with a price. Even if we assume $\dfrac{\p g}{\p\phi}=O(1)$, after the decomposition, we can at most have $\dfrac{\p\gb}{\p\phi}=O(\e^{-\alpha})$ and $\dfrac{\p\gf}{\p\phi}=O(\e^{-\alpha})$. We have to prove a much stronger weighted $W^{1,\infty}$ estimates to suppress such loss of power in $\e$. Moreover, this decomposition introduces two contradictory goals in the estimates:
\begin{itemize}
\item
to obtain $W^{2,\infty}$ estimate of $\ub$ with data $\gb$, we want $\alpha$ to be as small as possible; the smaller $\alpha$ is (better smoothness of $\gb$), the better estimates we get;
\item
to obtain $W^{1,\infty}$ estimate of $\uf$ with data $\gf$, we want $\alpha$ to be as large as possible; the larger $\alpha$ is (smaller support of $\gf$), the better estimates we get.
\end{itemize}
This balance is quite delicate and the estimates for the $\e$-Milne problem with geometric correction in  \cite{AA003}, \cite{AA006}, \cite{AA007} and \cite{AA009} are not sufficient. We have to start from scratch and prove the stronger version.

\subsubsection{Main Methods}

To fully solve such a problem, we need an intricate synthesis of previously developed methods, and the fresh regularization argument stated as above.

We inherit and modify the following ideas and techniques, which can be considered the minor contribution:
\begin{itemize}
\item
{\bf Geometric Correction:}\\
The $\e$-Milne problem with geometric correction for $f=\ub$ or $\uf$,
\begin{eqnarray}\label{intro 11}
\sin\phi\frac{\p f}{\p\eta}+\frac{\e}{R_{\kappa}-\e\eta}\cos\phi\frac{\p
f}{\p\phi}+f-\bar f&=&S,
\end{eqnarray}
has been shown to be the correct formulation to describe kinetic boundary layers (see \cite{AA003}). In this paper, we start from scratch and justify the detailed dependence on the source term $S$. In particular, we isolate the contribution of $\bar S$ and $S-\bar S$.
\item
{\bf Canonical Weighted $W^{1,\infty}$ Estimates of Boundary Layers:}\\
The weighted $W^{1,\infty}$ estimates in $\e$-Milne problem with geometric correction is the key to estimate $\dfrac{\p f}{\p\tau}$ (see \cite{AA007}). In this paper, we highlight the dependence on the characteristic curves and the boundary data. The convexity and the kinetic distance
\begin{eqnarray}
\zeta(\eta,\phi)=\Bigg(1-\bigg(\frac{\rk-\e\eta}{\rk}\cos\phi\bigg)^2\Bigg)^{\frac{1}{2}},
\end{eqnarray}
is key to this proof.
\item
{\bf Remainder Estimates:}\\
This is the key step to reduce the regularity requirement in boundary layers. It is originally developed in \cite{AA003} and later strengthened in \cite{AA007}. In the remainder equation for $R(\vx,\vw)=u^{\e}-\u-\uu$,
\begin{eqnarray}
\e\vw\cdot\nx
R+R-\bar R&=&S,
\end{eqnarray}
the estimate justified in \cite{AA003} using $L^2-L^{\infty}$ framework is
\begin{eqnarray}
\lnm{R}\ls \frac{1}{\e^3}\tnm{S}+\text{higher order terms}.
\end{eqnarray}
We intend to show that $\lnm{R}=o(1)$ as $\e\rt0$. Since we cannot expand to higher-order boundary layers to further improve $S$, the coefficients $\e^{-3}$ is too singularity. A key improvement in \cite{AA007} for diffusive boundary case is to develop the $L^{2m}-L^{\infty}$ framework to prove a stronger estimate for $m\geq2$,
\begin{eqnarray}
\lnm{R}\ls \frac{1}{\e^{2+\frac{1}{m}}}\nm{S}_{L^{\frac{2m}{2m-1}}}+\text{higher order terms}.
\end{eqnarray}
In this paper, we adapt it to treat the in-flow boundary case with a modified $L^{2m}-L^{\infty}$ framework. The main idea is to introduce a special test function in the weak formulation to treat $\bar R$ and $R-\bar R$ separately, and further to bootstrap to improve the $L^{\infty}$ estimate by a modified double Duhamel's principle. The proof relies on a delicate analysis using interpolation and Young's inequality.
\end{itemize}
The key novelty of this paper lies in the innovative regularization argument and the corresponding regularity estimates, which constitute the major contribution:
\begin{itemize}
\item
{\bf Improved Weighted $W^{1,\infty}$ Estimates of Boundary Layers:}\\
We combine several different formulations to track the characteristics and justify that the solution of (\ref{intro 11}) satisfies
\begin{eqnarray}
&&\lnnm{\zeta\frac{\p f}{\p\eta}}+\lnnm{\frac{\e}{R_{\kappa}-\e\eta}\cos\phi\frac{\p f}{\p\phi}}\\
&\leq&C\abs{\ln(\e)}^8\bigg(\lnm{p}+\lnm{(\e+\zeta)\dfrac{\p p}{\p\phi}}+\lnnm{S}+\lnnm{\zeta\dfrac{\p S}{\p\eta}}+\lnnm{f}\bigg).\no
\end{eqnarray}
where the boundary data $p=\gb$ or $\gf$. The extra weight $\e+\zeta$ suppresses the singularity in $\dfrac{\p\gb}{\p\phi}$ and $\dfrac{\p\gf}{\p\phi}$. In particular, the estimate does not depend on $\dfrac{\p S}{\p\phi}$. This is the key step to isolate the contribution of $\sin\dfrac{\p f}{\p\eta}$ and $\dfrac{\e}{R_{\kappa}-\e\eta}\cos\phi\dfrac{\p f}{\p\phi}$, which is crucial to later $W^{2,\infty}$ estimates.

The estimate is obtained through a delicate absorbing argument and novel characteristic analysis in half-space kinetic equations.

\item
{\bf $\dfrac{\p^2}{\p\tau^2}$ Estimate of Regular Boundary Layer}:\\
\cite{Guo.Kim.Tonon.Trescases2013} pointed out that weighted $W^{2,\infty}$ estimates of general kinetic equations is not available. This is true even for $\ub$ with modified boundary data. In principle, we cannot bound $\dfrac{\p^2\ub_0}{\p\eta^2}$ and $\dfrac{\p^2\ub_0}{\p\phi^2}$. Instead, we propose a delicate analysis to show that we can estimate $\dfrac{\p\ub_1}{\p\tau}$ without referring to the other second-order derivatives. This is quite unusual and cannot be done in a direct fashion.

Roughly speaking, we need a chain of estimates
\begin{eqnarray}
\lnnm{\dfrac{\p\ub_1}{\p\tau}}&\Leftarrow& \lnnm{\frac{\p}{\p\tau}\left(\dfrac{\p\ub_0}{\p\tau}\right)} \\
&\Leftarrow& \lnnm{\zeta\frac{\p}{\p\eta}\left(\dfrac{\p\ub_0}{\p\tau}\right)}
+\lnnm{\frac{\e}{R_{\kappa}-\e\eta}\cos\phi\frac{\p}{\p\phi}\left(\dfrac{\p\ub_0}{\p\tau}\right)}\no\\
&\Leftarrow&\lnnm{\frac{\e}{R_{\kappa}-\e\eta}\cos\phi\frac{\p}{\p\phi}\left(\dfrac{\p\ub_0}{\p\eta}\right)}\no\\
&\Leftarrow&\lnnm{\frac{\e}{R_{\kappa}-\e\eta}\cos\phi\dfrac{\p\ub_0}{\p\phi}}.\no
\end{eqnarray}
Here, none of these steps are direct application of above improved weighted $W^{1,\infty}$ estimates. Instead, we need careful arrangement of dependent terms and utilize absorbing argument in a delicate way. Eventually, we can justify that
\begin{eqnarray}
\lnnm{\dfrac{\p\ub_1}{\p\tau}}\sim \e^{-\alpha}.
\end{eqnarray}
\item
{\bf $\dfrac{\p}{\p\tau}$ Estimate with Smallness of Singular Boundary Layer}:\\
Here, the major difficulty is how to preserve the smallness of boundary data. The key observation is that in our proof of well-posedness and $W^{1,\infty}$ estimates, we only use two types of quantities: the integral in $\phi$ and the value along the characteristics. Therefore, we introduce a domain decomposition as $\chi_1: \zeta\leq\e^{\alpha}$ and $\chi_2:\zeta\geq\e^{\alpha}$, and estimate $\uf$ in each domain separately.
\begin{enumerate}
\item
$\chi_1$: since $\gf=O(1)$, we know $\uf=O(1)$ whose major contribution is from the boundary data, so it is relatively large but is only restricted to a small domain for $\alpha>0$.
\item
$\chi_2$: since $\gf=0$, we know $\uf=O(\e^{\alpha})$ whose major contribution is from the non-local operator $\buf$, so it is relatively small and spread most of the domain.
\end{enumerate}
In the remainder estimate, the estimates of $\uf$ is in $L^{\frac{2m}{2m-1}}$, so we can combine these two contribution in an integral to obtain smallness
\begin{eqnarray}
\nm{\dfrac{\p\uf_0}{\p\tau}}_{L^{\frac{2m}{2m-1}}}\sim \e^{1-\frac{1}{2m}+\frac{(2m-1)\alpha}{2m}}.
\end{eqnarray}
\end{itemize}
Applying these new techniques, we successfully obtain the diffusive limit that $u^{\e}$ converges to the solution of a Laplace's equation with Dirichlet boundary condition.

\subsection{Main Theorem}

\begin{theorem}\label{main}
Assume $g(\vx_0,\vw)\in C^3(\Gamma^-)$. Then for the steady neutron
transport equation (\ref{transport}), there exists a unique solution
$u^{\e}(\vx,\vw)\in L^{\infty}(\Omega\times\s^1)$. Moreover,
\begin{eqnarray}
\lim_{\e\rt0}\nm{u^{\e}-\u-\uu}_{L^{\infty}(\Omega\times\s^1)}=0,
\end{eqnarray}
where $\u(\vx)$ satisfies the Laplace equation with Dirichlet boundary condition
\begin{eqnarray}
\left\{
\begin{array}{l}
\Delta_x\u(\vx)=0\ \ \text{in}\
\ \Omega,\\\rule{0ex}{1.5em}
\u(\vx_0)=D(\vx_0)\ \ \text{on}\ \
\p\Omega,
\end{array}
\right.
\end{eqnarray}
and $\uu(\eta,\tau,\phi)$ satisfies the $\e$-Milne problem with geometric correction
\begin{eqnarray}
\left\{
\begin{array}{l}
\sin\phi\dfrac{\p \uu }{\p\eta}-\dfrac{\e}{\rk(\tau)-\e\eta}\cos\phi\dfrac{\p
\uu }{\p\phi}+\uu -\buu =0,\\\rule{0ex}{1.5em}
\uu (0,\tau,\phi)=g(\tau,\phi)-D(\tau)\ \ \text{for}\ \
\sin\phi>0,\\\rule{0ex}{1.5em}
\uu (L,\tau,\phi)=\uu (L,\tau,\rr[\phi]),
\end{array}
\right.
\end{eqnarray}
for $L=\e^{-\frac{1}{2}}$, $\rr[\phi]=-\phi$, $\eta$ the rescaled normal variable, $\tau$ the tangential variable, and $\phi$ the velocity variable.
\end{theorem}
\begin{remark}
Note that the effects of the boundary layer decays very fast when it is away from the boundary. Roughly speaking, this theorem states that for $\vx$ not very close to the boundary, $u^{\e}(\vx,\vw)$ can be approximated by the solution of a Laplace equation with Dirichlet boundary condition.
\end{remark}

\subsection{Notation and Paper Structure}

Throughout this paper, $C>0$ denotes a constant that only depends on
the parameter $\Omega$, but does not depend on the data. It is
referred as universal and can change from one inequality to another.
When we write $C(z)$, it means a certain positive constant depending
on the quantity $z$. We write $a\ls b$ to denote $a\leq Cb$.

This paper is organized as follows: in Section 2, we present the asymptotic analysis of the equation (\ref{transport.}) and introduce the decomposition of boundary layers;  in Section 3, we establish the $L^{\infty}$ well-posedness of
the remainder equation; in Section 4, we prove the well-posedness and decay of the $\e$-Milne problem with geometric correction; in Section 5, we study the weighted regularity of the $\e$-Milne problem with geometric correction; finally, in Section 6, we give a detailed analysis of the asymptotic expansion and prove the main theorem.

\begin{remark}
The general structure of this paper is very similar to that of \cite{AA007} and \cite{AA009}. In particular, Section 3, 4 and 5 seem to be an obvious adaption of the corresponding theorems there. However, we introduce new techniques to delicately improve the results in \cite{AA007}, so it needs a careful handling and a fresh start from scratch.
\end{remark}

\newpage

\section{Asymptotic Analysis}

In this section, we will present the asymptotic expansions of the neutron transport equation
\begin{eqnarray}\label{transport.}
\left\{
\begin{array}{l}\displaystyle
\e \vw\cdot\nabla_x u^{\e}+u^{\e}-\bar u^{\e}=0\ \ \text{in}\ \ \Omega,\\\rule{0ex}{1.0em}
u^{\e}(\vx_0,\vw)=g(\vx_0,\vw)\ \ \text{for}\
\ \vw\cdot\vn<0\ \ \text{and}\ \ \vx_0\in\p\Omega.
\end{array}
\right.
\end{eqnarray}

\subsection{Interior Expansion}

We define the interior expansion as follows:
\begin{eqnarray}\label{interior expansion}
\u(\vx,\vw)\sim\u_0(\vx,\vw)+\e\u_1(\vx,\vw)+\e^2\u_2(\vx,\vw),
\end{eqnarray}
where $\u_k$ can be determined by comparing the order of $\e$ by
plugging (\ref{interior expansion}) into the equation
(\ref{transport.}). Thus we have
\begin{eqnarray}
\u_0-\bu_0&=&0,\label{expansion temp 1}\\
\u_1-\bu_1&=&-\vw\cdot\nx\u_0,\label{expansion temp 2}\\
\u_2-\bu_2&=&-\vw\cdot\nx\u_1.\label{expansion temp 3}
\end{eqnarray}
Plugging (\ref{expansion temp 1}) into (\ref{expansion temp 2}),
we obtain
\begin{eqnarray}
\u_1=\bu_1-\vw\cdot\nx\bu_0.\label{expansion temp 4}
\end{eqnarray}
Plugging (\ref{expansion temp 4}) into (\ref{expansion temp 3}),
we get
\begin{eqnarray}\label{expansion temp 5}
\u_2-\bu_2&=&-\vw\cdot\nx(\bu_1-\vw\cdot\nx\bu_0)\\
&=&-\vw\cdot\nx\bu_1+w_1^2\p_{x_1x_1}\bu_0+w_2^2\p_{x_2x_2}\bu_0+2w_1w_2\p_{x_1x_2}\bu_0.\no
\end{eqnarray}
Integrating (\ref{expansion temp 5}) over $\vw\in\s^1$, we achieve
the final form
\begin{eqnarray}
\Delta_x\bu_0=0.
\end{eqnarray}
which further implies $\u_0(\vx,\vw)$ satisfies the equation
\begin{eqnarray}\label{interior 1}
\left\{ \begin{array}{rcl} \u_0&=&\bu_0,\\
\Delta_x\bu_0&=&0.
\end{array}
\right.
\end{eqnarray}
In a similar fashion, for $k=1,2$, $\u_k$ satisfies
\begin{eqnarray}\label{interior 2}
\left\{ \begin{array}{rcl} \u_k&=&\bu_k-\vw\cdot\nx\u_{k-1},\\
\Delta_x\bu_k&=&\displaystyle-\int_{\s^1}\vw\cdot\nx\u_{k-1}\ud{\vw}.\end{array}
\right.
\end{eqnarray}
It is easy to see that $\bu_k$ satisfies an elliptic equation. However, the boundary condition of $\bu_k$ is unknown at this stage, since generally $\u_k$ does not necessarily satisfy the diffusive boundary condition of (\ref{transport.}). Therefore, we have to resort to boundary layers.

\subsection{Boundary Layer Expansion}

Besides the Cartesian coordinate
system for interior solutions, we need a local coordinate system in a neighborhood of the boundary to describe boundary layers.

Assume the Cartesian coordinate system is $\vx=(x_1,x_2)$. Using polar coordinates system $(r,\theta)\in[0,\infty)\times[-\pi,\pi)$ and choosing pole in $\Omega$, we assume $\vx_0\in\p\Omega$ is
\begin{eqnarray}
\left\{
\begin{array}{rcl}
x_{1,0}&=&r(\theta)\cos\theta,\\
x_{2,0}&=&r(\theta)\sin\theta,
\end{array}
\right.
\end{eqnarray}
where $r(\theta)>0$ is a given function. Our local coordinate system is similar to the polar coordinate
system, but varies to satisfy the specific requirements.

In a neighborhood of the boundary, for each $\theta$, we have the
outward unit normal vector
\begin{eqnarray}
\vn=\left(\frac{r(\theta)\cos\theta+r'(\theta)\sin\theta}{\sqrt{r(\theta)^2+r'(\theta)^2}},\frac{r(\theta)\sin\theta-r'(\theta)\cos\theta}{\sqrt{r(\theta)^2+r'(\theta)^2}}\right).
\end{eqnarray}
We can determine each
point $\vx\in\bar\Omega$ as $\vx=\vx_0-\mu\vn$ where $\mu$ is the normal distance to a boundary point $\vx_0$. In detail, this means
\begin{eqnarray}\label{local}
\left\{
\begin{array}{rcl}
x_1&=&r(\theta)\cos\theta-\mu\dfrac{r(\theta)\cos\theta+r'(\theta)\sin\theta}{\sqrt{r(\theta)^2+r'(\theta)^2}},\\\rule{0ex}{2.0em}
x_2&=&r(\theta)\sin\theta-\mu\dfrac{r(\theta)\sin\theta-r'(\theta)\cos\theta}{\sqrt{r(\theta)^2+r'(\theta)^2}},
\end{array}
\right.
\end{eqnarray}
where $r'(\theta)=\dfrac{\ud{r}}{\ud{\theta}}$. It is easy to see that $\mu=0$ denotes the boundary $\p\Omega$ and $\mu>0$ denotes the interior of $\Omega$. $(\mu,\theta)$ is the desired local coordinate system.

By chain rule (see \cite{AA007}), we may deduce that
\begin{eqnarray}
\frac{\p\theta}{\p x_1}=\frac{MP}{P^3+Q\mu},&\quad&
\frac{\p\mu}{\p x_1}=-\frac{N}{P},\\
\frac{\p\theta}{\p x_2}=\frac{NP}{P^3+Q\mu},&\quad&
\frac{\p\mu}{\p x_2}=\frac{M}{P},
\end{eqnarray}
where
\begin{eqnarray}
P&=&(r^2+r'^2)^{\frac{1}{2}},\\
Q&=&rr''-r^2-2r'^2,\\
M&=&-r\sin\theta+r'\cos\theta,\\
N&=&r\cos\theta+r'\sin\theta.
\end{eqnarray}
Therefore, note the fact that for $C^2$ convex domains, the curvature
\begin{eqnarray}
\kappa(\theta)=\frac{r^2+2r'^2-rr''}{(r^2+r'^2)^{\frac{3}{2}}},
\end{eqnarray}
and the radius of curvature
\begin{eqnarray}
R_{\kappa}(\theta)=\frac{1}{\kappa(\theta)}=\frac{(r^2+r'^2)^{\frac{3}{2}}}{r^2+2r'^2-rr''}.
\end{eqnarray}
We define substitutions as follows:\\
\ \\
Substitution 1: \\
Let $(x_1,x_2)\rt (\mu,\theta)$ with
$(\mu,\theta)\in [0,R_{\min})\times[-\pi,\pi)$ for $R_{\min}=\min_{\theta}R_{\kappa}$ as
\begin{eqnarray}\label{substitution 1}
\left\{
\begin{array}{rcl}
x_1&=&r(\theta)\cos\theta-\mu\dfrac{r(\theta)\cos\theta+r'(\theta)\sin\theta}{\sqrt{r(\theta)^2+r'(\theta)^2}},\\\rule{0ex}{2.0em}
x_2&=&r(\theta)\sin\theta-\mu\dfrac{r(\theta)\sin\theta-r'(\theta)\cos\theta}{\sqrt{r(\theta)^2+r'(\theta)^2}},
\end{array}
\right.
\end{eqnarray}
and then the equation (\ref{transport.}) is transformed into
\begin{eqnarray}\label{transport 1}
\left\{
\begin{array}{l}
\displaystyle\e\Bigg(w_1\frac{-r\cos\theta-r'\sin\theta}{(r^2+r'^2)^{\frac{1}{2}}}+w_2\frac{-r\sin\theta+r'\cos\theta}{(r^2+r'^2)^{\frac{1}{2}}}\Bigg)\frac{\p u^{\e}}{\p\mu}\\
\displaystyle+\e\Bigg(w_1\frac{-r\sin\theta+r'\cos\theta}{(r^2+r'^2)(1-\kappa\mu)}+w_2\frac{r\cos\theta+r'\sin\theta}{(r^2+r'^2)(1-\kappa\mu)}\Bigg)\frac{\p u^{\e}}{\p\theta}+u^{\e}-\bar u^{\e}=0,\\\rule{0ex}{2.0em}
u^{\e}(0,\theta,\vw)=g(\theta,\vw)\ \ \text{for}\
\ \vw\cdot\vn<0,
\end{array}
\right.
\end{eqnarray}
where
\begin{eqnarray}
\vw\cdot\vn=w_1\frac{r\cos\theta+r'\sin\theta}{(r^2+r'^2)^{\frac{1}{2}}}+w_2\frac{r\sin\theta-r'\cos\theta}{(r^2+r'^2)^{\frac{1}{2}}}.
\end{eqnarray}

Noting the fact that
\begin{eqnarray}
\left(\frac{M}{P}\right)^2+\left(\frac{N}{P}\right)^2=
\left(\frac{-r\cos\theta-r'\sin\theta}{(r^2+r'^2)^{\frac{1}{2}}}\right)^2+\left(\frac{-r\sin\theta+r'\cos\theta}{(r^2+r'^2)^{\frac{1}{2}}}\right)^2=1,
\end{eqnarray}
we can further simplify (\ref{transport 1}).\\
\ \\
Substitution 2: \\
Let $\theta\rt \tau$ with
$\tau\in [-\pi,\pi)$ as
\begin{eqnarray}\label{substitution 2}
\left\{
\begin{array}{rcl}
\sin\tau&=&\dfrac{r\sin\theta-r'\cos\theta}{(r^2+r'^2)^{\frac{1}{2}}},\\\rule{0ex}{2.0em}
\cos\tau&=&\dfrac{r\cos\theta+r'\sin\theta}{(r^2+r'^2)^{\frac{1}{2}}},
\end{array}
\right.
\end{eqnarray}
which implies
\begin{eqnarray}
\frac{\ud{\tau}}{\ud{\theta}}=\kappa(r^2+r'^2)^{\frac{1}{2}}>0.
\end{eqnarray}
Then the equation (\ref{transport.}) is transformed into
\begin{eqnarray}\label{transport 2}
\left\{
\begin{array}{l}\displaystyle
-\e\left(w_1\cos\tau+w_2\sin\tau\right)\frac{\p
u^{\e}}{\p\mu}-\frac{\e}{\rk-\mu}\left(w_1\sin\tau-w_2\cos\tau\right)\frac{\p
u^{\e}}{\p\tau}+u^{\e}-\bar u^{\e}=0,\\\rule{0ex}{2.0em}
u^{\e}(0,\tau,\vw)=g(\tau,\vw)\ \
\text{for}\ \ \vw\cdot\vn<0,
\end{array}
\right.
\end{eqnarray}
where
\begin{eqnarray}
\vw\cdot\vn=w_1\cos\tau+w_2\sin\tau.
\end{eqnarray}
\ \\
Substitution 3:\\
We further make the scaling transform for $\mu\rt
\eta$ with $\eta\in
\left[0,\dfrac{R_{\min}}{\e}\right)$ as
\begin{eqnarray}\label{substitution 3}
\eta&=&\frac{\mu}{\e},
\end{eqnarray}
which implies
\begin{eqnarray}
\frac{\p u^{\e}}{\p\mu}=\frac{1}{\e}\frac{\p u^{\e}}{\p\eta}.
\end{eqnarray}
Then the equation (\ref{transport.}) is transformed into
\begin{eqnarray}\label{transport 3}
\left\{\begin{array}{l}\displaystyle
-\bigg(w_1\cos\tau+w_2\sin\tau\bigg)\frac{\p
u^{\e}}{\p\eta}-\frac{\e}{\rk-\e\eta}\bigg(w_1\sin\tau-w_2\cos\tau\bigg)\frac{\p
u^{\e}}{\p\tau}+u^{\e}-\bar u^{\e}=0,\\\rule{0ex}{2.0em}
u^{\e}(0,\tau,\vw)=g(\tau,\vw)\ \
\text{for}\ \ \vw\cdot\vn<0,
\end{array}
\right.
\end{eqnarray}
where
\begin{eqnarray}
\vw\cdot\vn=w_1\cos\tau+w_2\sin\tau.
\end{eqnarray}
\ \\
Substitution 4:\\
Define the velocity substitution for $(w_1,w_2)\rt
\xi$ with $\xi\in
[-\pi,\pi)$ as
\begin{eqnarray}\label{substitution 4}
\left\{
\begin{array}{rcl}
w_1&=&-\sin\xi\\
w_2&=&-\cos\xi
\end{array}
\right.
\end{eqnarray}
We have the succinct form of the equation (\ref{transport.}) as
\begin{eqnarray}\label{transport 4}
\left\{\begin{array}{l}\displaystyle \sin(\tau+\xi)\frac{\p
u^{\e}}{\p\eta}-\frac{\e}{\rk-\e\xi}\cos(\tau+\xi)\frac{\p
u^{\e}}{\p\tau}+u^{\e}-\bar u^{\e}=0,\\\rule{0ex}{2.0em}
u^{\e}(0,\tau,\xi)=g(\tau,\xi)\ \ \text{for}\ \
\sin(\tau+\xi)>0.
\end{array}
\right.
\end{eqnarray}
\ \\
Substitution 5:\\
As \cite{AA003} and \cite{AA007} reveal, we need a further rotational substitution for $\xi\rt
\phi$ with $\phi\in
[-\pi,\pi)$ as
\begin{eqnarray}\label{substitution 5}
\phi&=&\tau+\xi
\end{eqnarray}
and achieve the form
\begin{eqnarray}\label{transport temp}
\left\{\begin{array}{l}\displaystyle \sin\phi\frac{\p
u^{\e}}{\p\eta}-\frac{\e}{\rk-\e\eta}\cos\phi\bigg(\frac{\p
u^{\e}}{\p\phi}+\frac{\p
u^{\e}}{\p\tau}\bigg)+u^{\e}-\bar u^{\e}=0\\\rule{0ex}{2.0em}
u^{\e}(0,\tau,\phi)=g(\tau,\phi)\ \ \text{for}\ \
\sin\phi>0.
\end{array}
\right.
\end{eqnarray}
This step is trying to compensate the variants of the normal vector $\nu$ along the boundary. A bi-product of such substitution is that we decompose the tangential derivative and introduce a new velocity derivative.

We define the boundary layer expansion as follows:
\begin{eqnarray}\label{boundary layer expansion}
\uu(\eta,\tau,\phi)\sim\uu_0(\eta,\tau,\phi)+\e\uu_1(\eta,\tau,\phi),
\end{eqnarray}
where $\ub_k$ can be determined by comparing the order of $\e$ via
plugging (\ref{boundary layer expansion}) into the equation
(\ref{transport temp}). Thus, in a neighborhood of the boundary, we have
\begin{eqnarray}
\sin\phi\frac{\p \uu_0}{\p\eta}-\frac{\e}{\rk-\e\eta}\cos\phi\frac{\p
\uu_0}{\p\phi}+\uu_0-\buu_0&=&0,\label{expansion temp 6}\\
\sin\phi\frac{\p \uu_1}{\p\eta}-\frac{\e}{\rk-\e\eta}\cos\phi\frac{\p
\uu_1}{\p\phi}+\uu_1-\buu_1&=&\frac{1}{\rk-\e\eta}\cos\phi\frac{\p
\uu_0}{\p\tau},\label{expansion temp 7}
\end{eqnarray}
where
\begin{eqnarray}
\buu_k(\eta,\tau)=\frac{1}{2\pi}\int_{-\pi}^{\pi}\uu_k(\eta,\tau,\phi)\ud{\phi}.
\end{eqnarray}
We call this type of equations the $\e$-Milne problem with geometric correction.

\subsection{Decomposition and Modification}

%

In this section, we prove the important decomposition of boundary data, which can greatly improve the regularity.

Consider the $\e$-Milne problem with geometric correction with $L=\e^{-\frac{1}{2}}$ and $\rr[\phi]=-\phi$,
\begin{eqnarray}\label{deco}
\left\{
\begin{array}{l}
\sin\phi\dfrac{\p f}{\p\eta}-\dfrac{\e}{\rk-\e\eta}\cos\phi\dfrac{\p
f}{\p\phi}+f-\bar f=0,\\\rule{0ex}{1.5em}
f(0,\phi)=g(\phi)\ \ \text{for}\ \
\sin\phi>0,\\\rule{0ex}{1.5em}
f(L,\phi)=f(L,\rr[\phi]).
\end{array}
\right.
\end{eqnarray}
We assume that $g(\phi)$ is not a constant and $0\leq g(\phi)\leq 1$. This is always achievable and we do not lose the generality since the equation is linear. For some $\alpha>0$ which will be determined later, define two $C^{\infty}$ auxiliary functions
\begin{eqnarray}
g_1(\phi)=\left\{
\begin{array}{ll}
0&\ \ \text{for}\ \ \phi\in(0,\e^{\alpha}]\cup[\pi-\e^{\alpha},\pi),\\
g(\phi)&\ \ \text{for}\ \ \phi\in[2\e^{\alpha},\pi-2\e^{\alpha}],\\
\end{array}
\right.
\end{eqnarray}
and
\begin{eqnarray}
g_2(\phi)=\left\{
\begin{array}{ll}
1&\ \ \text{for}\ \ \phi\in(0,\e^{\alpha}]\cup[\pi-\e^{\alpha},\pi),\\
g(\phi)&\ \ \text{for}\ \ \phi\in[2\e^{\alpha},\pi-2\e^{\alpha}].\\
\end{array}
\right.
\end{eqnarray}
Let $f_1(\eta,\phi)$ and $f_2(\eta,\phi)$ be the solutions to the equation (\ref{deco}) with in-flow data $g_1(\phi)$ and $g_2(\phi)$ respectively. Then by Theorem \ref{Milne theorem 1}, we know $f_1$ and $f_2$ are well-defined in $L^{\infty}$. By Theorem \ref{Milne theorem 3}, they satisfy the maximum principle, which means
\begin{eqnarray}
&&f_1(0,0^+)-\bar f_1(0)=f_1(0,\pi^-)-\bar f_1(0)=-\bar f_1(0)<0,\\
&&f_2(0,0^+)-\bar f_2(0)=f_2(0,\pi^-)-\bar f_2(0)=1-\bar f_2(0)>0.
\end{eqnarray}
Therefore, there exists a constant $0<\l<1$ such that
\begin{eqnarray}
\l\Big(f_1(0,0^+)-\bar f_1(0)\Big)+(1-\l)\Big(f_2(0,0^+)-\bar f_2(0)\Big)&=&0,\\
\l\Big(f_1(0,\pi^-)-\bar f_1(0)\Big)+(1-\l)\Big(f_2(0,\pi^-)-\bar f_2(0)\Big)&=&0.
\end{eqnarray}
Let $g_{\l}(\phi)=\l g_1(\phi)+(1-\l)g_2(\phi)$ and the corresponding solution to the equation (\ref{deco}) is $f_{\l}(\eta,\phi)$. We have
\begin{eqnarray}
f_{\l}(0,0^+)-\bar f_{\l}(0)=f_{\l}(0,\pi^-)-\bar f_{\l}(0)=0.
\end{eqnarray}
Since for $\phi\in(0,\e^{\alpha}]\cup[\pi-\e^{\alpha},\pi)$, $g_{\l}=1-\l$ is a constant, we naturally have $\dfrac{\p g_{\l}}{\p\phi}=0$. We may solve from the equation (\ref{deco}) that
\begin{eqnarray}
\\
\dfrac{\p f_{\l}}{\p\eta}\bigg|_{\eta=0,\phi\in(0,\e^{\alpha}]\cup[\pi-\e^{\alpha},\pi)}&=&\frac{1}{\sin\phi}\bigg(\dfrac{\e}{\rk-\e\eta}\cos\phi\dfrac{\p
g_{\l}}{\p\phi}\bigg|_{\phi\in(0,\e^{\alpha}]\cup[\pi-\e^{\alpha},\pi)}-\Big(f_{\l}-\bar f_{\l}\Big)\bigg|_{\eta=0,\phi\in(0,\e^{\alpha}]\cup[\pi-\e^{\alpha},\pi)}\bigg)=0.\no
\end{eqnarray}
Note that $g_{\l}(\phi)=g(\phi)$ for $\phi\in[2\e^{\alpha},\pi-2\e^{\alpha}]$, so our modification is restricted to a small region near the grazing set and we can smoothen the normal derivative at the boundary.

This method can be easily generalized to treat other $g(\phi)$. In principle, for $g(\phi)\in C^1$, we can define a decomposition
\begin{eqnarray}
g(\phi)=\gb(\phi)+\gf(\phi),
\end{eqnarray}
such that $\gf(\phi)=0$ for $\sin\phi\geq2\e^{\alpha}$, and the solution to the equation (\ref{deco}) with in-flow data $\gb(\phi)$ has $L^{\infty}$ normal derivative at $\eta$=0.
Such a decomposition comes with a price. Originally, we have $\lnm{\dfrac{\p g}{\p\phi}}\leq C$. However, now we only have $\lnm{\dfrac{\p\gb}{\p\phi}}\leq C\e^{-\alpha}$ and $\lnm{\dfrac{\p\gf}{\p\phi}}\leq C\e^{-\alpha}$ due to the short-ranged cut-off function.

\subsection{Matching Procedure}

The bridge between the interior solution and boundary layer
is the boundary condition of (\ref{transport.}), so we
first consider the boundary expansion:
\begin{eqnarray}
\u_0+\ub_0+\uf_0&=&g,\\
\u_1+\ub_1&=&0.
\end{eqnarray}
Here $\ub_0$ and $\uf_0$ are boundary layers with corresponding decomposed boundary data $\gb$ and $\gf$. We call $\ub$ the regular boundary layer and $\uf$ the singular boundary layer. They should both satisfy the $\e$-Milne problem with geometric correction.\\
\ \\
Step 0: Preliminaries.\\
Define the weight function
\begin{eqnarray}\label{weight function}
\zeta(\eta,\phi)=\Bigg(1-\bigg(\frac{\rk-\e\eta}{\rk}\cos\phi\bigg)^2\Bigg)^{\frac{1}{2}}.
\end{eqnarray}
Define the force as
\begin{eqnarray}\label{force}
F(\e;\eta,\tau)=-\frac{\e}{\rk(\tau)-\e\eta},
\end{eqnarray}
and the length for $\e$-Milne problem as $L=\e^{-\frac{1}{2}}$. For $\phi\in[-\pi,\pi]$, denote $\rr[\phi]=-\phi$.\\
\ \\
Step 1: Construction of $\ub_0$, $\uf_0$ and $\u_0$.\\
Define the zeroth-order boundary layer as
\begin{eqnarray}\label{expansion temp 9}
\left\{
\begin{array}{l}
\ub_0(\eta,\tau,\phi)=\mathscr{F}_0 (\eta,\tau,\phi)-\mathscr{F}_{0,L}(\tau),\\\rule{0ex}{1.5em}
\sin\phi\dfrac{\p \mathscr{F}_0 }{\p\eta}+F(\e;\eta,\tau)\cos\phi\dfrac{\p
\mathscr{F}_0 }{\p\phi}+\mathscr{F}_0 -\bar{\mathscr{F}}_0 =0,\\\rule{0ex}{1.5em}
\mathscr{F}_0 (0,\tau,\phi)=\gb(\tau,\phi)\ \ \text{for}\ \
\sin\phi>0,\\\rule{0ex}{1.5em}
\mathscr{F}_0 (L,\tau,\phi)=\mathscr{F}_0 (L,\tau,\rr[\phi]),
\end{array}
\right.
\end{eqnarray}
with $\mathscr{F}_{0,L}(\tau)$ is defined in Theorem \ref{Milne theorem 1}, and
\begin{eqnarray}\label{expansion temp 9.}
\left\{
\begin{array}{l}
\uf_0(\eta,\tau,\phi)=\mathfrak{F}_0 (\eta,\tau,\phi)-\mathfrak{F} _{0,L}(\tau),\\\rule{0ex}{1.5em}
\sin\phi\dfrac{\p \mathfrak{F}_0 }{\p\eta}+F(\e;\eta,\tau)\cos\phi\dfrac{\p
\mathfrak{F}_0 }{\p\phi}+\mathfrak{F}_0 -\bar{\mathfrak{F}}_0 =0,\\\rule{0ex}{1.5em}
\mathfrak{F}_0 (0,\tau,\phi)=\gf(\tau,\phi)\ \ \text{for}\ \
\sin\phi>0,\\\rule{0ex}{1.5em}
\mathfrak{F}_0 (L,\tau,\phi)=\mathfrak{F}_0 (L,\tau,\rr[\phi]),
\end{array}
\right.
\end{eqnarray}
with $\mathfrak{F} _{0,L}(\tau)$ is defined in Theorem \ref{Milne theorem 1}.
Also, we define the zeroth-order interior solution $\u_0(\vx,\vw)$ as
\begin{eqnarray}\label{expansion temp 11}
\left\{
\begin{array}{l}
\u_0(\vx,\vw)=\bu_0(\vx) ,\\\rule{0ex}{1.5em} \Delta_x\bu_0(\vx)=0\ \ \text{in}\
\ \Omega,\\\rule{0ex}{1.5em}
\bu_0(\vx_0)=\mathscr{F}_{0,L}(\tau)+\mathfrak{F}_{0,L}(\tau)\ \ \text{on}\ \
\p\Omega.
\end{array}
\right.
\end{eqnarray}
\ \\
Step 2: Construction of $\ub_1$ and $\u_1$.\\
Define the first-order boundary layer as
\begin{eqnarray}\label{expansion temp 10}
\left\{
\begin{array}{l}
\ub_1(\eta,\tau,\phi)=\mathscr{F}_1 (\eta,\tau,\phi)-\mathscr{F} _{1,L}(\tau),\\\rule{0ex}{1.5em}
\sin\phi\dfrac{\p \mathscr{F}_1 }{\p\eta}+F(\e;\eta,\tau)\cos\phi\dfrac{\p
\mathscr{F}_1 }{\p\phi}+\mathscr{F}_1 -\bar{\mathscr{F}}_1 =\dfrac{1}{\rk-\e\eta}\cos\phi\dfrac{\p
\ub_0}{\p\tau},\\\rule{0ex}{1.5em}
\mathscr{F}_1 (0,\tau,\phi)=\vw\cdot\nx\u_0(0,\tau,\vw)\ \ \text{for}\ \
\sin\phi>0,\\\rule{0ex}{1.5em}
\mathscr{F}_1 (L,\tau,\phi)=\mathscr{F}_1 (L,\tau,\rr[\phi]),
\end{array}
\right.
\end{eqnarray}
with $\mathscr{F}_{1,L}(\tau)$ is defined in Theorem \ref{Milne theorem 1}.
Then we define the first-order interior solution $\u_1(\vx,\vw)$ as
\begin{eqnarray}\label{expansion temp 12.}
\left\{
\begin{array}{rcl}
\u_1(\vx,\vw)&=&\bu_1(\vx)-\vw\cdot\nx\u_0(\vx,\vw),\\\rule{0ex}{1.5em}
\Delta_x\bu_1(\vx)&=&-\displaystyle\int_{\s^1}\Big(\vw\cdot\nx\u_{0}(\vx,\vw)\Big)\ud{\vw}\
\ \text{in}\ \ \Omega,\\\rule{0ex}{1.5em} \bu_1(\vx_0)&=&f _{1,L}(\tau)\ \ \text{on}\ \
\p\Omega.
\end{array}
\right.
\end{eqnarray}
Note that we do not define $\uf_1$ here.\\
\ \\
Step 3: Construction of $\u_2$.\\
Since we do not expand to $\ub_2$ and $\uf_2$, we define the second-order interior solution as
\begin{eqnarray}
\left\{
\begin{array}{rcl}
\u_{2}(\vx,\vw)&=&\bu_{2}(\vx)-\vw\cdot\nx\u_{1}(\vx,\vw),\\\rule{0ex}{1.5em}
\Delta_x\bu_{2}(\vx)&=&-\displaystyle\int_{\s^1}\Big(\vw\cdot\nx\u_{1}(\vx,\vw)\Big)\ud{\vw}\
\ \text{in}\ \ \Omega,\\\rule{0ex}{1.5em} \bu_2(\vx_0)&=&0\ \ \text{on}\ \
\p\Omega.
\end{array}
\right.
\end{eqnarray}
Here, we might have $O(\e^3)$ error in this step due to the trivial boundary data. Thanks to the remainder estimate, it will not affect the diffusive limit.

\newpage

\section{Remainder Estimate}

In this section, we consider the remainder equation for $u(\vx,\vw)$ as
\begin{eqnarray}\label{neutron}
\left\{
\begin{array}{l}\displaystyle
\e\vw\cdot\nx u+u-\bar
u=f(\vx,\vw)\ \ \text{in}\ \ \Omega,\\\rule{0ex}{1.0em}
u(\vx_0,\vw)=h(\vx_0,\vw)\ \ \text{for}\ \
\vw\cdot\vn<0\ \ \text{and}\ \ \vx_0\in\p\Omega,
\end{array}
\right.
\end{eqnarray}
where
\begin{eqnarray}
\bar u(\vx)=\frac{1}{2\pi}\int_{\s^1}u(\vx,\vw)\ud{\vw},
\end{eqnarray}
$\vn$ is the outward unit normal vector, with the Knudsen number $0<\e<<1$.

We define the $L^p$ norm with $1\leq p<\infty$ and $L^{\infty}$ norms in $\Omega\times\s^1$ as
usual:
\begin{eqnarray}
\nm{f}_{L^p(\Omega\times\s^1)}&=&\bigg(\int_{\Omega}\int_{\s^1}\abs{f(\vx,\vw)}^p\ud{\vw}\ud{\vx}\bigg)^{\frac{1}{p}},\\
\nm{f}_{L^{\infty}(\Omega\times\s^1)}&=&\sup_{(\vx,\vw)\in\Omega\times\s^1}\abs{f(\vx,\vw)}.
\end{eqnarray}
Define the $L^p$ norm with $1\leq p<\infty$ and $L^{\infty}$ norms on the boundary as follows:
\begin{eqnarray}
\nm{f}_{L^p(\Gamma)}&=&\bigg(\iint_{\Gamma}\abs{f(\vx,\vw)}^p\abs{\vw\cdot\vn}\ud{\vw}\ud{\vx}\bigg)^{\frac{1}{p}},\\
\nm{f}_{L^p(\Gamma^{\pm})}&=&\bigg(\iint_{\Gamma^{\pm}}\abs{f(\vx,\vw)}^p\abs{\vw\cdot\vn}\ud{\vw}\ud{\vx}\bigg)^{\frac{1}{p}},\\
\nm{f}_{L^{\infty}(\Gamma)}&=&\sup_{(\vx,\vw)\in\Gamma}\abs{f(\vx,\vw)},\\
\nm{f}_{L^{\infty}(\Gamma^{\pm})}&=&\sup_{(\vx,\vw)\in\Gamma^{\pm}}\abs{f(\vx,\vw)}.
\end{eqnarray}
In particular, we denote $\ud{\gamma}=(\vw\cdot\vn)\ud{\vw}\ud{\vx_0}$ on the boundary.

The remainder estimates for neutron transport equation with diffusive boundary was proved in \cite{AA007} and \cite{AA009}. Here, the case with in-flow boundary was first shown in \cite{AA003}, so here we will focus on the a priori estimates and prove an improved version.

\subsection{$L^2$ Estimate}

\begin{lemma}(Green's Identity)\label{well-posedness lemma 1}
Assume $u(\vx,\vw),\ v(\vx,\vw)\in L^2(\Omega\times\s^1)$ and
$\vw\cdot\nx u,\ \vw\cdot\nx v\in L^2(\Omega\times\s^1)$ with $u,\
v\in L^2(\Gamma)$. Then
\begin{eqnarray}
\iint_{\Omega\times\s^1}\bigg((\vw\cdot\nx u)v+(\vw\cdot\nx
u)v\bigg)\ud{\vx}\ud{\vw}=\int_{\Gamma}uv\ud{\gamma}.
\end{eqnarray}
\end{lemma}
\begin{proof}
See \cite[Chapter 9]{Cercignani.Illner.Pulvirenti1994} and
\cite{Esposito.Guo.Kim.Marra2013}.
\end{proof}
\begin{lemma}\label{LT estimate}
The unique solution $u(\vx,\vw)$ to the equation (\ref{neutron}) satisfies
\begin{eqnarray}
\frac{1}{\e^{\frac{1}{2}}}\nm{u}_{L^2(\Gamma^+)}+\nm{u}_{L^2(\Omega\times\s^1)}\leq
C \bigg(
\frac{1}{\e^2}\nm{f}_{L^2(\Omega\times\s^1)}+\frac{1}{\e^{\frac{1}{2}}}\nm{h}_{L^2(\Gamma^-)}\bigg).
\end{eqnarray}
\end{lemma}
\begin{proof}
We divide the proof into several steps:\\
\ \\
Step 1: Kernel Estimate.\\
Applying Lemma \ref{well-posedness lemma 1} to the
equation (\ref{neutron}). Then for any
$\phi\in L^2(\Omega\times\s^1)$ satisfying $\vw\cdot\nx\phi\in
L^2(\Omega\times\s^1)$ and $\phi\in L^2(\Gamma)$, we have
\begin{eqnarray}\label{well-posedness temp 4}
\e\int_{\Gamma}u\phi\ud{\gamma}
-\e\iint_{\Omega\times\s^1}(\vw\cdot\nx\phi)u+\iint_{\Omega\times\s^1}(u-\bar
u)\phi=\iint_{\Omega\times\s^1}f\phi.
\end{eqnarray}
Our goal is to choose a particular test function $\phi$. We first
construct an auxiliary function $\xi$. Since $u\in
L^{2}(\Omega\times\s^1)$, it naturally implies $\bar u\in
L^2(\Omega)$. We define $\xi(\vx)$ on $\Omega$ satisfying
\begin{eqnarray}\label{test temp 1}
\left\{
\begin{array}{rcl}
\Delta \xi&=&\bar u\ \ \text{in}\ \
\Omega,\\\rule{0ex}{1.0em} \xi&=&0\ \ \text{on}\ \ \p\Omega.
\end{array}
\right.
\end{eqnarray}
Hence, in the bounded domain $\Omega$, based on the standard elliptic
estimate, there exists a unique $\xi\in H^2(\Omega)$ such that
\begin{eqnarray}\label{test temp 3}
\nm{\xi}_{H^2(\Omega)}\leq C \nm{\bar
u}_{L^2(\Omega)}.
\end{eqnarray}
We plug the test function
\begin{eqnarray}\label{test temp 2}
\phi=-\vw\cdot\nx\xi
\end{eqnarray}
into the weak formulation (\ref{well-posedness temp 4}) and estimate
each term there. Naturally, we have
\begin{eqnarray}\label{test temp 4}
\nm{\phi}_{L^2(\Omega)}\leq C\nm{\xi}_{H^1(\Omega)}\leq
C \nm{\bar u}_{L^2(\Omega)}.
\end{eqnarray}
Easily we can decompose
\begin{eqnarray}\label{test temp 5}
-\e\iint_{\Omega\times\s^1}(\vw\cdot\nx\phi)u&=&-\e\iint_{\Omega\times\s^1}(\vw\cdot\nx\phi)\bar
u-\e\iint_{\Omega\times\s^1}(\vw\cdot\nx\phi)(u-\bar
u).
\end{eqnarray}
We estimate the two term on the right-hand side separately. By
(\ref{test temp 1}) and (\ref{test temp 2}), we have
\begin{eqnarray}\label{wellposed temp 1}
-\e\iint_{\Omega\times\s^1}(\vw\cdot\nx\phi)\bar
u&=&\e\iint_{\Omega\times\s^1}\bar
u\bigg(w_1(w_1\p_{11}\xi+w_2\p_{12}\xi)+w_2(w_1\p_{12}\xi+w_2\p_{22}\xi)\bigg)\\
&=&\e\iint_{\Omega\times\s^1}\bar
u\bigg(w_1^2\p_{11}\xi+w_2^2\p_{22}\xi\bigg)\nonumber\\
&=&2\e\pi\int_{\Omega}\bar u(\p_{11}\xi+\p_{22}\xi)\nonumber\\
&=&2\e\pi\nm{\bar u}_{L^2(\Omega)}^2\nonumber\\
&=&\e\nm{\bar u}_{L^2(\Omega\times\s^1)}^2\nonumber.
\end{eqnarray}
In the second equality, above cross terms vanish due to the symmetry
of the integral over $\s^1$. On the other hand, for the second term
in (\ref{test temp 5}), H\"older's inequality and the elliptic
estimate imply
\begin{eqnarray}\label{wellposed temp 2}
-\e\iint_{\Omega\times\s^1}(\vw\cdot\nx\phi)(u-\bar
u)&\leq&C \e\nm{u-\bar u}_{L^2(\Omega\times\s^1)}\nm{\xi}_{H^2(\Omega)}\\
&\leq&C \e\nm{u-\bar
u}_{L^2(\Omega\times\s^1)}\nm{\bar
u}_{L^2(\Omega\times\s^1)}\nonumber.
\end{eqnarray}
Using the trace theorem, we have
\begin{eqnarray}\label{wellposed temp 3}
\e\int_{\Gamma}u\phi\ud{\gamma}&=&\e\int_{\Gamma^+}u\phi\ud{\gamma}+\e\int_{\Gamma^-}u\phi\ud{\gamma}
\leq C\e\nm{\phi}_{L^2(\Gamma)}\bigg(\nm{u}_{L^2(\Gamma^+)}+\nm{h}_{L^2(\Gamma^-)}\bigg)\\
&\leq& C\e\nm{\phi}_{H^1(\Omega\times\s^1)}\bigg(\nm{u}_{L^2(\Gamma^+)}+\nm{h}_{L^2(\Gamma^-)}\bigg)\leq C\e\nm{\bar u}_{L^2(\Omega\times\s^1)}\bigg(\nm{u}_{L^2(\Gamma^+)}+\nm{h}_{L^2(\Gamma^-)}\bigg).\no
\end{eqnarray}
Also, we obtain
\begin{eqnarray}\label{wellposed temp 5}
\iint_{\Omega\times\s^1}(u-\bar u)\phi\leq
C \nm{\bar u}_{L^2(\Omega\times\s^1)}\nm{u-\bar
u}_{L^2(\Omega\times\s^1)},
\end{eqnarray}
\begin{eqnarray}\label{wellposed temp 6}
\iint_{\Omega\times\s^1}f\phi\leq C \nm{\bar
u}_{L^2(\Omega\times\s^1)}\nm{f}_{L^2(\Omega\times\s^1)}.
\end{eqnarray}
Collecting terms in (\ref{wellposed temp 1}), (\ref{wellposed temp
2}), (\ref{wellposed temp 3}),
(\ref{wellposed temp 5}) and (\ref{wellposed temp 6}), we obtain
\begin{eqnarray}
\e\nm{\bar u}_{L^2(\Omega\times\s^1)}^2&\leq&
C \nm{\bar u}_{L^2(\Omega\times\s^1)}\bigg(\nm{u-\bar
u}_{L^2(\Omega\times\s^1)}+\e\nm{u}_{L^2(\Gamma^+)}+\nm{f}_{L^2(\Omega\times\s^1)}+\e\tm{h}{\Gamma^-}\bigg).
\end{eqnarray}
Then this naturally implies that
\begin{eqnarray}\label{well-posedness temp 8}
\e\nm{\bar u}_{L^2(\Omega\times\s^1)}&\leq&
C \bigg(\nm{u-\bar
u}_{L^2(\Omega\times\s^1)}+\e\nm{u}_{L^2(\Gamma^+)}+\nm{f}_{L^2(\Omega\times\s^1)}+\e\tm{h}{\Gamma^-}\bigg).
\end{eqnarray}
\ \\
Step 2: Energy Estimate.\\
In the weak formulation (\ref{well-posedness temp 4}), we may take
the test function $\phi=u$ to get the energy estimate
\begin{eqnarray}
\half\e\int_{\Gamma}\abs{u}^2\ud{\gamma}+\nm{u-\bar
u}_{L^2(\Omega\times\s^1)}^2=\iint_{\Omega\times\s^1}fu.
\end{eqnarray}
Then we have
\begin{eqnarray}\label{well-posedness temp 5}
&&\half\e\nm{u}^2_{L^2(\Gamma^+)}+\nm{u-\bar
u}_{L^2(\Omega\times\s^1)}^2= \iint_{\Omega\times\s^1}fu+\e\nm{h}_{L^2(\Gamma^-)}^2.
\end{eqnarray}
On the other hand, we can square on both sides of
(\ref{well-posedness temp 8}) to obtain
\begin{eqnarray}\label{well-posedness temp 6}
\e^2\nm{\bar u}_{L^2(\Omega\times\s^1)}^2&\leq&
C\bigg(\nm{u-\bar
u}_{L^2(\Omega\times\s^1)}^2+\e^2\nm{u}_{L^2(\Gamma^+)}^2+\nm{f}_{L^2(\Omega\times\s^1)}^2+\e^2\tm{h}{\Gamma^-}^2\bigg).
\end{eqnarray}
Multiplying a sufficiently small constant on both sides of
(\ref{well-posedness temp 6}) and adding it to (\ref{well-posedness
temp 5}) to absorb $\nm{u}_{L^2(\Gamma^+)}^2$ and
$\nm{u-\bar u}_{L^2(\Omega\times\s^1)}^2$, we deduce
\begin{eqnarray}
&&\e\nm{u}_{L^2(\Gamma^+)}^2+\e^2\nm{\bar
u}_{L^2(\Omega\times\s^1)}^2+\nm{u-\bar
u}_{L^2(\Omega\times\s^1)}^2\leq
C \bigg(\tm{f}{\Omega\times\s^1}^2+
\iint_{\Omega\times\s^1}fu+\e\nm{h}_{L^2(\Gamma^-)}^2\bigg).
\end{eqnarray}
Hence, we have
\begin{eqnarray}
\e\nm{u}_{L^2(\Gamma^+)}^2+\e^2\nm{u}_{L^2(\Omega\times\s^1)}^2\leq
C \bigg(\tm{f}{\Omega\times\s^1}^2+
\iint_{\Omega\times\s^1}fu+\e\nm{h}_{L^2(\Gamma^-)}^2\bigg).
\end{eqnarray}
A simple application of Cauchy's inequality leads to
\begin{eqnarray}
\iint_{\Omega\times\s^1}fu\leq\frac{1}{4C\e^2}\tm{f}{\Omega\times\s^1}^2+C\e^2\tm{u}{\Omega\times\s^1}^2.
\end{eqnarray}
Taking $C$ sufficiently small to absorb $C\e^2\tm{u}{\Omega\times\s^1}^2$, we obtain
\begin{eqnarray}\label{well-posedness temp 7}
\e\nm{u}_{L^2(\Gamma^+)}^2+\e^2\nm{u}_{L^2(\Omega\times\s^1)}^2\leq
C \bigg(\frac{1}{\e^2}\tm{f}{\Omega\times\s^1}^2+\e\nm{h}_{L^2(\Gamma^-)}^2\bigg).
\end{eqnarray}
Then we can divide $\e^2$ on both sides of (\ref{well-posedness
temp 7}) to obtain
\begin{eqnarray}\label{well-posedness temp 21}
\frac{1}{\e}\nm{u}_{L^2(\Gamma^+)}^2+\nm{u}_{L^2(\Omega\times\s^1)}^2\leq
C \bigg(
\frac{1}{\e^4}\nm{f}_{L^2(\Omega\times\s^1)}^2+\frac{1}{\e}\nm{h}_{L^2(\Gamma^-)}^2\bigg).
\end{eqnarray}
Hence, we naturally have
\begin{eqnarray}
\frac{1}{\e^{\frac{1}{2}}}\nm{u}_{L^2(\Gamma^+)}+\nm{u}_{L^2(\Omega\times\s^1)}\leq
C \bigg(
\frac{1}{\e^2}\nm{f}_{L^2(\Omega\times\s^1)}+\frac{1}{\e^{\frac{1}{2}}}\nm{h}_{L^2(\Gamma^-)}\bigg).
\end{eqnarray}
\end{proof}

\subsection{$L^{\infty}$ Estimate - First Round}

\begin{theorem}\label{LI estimate.}
The unique solution $u(\vx,\vw)$ to the equation (\ref{neutron}) satisfies
\begin{eqnarray}
\im{u}{\Omega\times\s^1}&\leq& C\bigg(\frac{1}{\e^{3}}\tm{f}{\Omega\times\s^1}+\im{f}{\Omega\times\s^1}+\frac{1}{\e^{\frac{3}{2}}}\tm{h}{\Gamma^-}+\im{h}{\Gamma^-}\bigg).
\end{eqnarray}
\end{theorem}
\begin{proof}
We divide the proof into several steps:\\
\ \\
Step 1: Double Duhamel iterations.\\
We can rewrite the equation
(\ref{neutron}) along the characteristics as
\begin{eqnarray}
u(\vx,\vw)&=&h(\vx-\e t_b\vw,\vw)\ue^{-t_b}+\int_{0}^{t_b}f\Big(\vx-\e(t_b-s)\vw,\vw\Big)\ue^{-(t_b-s)}\ud{s}\\
&&+\frac{1}{2\pi}\int_{0}^{t_b}\bigg(\int_{\s^1}u\Big(\vx-\e(t_b-s)\vw,\vw_t\Big)\ud{\vw_t}\bigg)\ue^{-(t_b-s)}\ud{s}\nonumber.
\end{eqnarray}
where the backward exit time $t_b$ is defined as
\begin{eqnarray}\label{exit time}
t_b(\vx,\vw)=\inf\{t\geq0: (\vx-\e t\vw,\vw)\in\Gamma^-\},
\end{eqnarray}
which represents the first time that the characteristics track back and hit the in-flow boundary. Note we have replaced $\bar u$ by the integral of $u$ over the dummy velocity
variable $\vw_t$. For the last term in this formulation, we apply
the Duhamel's principle again to
$u\Big(\vx-\e(t_b-s)\vw,\vw_t\Big)$ and obtain
\begin{eqnarray}\label{well-posedness temp 8}
u(\vx,\vw)&=&h(\vx-\e t_b\vw,\vw)\ue^{-t_b}+\int_{0}^{t_b}f\Big(\vx-\e(t_b-s)\vw,\vw\Big)\ue^{-(t_b-s)}\ud{s}\\
&&+\frac{1}{2\pi}\int_{0}^{t_b}\bigg(\int_{\s^1}h\Big(\vx-\e(t_b-s)\vw-\e
s_b\vw_t,\vw_t\Big)\ue^{-s_b}\ud{\vw_t}\bigg)\ue^{-(t_b-s)}\ud{s}\nonumber\\
&&+\frac{1}{2\pi}\int_{0}^{t_b}\Bigg(\int_{\s^1}\bigg(\int_{0}^{s_b}f\Big(\vx-\e(t_b-s)\vw-\e
(s_b-r)\vw_t,\vw_t\Big)\ue^{-(s_b-r)}\ud{r}\bigg)\ud{\vw_t}\Bigg)\ue^{-(t_b-s)}\ud{s}\nonumber\\
&&+\frac{1}{2\pi}\int_{0}^{t_b}\Bigg(\int_{\s^1}\bigg(\int_{0}^{s_b}\bar u\Big(\vx-\e(t_b-s)\vw-\e
(s_b-r)\vw_t\Big)\ue^{-(s_b-r)}\ud{r}\bigg)\ud{\vw_t}\Bigg)\ue^{-(t_b-s)}\ud{s}\nonumber,
\end{eqnarray}
where the exiting time from $\Big(\vx-\e(t_b-s)\vw,\vw_t\Big)$ is defined as
\begin{eqnarray}
s_b(\vx,\vw,s,\vw_t)=\inf\{r\geq0: \Big(\vx-\e(t_b-s)\vw-\e
r\vw_t,\vw_t\Big)\in\Gamma^-\}.
\end{eqnarray}
\ \\
Step 2: Estimates of all but the last term in (\ref{well-posedness temp 8}).\\
We can directly estimate as follows:
\begin{eqnarray}\label{im temp 1}
\abs{h(\vx-\e t_b\vw,\vw)\ue^{-t_b}}\leq\im{h}{\Gamma^-},
\end{eqnarray}
\begin{eqnarray}\label{im temp 2}
\abs{\frac{1}{2\pi}\int_{0}^{t_b}\bigg(\int_{\s^1}h\Big(\vx-\e(t_b-s)\vw-\e
s_b\vw_t,\vw_t\Big)\ue^{-s_b}\ud{\vw_t}\bigg)\ue^{-(t_b-s)}\ud{s}} \leq
\im{h}{\Gamma^-},
\end{eqnarray}
\begin{eqnarray}\label{im temp 3}
\abs{\int_{0}^{t_b}f\Big(\vx-\e(t_b-s)\vw,\vw\Big)\ue^{-(t_b-s)}\ud{s}}\leq
\im{f}{\Omega\times\s^1},
\end{eqnarray}
\begin{eqnarray}\label{im temp 4}
\\
\abs{\frac{1}{2\pi}\int_{0}^{t_b}\Bigg(\int_{\s^1}\bigg(\int_{0}^{s_b}f(\vx-\e(t_b-s)\vw-\e
(s_b-r)\vw_t,\vw_t)\ue^{-(s_b-r)}\ud{r}\bigg)\ud{\vw_t}\Bigg)\ue^{-(t_b-s)}\ud{s}}
\leq \im{f}{\Omega\times\s^1}\nonumber.
\end{eqnarray}
\ \\
Step 3: Estimates of the last term in (\ref{well-posedness temp 8}).\\
Now we decompose the last term in (\ref{well-posedness temp 8}) as
\begin{eqnarray}
\int_{0}^{t_b}\int_{\s^1}\int_0^{s_b}=\int_{0}^{t_b}\int_{\s^1}\int_{s_b-r\leq\delta}+
\int_{0}^{t_b}\int_{\s^1}\int_{s_b-r\geq\delta}=I_1+I_2,
\end{eqnarray}
for some $\delta>0$. We can estimate $I_1$ directly as
\begin{eqnarray}\label{im temp 5}
I_1
&\leq&\int_{0}^{t_b}\ue^{-(t_b-s)}\bigg(\int_{\max(0,s_b-\delta)}^{s_b}\im{u}{\Omega\times\s^1}\ud{r}\bigg)\ud{s}\leq\delta\im{u}{\Omega\times\s^1}.
\end{eqnarray}
Then we can bound $I_2$ as
\begin{eqnarray}
I_2&\leq&C\int_{0}^{t_b}\int_{\s^1}\int_{0}^{\max(0,s_b-\delta)}\abs{\bar u\Big(\vx-\e(t_b-s)\vw-\e
(s_b-r)\vw_t\Big)}\ue^{-(t_b-s)}\ud{r}\ud{\vw_t}\ud{s}.
\end{eqnarray}
By the definition of $t_b$ and $s_b$, we always have
$\vx-\e(t_b-s)\vw-\e (s_b-r)\vw_t\in\bar\Omega$. Hence, we may
interchange the order of integration and apply H\"older's inequality
to obtain
\begin{eqnarray}\label{well-posedness temp 22}
I_2&\leq&C\int_{0}^{t_b}\Bigg(\int_{0}^{\max(0,s_b-\delta)}\int_{\s^1}{\bf{1}}_{\Omega}\Big(\vx-\e(t_b-s)\vw-\e
(s_b-r)\vw_t\Big)\\
&&\abs{\bar u\Big(\vx-\e(t_b-s)\vw-\e
(s_b-r)\vw_t\Big)}\ud{\vw_t}\ud{r}\Bigg)\ue^{-(t_b-s)}\ud{s}\nonumber\\
&\leq&C\int_{0}^{t_b}\Bigg(\bigg(\int_{\s^1}\int_{0}^{\max(0,s_b-\delta)}{\bf{1}}_{\Omega}\Big(\vx-\e(t_b-s)\vw-\e
(s_b-r)\vw_t\Big)\nonumber\\
&&\abs{\bar u\Big(\vx-\e(t_b-s)\vw-\e
(s_b-r)\vw_t\Big)}^2\ud{\vw_t}\ud{r}\bigg)^{\frac{1}{2}}\no\\
&&\times\bigg(\int_{\s^1}\int_{0}^{\max(0,s_b-\delta)}{\bf{1}}_{\Omega}\Big(\vx-\e(t_b-s)\vw-\e
(s_b-r)\vw_t\Big)\ud{\vw_t}\ud{r}\bigg)^{\frac{1}{2}}\Bigg)\ue^{-(t_b-s)}\ud{s}\no
\end{eqnarray}
Note $\vw_t\in\s^1$, which is essentially a one-dimensional
variable. Thus, we may write it in a new variable $\psi$ as
$\vw_t=(\cos\psi,\sin\psi)$. Then we define the change of variable
$[-\pi,\pi)\times\r\rt \Omega: (\psi,r)\rt(y_1,y_2)=\vec
y=\vx-\e(t_b-s)\vw-\e (s_b-r)\vw_t$, i.e.
\begin{eqnarray}
\left\{
\begin{array}{rcl}
y_1&=&x_1-\e(t_b-s)w_1-\e (s_b-r)\cos\psi,\\
y_2&=&x_2-\e(t_b-s)w_2-\e (s_b-r)\sin\psi.
\end{array}
\right.
\end{eqnarray}
Therefore, for $s_b-r\geq\delta$, we can directly compute the Jacobian
\begin{eqnarray}
\abs{\frac{\p{(y_1,y_2)}}{\p{(\psi,r)}}}=\abs{\abs{\begin{array}{rc}
-\e(s_b-r)\sin\psi&\e\cos\psi\\
\e(s_b-r)\cos\psi&\e\sin\psi
\end{array}}}=\e^2(s_b-r)\geq\e^2\delta.
\end{eqnarray}
Hence, we may simplify (\ref{well-posedness temp 22}) as
\begin{eqnarray}\label{im temp 6}
I_2&\leq&C\int_{0}^{t_b}\bigg(\int_{\Omega}\frac{1}{\e^2\d}\abs{\bar u(\vec
y)}^2\ud{\vec y}\bigg)^{\frac{1}{2}}\ue^{-(t_b-s)}\ud{s}\\
&\leq&\frac{C}{\e\sqrt{\delta}}\int_{0}^{t_b}\bigg(\int_{\Omega}\abs{\bar u(\vec
y)}^2\ud{\vec y}\bigg)^{\frac{1}{2}}\ue^{-(t_b-s)}\ud{s}\no\\
&\leq&\frac{C}{\e\d^{\frac{1}{2}}}\tm{\bar u}{\Omega}.\no
\end{eqnarray}
\ \\
Step 4: Synthesis.\\
In summary, collecting (\ref{im temp 1}), (\ref{im temp 2}),
(\ref{im temp 3}), (\ref{im temp 4}), (\ref{im temp 5}) and (\ref{im
temp 6}), for fixed $0<\delta<1$, we have
\begin{eqnarray}
\im{u}{\Omega\times\s^1}\leq \delta
\im{u}{\Omega\times\s^1}+\frac{C}{\e\d^{\frac{1}{2}}}\tm{\bar u}{\Omega}+C
\bigg(\im{f}{\Omega\times\s^1}+\im{h}{\Gamma^-}\bigg).
\end{eqnarray}
Then taking $\d$ small to absorb $\delta
\im{u}{\Omega\times\s^1}$ into the left-hand side to get
\begin{eqnarray}
\im{u}{\Omega\times\s^1}\leq
\frac{C}{\e}\tm{\bar u}{\Omega}+C
\bigg(\im{f}{\Omega\times\s^1}+\im{h}{\Gamma^-}\bigg).
\end{eqnarray}
Using Theorem \ref{LT estimate}, we get
\begin{eqnarray}
\im{u}{\Omega\times\s^1}&\leq& C\bigg(\frac{1}{\e^{3}}\tm{f}{\Omega\times\s^1}+\im{f}{\Omega\times\s^1}+\frac{1}{\e^{\frac{3}{2}}}\tm{h}{\Gamma^-}+\im{h}{\Gamma^-}\bigg).
\end{eqnarray}

\end{proof}

\subsection{$L^{2m}$ Estimate}

In this subsection, we try to improve previous estimates. In the following, we assume $m>2$ is an integer and let $o(1)$ denote a sufficiently small constant.
\begin{theorem}\label{LN estimate}
The unique solution $u(\vx,\vw)$ to the equation (\ref{neutron}) satisfies
\begin{eqnarray}
&&\frac{1}{\e^{\frac{1}{2}}}\nm{u}_{L^2(\Gamma^+)}+\nm{
\bar u}_{L^{2m}(\Omega\times\s^1)}+\frac{1}{\e}\nm{u-\bar
u}_{L^2(\Omega\times\s^1)}\\
&\leq&
C\bigg(o(1)\e^{\frac{1}{m}}\nm{u}_{L^{\infty}(\Omega\times\s^1)}+\frac{1}{\e}\tm{f}{\Omega\times\s^1}+
\frac{1}{\e^2}\nm{f}_{L^{\frac{2m}{2m-1}}(\Omega\times\s^1)}+\frac{1}{\e^{\frac{1}{2}}}\nm{h}_{L^2(\Gamma^-)}+\nm{h}_{L^{m}(\Gamma^-)}\bigg).\nonumber
\end{eqnarray}
\end{theorem}
\begin{proof}
We divide the proof into several steps:\\
\ \\
Step 1: Kernel Estimate.\\
Applying Green's identity to the
equation (\ref{neutron}). Then for any
$\phi\in L^2(\Omega\times\s^1)$ satisfying $\vw\cdot\nx\phi\in
L^2(\Omega\times\s^1)$ and $\phi\in L^2(\Gamma)$, we have
\begin{eqnarray}\label{well-posedness temp 4.}
\e\int_{\Gamma}u\phi\ud{\gamma}
-\e\iint_{\Omega\times\s^1}(\vw\cdot\nx\phi)u+\iint_{\Omega\times\s^1}(u-\bar
u)\phi=\iint_{\Omega\times\s^1}f\phi.
\end{eqnarray}
Our goal is to choose a particular test function $\phi$. We first
construct an auxiliary function $\xi$. Naturally $u\in
L^{\infty}(\Omega\times\s^1)$ implies $\bar u\in
L^{2m}(\Omega)$ which further leads to $(\bar u)^{2m-1}\in
L^{\frac{2m}{2m-1}}(\Omega)$. We define $\xi(\vx)$ on $\Omega$ satisfying
\begin{eqnarray}\label{test temp 1.}
\left\{
\begin{array}{rcl}
\Delta \xi&=&(\bar u)^{2m-1}\ \ \text{in}\ \
\Omega,\\\rule{0ex}{1.0em} \xi&=&0\ \ \text{on}\ \ \p\Omega.
\end{array}
\right.
\end{eqnarray}
In the bounded domain $\Omega$, based on the standard elliptic
estimates, we have a unique $\xi$ satisfying
\begin{eqnarray}\label{test temp 3.}
\nm{\xi}_{W^{2,\frac{2m}{2m-1}}(\Omega)}\leq C\nm{(\bar
u)^{2m-1}}_{L^{\frac{2m}{2m-1}}(\Omega)}= C\nm{\bar
u}_{L^{2m}(\Omega)}^{2m-1}.
\end{eqnarray}
We plug the test function
\begin{eqnarray}\label{test temp 2.}
\phi=-\vw\cdot\nx\xi
\end{eqnarray}
into the weak formulation (\ref{well-posedness temp 4.}) and estimate
each term there. By Sobolev embedding theorem, we have
\begin{eqnarray}
\nm{\phi}_{L^2(\Omega)}&\leq& C\nm{\xi}_{H^1(\Omega)}\leq C\nm{\xi}_{W^{2,\frac{2m}{2m-1}}(\Omega)}\leq
C\nm{\bar
u}_{L^{2m}(\Omega)}^{2m-1},\label{test temp 6.}\\
\nm{\phi}_{L^{\frac{2m}{2m-1}}(\Omega)}&\leq&C\nm{\xi}_{W^{1,\frac{2m}{2m-1}}(\Omega)}\leq
C\nm{\bar
u}_{L^{2m}(\Omega)}^{2m-1}.\label{test temp 4.}
\end{eqnarray}
Easily we can decompose
\begin{eqnarray}\label{test temp 5.}
-\e\iint_{\Omega\times\s^1}(\vw\cdot\nx\phi)u&=&-\e\iint_{\Omega\times\s^1}(\vw\cdot\nx\phi)\bar
u-\e\iint_{\Omega\times\s^1}(\vw\cdot\nx\phi)(u-\bar
u).
\end{eqnarray}
We estimate the two term on the right-hand side of (\ref{test temp 5.}) separately. By
(\ref{test temp 1.}) and (\ref{test temp 2.}), we have
\begin{eqnarray}\label{wellposed temp 1.}
-\e\iint_{\Omega\times\s^1}(\vw\cdot\nx\phi)\bar
u&=&\e\iint_{\Omega\times\s^1}\bar
u\bigg(w_1(w_1\p_{11}\xi+w_2\p_{12}\xi)+w_2(w_1\p_{12}\xi+w_2\p_{22}\xi)\bigg)\\
&=&\e\iint_{\Omega\times\s^1}\bar
u\bigg(w_1^2\p_{11}\xi+w_2^2\p_{22}\xi\bigg)\nonumber\\
&=&2\e\pi\int_{\Omega}\bar u(\p_{11}\xi+\p_{22}\xi)\nonumber\\
&=&\e\nm{\bar u}_{L^{2m}(\Omega)}^{2m}\nonumber.
\end{eqnarray}
In the second equality, the cross terms vanish due to the symmetry
of the integral over $\s^1$. On the other hand, for the second term
in (\ref{test temp 5.}), H\"older's inequality and the elliptic
estimate imply
\begin{eqnarray}\label{wellposed temp 2.}
-\e\iint_{\Omega\times\s^1}(\vw\cdot\nx\phi)(u-\bar
u)&\leq&C\e\nm{u-\bar u}_{L^{2m}(\Omega\times\s^1)}\nm{\nx\phi}_{L^{\frac{2m}{2m-1}}(\Omega)}\\
&\leq&C\e\nm{u-\bar u}_{L^{2m}(\Omega\times\s^1)}\nm{\xi}_{W^{2,\frac{2m}{2m-1}}(\Omega)}\no\\
&\leq&C\e\nm{u-\bar
u}_{L^{2m}(\Omega\times\s^1)}\nm{\bar
u}_{L^{2m}(\Omega)}^{2m-1}\nonumber.
\end{eqnarray}
Based on (\ref{test temp 3.}), (\ref{test temp 6.}), (\ref{test temp 4.}), Sobolev embedding theorem and the trace theorem, we have
\begin{eqnarray}
\\
\nm{\nx\xi}_{L^{\frac{m}{m-1}}(\Gamma)}\leq C\nm{\nx\xi}_{W^{\frac{1}{2m},\frac{2m}{2m-1}}(\Gamma)}\leq C\nm{\nx\xi}_{W^{1,\frac{2m}{2m-1}}(\Omega)}\leq C\nm{\xi}_{W^{2,\frac{2m}{2m-1}}(\Omega)}\leq
C\nm{\bar
u}_{L^{2m}(\Omega)}^{2m-1}.\no
\end{eqnarray}
Based on (\ref{test temp 3.}), (\ref{test temp 4.}) and
H\"older's inequality, we have
\begin{eqnarray}\label{wellposed temp 3.}
\e\int_{\Gamma}u\phi\ud{\gamma}&=&\e\int_{\Gamma^+}u\phi\ud{\gamma}+\e\int_{\Gamma^-}u\phi\ud{\gamma}\\
&\leq&C\e\nm{\nx\xi}_{L^{\frac{m}{m-1}}(\Gamma)}\bigg(\nm{u}_{L^{m}(\Gamma^+)}+\nm{h}_{L^{m}(\Gamma^-)}\bigg)\no\\
&\leq&C\e\nm{\bar u}_{L^{2m}(\Omega\times\s^1)}^{2m-1}\bigg(\nm{u}_{L^{m}(\Gamma^+)}+\nm{h}_{L^{m}(\Gamma^-)}\bigg).\no
\end{eqnarray}
Also, we have
\begin{eqnarray}\label{wellposed temp 5.}
\iint_{\Omega\times\s^1}(u-\bar u)\phi\leq
C\nm{\phi}_{L^2(\Omega\times\s^1)}\nm{u-\bar
u}_{L^2(\Omega\times\s^1)}\leq
C\nm{\bar
u}_{L^{2m}(\Omega)}^{2m-1}\nm{u-\bar
u}_{L^2(\Omega\times\s^1)},
\end{eqnarray}
\begin{eqnarray}\label{wellposed temp 6.}
\iint_{\Omega\times\s^1}f\phi\leq C\nm{\phi}_{L^2(\Omega\times\s^1)}\nm{f}_{L^2(\Omega\times\s^1)}\leq C\nm{\bar
u}_{L^{2m}(\Omega)}^{2m-1}\nm{f}_{L^2(\Omega\times\s^1)}.
\end{eqnarray}
Collecting terms in (\ref{wellposed temp 1.}), (\ref{wellposed temp
2.}), (\ref{wellposed temp 3.}),
(\ref{wellposed temp 5.}) and (\ref{wellposed temp 6.}), we obtain
\begin{eqnarray}\label{improve temp 1.}
\e\nm{\bar u}_{L^{2m}(\Omega\times\s^1)}&\leq&
C\bigg(\e\nm{u-\bar
u}_{L^{2m}(\Omega\times\s^1)}+\nm{u-\bar
u}_{L^2(\Omega\times\s^1)}+\e\nm{u}_{L^{m}(\Gamma^+)}\\
&&+\nm{f}_{L^2(\Omega\times\s^1)}+\e\nm{h}_{L^{m}(\Gamma^-)}\bigg)\nonumber,
\end{eqnarray}
\ \\
Step 2: Energy Estimate.\\
In the weak formulation (\ref{well-posedness temp 4.}), we may take
the test function $\phi=u$ to get the energy estimate
\begin{eqnarray}\label{well-posedness temp 9.}
\half\e\int_{\Gamma}\abs{u}^2\ud{\gamma}+\nm{u-\bar
u}_{L^2(\Omega\times\s^1)}^2=\iint_{\Omega\times\s^1}fu.
\end{eqnarray}
Hence, as in $L^2$ estimates, this naturally implies
\begin{eqnarray}\label{well-posedness temp 5.}
\e\nm{u}_{L^2(\Gamma^+)}^2+\nm{u-\bar
u}_{L^2(\Omega\times\s^1)}^2= \iint_{\Omega\times\s^1}fu+\e\nm{h}_{L^2(\Gamma^-)}^2.
\end{eqnarray}
On the other hand, we can square on both sides of
(\ref{improve temp 1.}) to obtain
\begin{eqnarray}\label{well-posedness temp 6.}
\e^2\nm{\bar u}_{L^{2m}(\Omega\times\s^1)}^2&\leq&
C\bigg(\e^2\nm{u-\bar
u}_{L^{2m}(\Omega\times\s^1)}^2+\nm{u-\bar
u}_{L^2(\Omega\times\s^1)}^2+\e^2\nm{u}_{L^{m}(\Gamma^+)}\\
&&+\nm{f}_{L^2(\Omega\times\s^1)}^2+\e^2\nm{h}_{L^{m}(\Gamma^-)}^2\bigg)\nonumber,
\end{eqnarray}
Multiplying a sufficiently small constant on both sides of
(\ref{well-posedness temp 6.}) and adding it to (\ref{well-posedness
temp 5.}) to absorb
$\nm{u-\bar u}_{L^2(\Omega\times\s^1)}^2$, we deduce
\begin{eqnarray}\label{wt 03.}
&&\e\nm{u}_{L^2(\Gamma^+)}^2+\e^2\nm{\bar
u}_{L^{2m}(\Omega\times\s^1)}^2+\nm{u-\bar
u}_{L^2(\Omega\times\s^1)}^2\\
&\leq&
C\bigg(\e^2\nm{u-\bar
u}_{L^{2m}(\Omega\times\s^1)}^2+\e^2\nm{u}_{L^{m}(\Gamma^+)}+\tm{f}{\Omega\times\s^1}^2+
\iint_{\Omega\times\s^1}fu+\e\nm{h}_{L^2(\Gamma^-)}^2+\e^2\nm{h}_{L^{m}(\Gamma^-)}^2\bigg).\nonumber
\end{eqnarray}
By interpolation estimate and Young's inequality, we have
\begin{eqnarray}
\nm{u}_{L^{m}(\Gamma^+)}&\leq&\nm{u}_{L^2(\Gamma^+)}^{\frac{2}{m}}\nm{u}_{L^{\infty}(\Gamma^+)}^{\frac{m-2}{m}}\\
&=&\bigg(\frac{1}{\e^{\frac{m-2}{m^2}}}\nm{u}_{L^2(\Gamma^+)}^{\frac{2}{m}}\bigg)
\bigg(\e^{\frac{m-2}{m^2}}\nm{u}_{L^{\infty}(\Gamma^+)}^{\frac{m-2}{m}}\bigg)\no\\
&\leq&C\bigg(\frac{1}{\e^{\frac{m-2}{m^2}}}\nm{u}_{L^2(\Gamma^+)}^{\frac{2}{m}}\bigg)^{\frac{m}{2}}+o(1)
\bigg(\e^{\frac{m-2}{m^2}}\nm{u}_{L^{\infty}(\Gamma^+)}^{\frac{m-2}{m}}\bigg)^{\frac{m}{m-2}}\no\\
&\leq&\frac{C}{\e^{\frac{m-2}{2m}}}\nm{u}_{L^2(\Gamma^+)}+o(1)\e^{\frac{1}{m}}\nm{u}_{L^{\infty}(\Gamma^+)}\no\\
&\leq&\frac{C}{\e^{\frac{m-2}{2m}}}\nm{u}_{L^2(\Gamma^+)}+o(1)\e^{\frac{1}{m}}\nm{u}_{L^{\infty}(\Omega\times\s^1)}.\no
\end{eqnarray}
Similarly, we have
\begin{eqnarray}
\nm{u-\bar u}_{L^{2m}(\Omega\times\s^1)}&\leq&\nm{u-\bar u}_{L^2(\Omega\times\s^1)}^{\frac{1}{m}}\nm{u-\bar u}_{L^{\infty}(\Omega\times\s^1)}^{\frac{m-1}{m}}\\
&=&\bigg(\frac{1}{\e^{\frac{m-1}{m^2}}}\nm{u-\bar u}_{L^2(\Omega\times\s^1)}^{\frac{1}{m}}\bigg)\bigg(\e^{\frac{m-1}{m^2}}\nm{u-\bar u}_{L^{\infty}(\Omega\times\s^1)}^{\frac{m-1}{m}}\bigg)\no\\
&\leq&C\bigg(\frac{1}{\e^{\frac{m-1}{m^2}}}\nm{u-\bar u}_{L^2(\Omega\times\s^1)}^{\frac{1}{m}}\bigg)^{m}+o(1)\bigg(\e^{\frac{m-1}{m^2}}\nm{u-\bar u}_{L^{\infty}(\Omega\times\s^1)}^{\frac{m-1}{m}}\bigg)^{\frac{m}{m-1}}\no\\
&\leq&\frac{C}{\e^{\frac{m-1}{m}}}\nm{u-\bar u}_{L^2(\Omega\times\s^1)}+o(1)\e^{\frac{1}{m}}\nm{u-\bar u}_{L^{\infty}(\Omega\times\s^1)}.\no
\end{eqnarray}
We need this extra $\e^{\frac{1}{m}}$ for the convenience of $L^{\infty}$ estimate.
Then we know for sufficiently small $\e$,
\begin{eqnarray}
\e^2\nm{u}_{L^{m}(\Gamma^+)}^2
&\leq&C\e^{2-\frac{m-2}{m}}\nm{u}_{L^2(\Gamma^+)}^2+o(1)\e^{2+\frac{2}{m}}\nm{u}_{L^{\infty}(\Gamma^+)}^2\\
&\leq&o(1)\e\nm{u}_{L^2(\Gamma^+)}^2+o(1)\e^{2+\frac{2}{m}}\nm{u}_{L^{\infty}(\Omega\times\s^1)}^2.\no
\end{eqnarray}
Similarly, we have
\begin{eqnarray}
\e^2\nm{u-\bar
u}_{L^{2m}(\Omega\times\s^1)}^2&\leq&\e^{2-\frac{2m-2}{m}}\nm{u-\bar u}_{L^2(\Omega\times\s^1)}^2+o(1)\e^{2+\frac{2}{m}}\nm{u}_{L^{\infty}(\Omega\times\s^1)}^2\\
&\leq& o(1)\nm{u-\bar u}_{L^2(\Omega\times\s^1)}^2+o(1)\e^{2+\frac{2}{m}}\nm{u}_{L^{\infty}(\Omega\times\s^1)}^2.\no
\end{eqnarray}
In (\ref{wt 03.}), we can absorb $\nm{u-\bar u}_{L^2(\Omega\times\s^1)}$ and $\e\nm{u}_{L^2(\Gamma^+)}^2$ into left-hand side to obtain
\begin{eqnarray}\label{wt 04.}
&&\e\nm{u}_{L^2(\Gamma^+)}^2+\e^2\nm{\bar
u}_{L^{2m}(\Omega\times\s^1)}^2+\nm{u-\bar
u}_{L^2(\Omega\times\s^1)}^2\\
&\leq&
C\bigg(o(1)\e^{2+\frac{2}{m}}\nm{u}_{L^{\infty}(\Omega\times\s^1)}^2+\tm{f}{\Omega\times\s^1}^2+
\iint_{\Omega\times\s^1}fu+\e\nm{h}_{L^2(\Gamma^-)}^2+\e^2\nm{h}_{L^{m}(\Gamma^-)}^2\bigg).\no
\end{eqnarray}
We can decompose
\begin{eqnarray}
\iint_{\Omega\times\s^1}fu=\iint_{\Omega\times\s^1}f\bar u+\iint_{\Omega\times\s^1}f(u-\bar u).
\end{eqnarray}
H\"older's inequality and Cauchy's inequality imply
\begin{eqnarray}
\iint_{\Omega\times\s^1}f\bar u\leq\nm{f}_{L^{\frac{2m}{2m-1}}(\Omega\times\s^1)}\nm{\bar u}_{L^{2m}(\Omega\times\s^1)}
\leq\frac{C}{\e^{2}}\nm{f}_{L^{\frac{2m}{2m-1}}(\Omega\times\s^1)}^2+o(1)\e^2\nm{\bar u}_{L^{2m}(\Omega\times\s^1)}^2,
\end{eqnarray}
and
\begin{eqnarray}
\iint_{\Omega\times\s^1}f(u-\bar u)\leq C\nm{f}_{L^{2}(\Omega\times\s^1)}^2+o(1)\nm{u-\bar u}_{L^2(\Omega\times\s^1)}^2.
\end{eqnarray}
Hence, absorbing $\e^2\nm{\bar u}_{L^{2m}(\Omega\times\s^1)}^2$ and $\nm{u-\bar u}_{L^2(\Omega\times\s^1)}^2$ into left-hand side of (\ref{wt 04.}), we get
\begin{eqnarray}\label{wt 06.}
&&\e\nm{u}_{L^2(\Gamma^+)}^2+\e^2\nm{\bar
u}_{L^{2m}(\Omega\times\s^1)}^2+\nm{u-\bar
u}_{L^2(\Omega\times\s^1)}^2\\
&\leq&
C\bigg(o(1)\e^{2+\frac{2}{m}}\nm{u}_{L^{\infty}(\Omega\times\s^1)}^2+\tm{f}{\Omega\times\s^1}^2+
\frac{1}{\e^2}\nm{f}_{L^{\frac{2m}{2m-1}}(\Omega\times\s^1)}^2+\e\nm{h}_{L^2(\Gamma^-)}^2+\e^2\nm{h}_{L^{m}(\Gamma^-)}^2\bigg),\nonumber
\end{eqnarray}
which implies
\begin{eqnarray}\label{wt 07.}
&&\frac{1}{\e^{\frac{1}{2}}}\nm{u}_{L^2(\Gamma^+)}+\nm{
\bar u}_{L^{2m}(\Omega\times\s^1)}+\frac{1}{\e}\nm{u-\bar
u}_{L^2(\Omega\times\s^1)}\\
&\leq&
C\bigg(o(1)\e^{\frac{1}{m}}\nm{u}_{L^{\infty}(\Omega\times\s^1)}+\frac{1}{\e}\tm{f}{\Omega\times\s^1}+
\frac{1}{\e^2}\nm{f}_{L^{\frac{2m}{2m-1}}(\Omega\times\s^1)}+\frac{1}{\e^{\frac{1}{2}}}\nm{h}_{L^2(\Gamma^-)}+\nm{h}_{L^{m}(\Gamma^-)}\bigg).\nonumber
\end{eqnarray}

\end{proof}

\subsection{$L^{\infty}$ Estimate - Second Round}

\begin{theorem}\label{LI estimate}
The unique solution $u(\vx,\vw)$ to the equation (\ref{neutron}) satisfies
\begin{eqnarray}
\im{u}{\Omega\times\s^1}&\leq& C\bigg(\frac{1}{\e^{1+\frac{1}{m}}}\tm{f}{\Omega\times\s^1}+
\frac{1}{\e^{2+\frac{1}{m}}}\nm{f}_{L^{\frac{2m}{2m-1}}(\Omega\times\s^1)}+\im{f}{\Omega\times\s^1}\\
&&+\frac{1}{\e^{\frac{1}{2}+\frac{1}{m}}}\nm{h}_{L^2(\Gamma^-)}+\frac{1}{\e^{\frac{1}{m}}}\nm{h}_{L^{m}(\Gamma^-)}+\im{h}{\Gamma^-}\bigg).\no
\end{eqnarray}
\end{theorem}
\begin{proof}
Following the argument in the proof of Theorem \ref{LI estimate.}, by double Duhamel's principle along the characteristics,
we may apply H\"{o}lder's inequality to obtain
\begin{eqnarray}\label{well-posedness temp 22.}
I_2&\leq&
C\int_{0}^{t_b}\Bigg(\bigg(\int_{\s^1}\int_{0}^{\max(0,s_b-\delta)}{\bf{1}}_{\Omega}\Big(\vx-\e(t_b-s)\vw-\e
(s_b-r)\vw_t\Big)\\
&&\abs{\bar u\Big(\vx-\e(t_b-s)\vw-\e
(s_b-r)\vw_t\Big)}^{2m}\ud{\vw_t}\ud{r}\bigg)^{\frac{1}{2m}}\no\\
&&\times\bigg(\int_{\s^1}\int_{0}^{\max(0,s_b-\delta)}{\bf{1}}_{\Omega}\Big(\vx-\e(t_b-s)\vw-\e
(s_b-r)\vw_t\Big)\ud{\vw_t}\ud{r}\bigg)^{\frac{2m-1}{2m}}\Bigg)\ue^{-(t_b-s)}\ud{s}\no
\end{eqnarray}
Then, using the same substitution, for
$\vw_t=(\cos\psi,\sin\psi)$, we define the change of variable
$[-\pi,\pi)\times\r\rt \Omega: (\psi,r)\rt(y_1,y_2)=\vec
y=\vx-\e(t_b-s)\vw-\e (s_b-r)\vw_t$, which, for $s_b-r\geq\delta$, implies the Jacobian
\begin{eqnarray}
\abs{\frac{\p{(y_1,y_2)}}{\p{(\psi,r)}}}=\abs{\abs{\begin{array}{rc}
-\e(s_b-r)\sin\psi&\e\cos\psi\\
\e(s_b-r)\cos\psi&\e\sin\psi
\end{array}}}=\e^2(s_b-r)\geq\e^2\delta.
\end{eqnarray}
Hence, we may simplify (\ref{well-posedness temp 22.}) as
\begin{eqnarray}
I_2&\leq&C\int_{0}^{t_b}\bigg(\int_{\Omega}\frac{1}{\e^2\d}\abs{\bar u(\vec
y)}^{2m}\ud{\vec y}\bigg)^{\frac{1}{2m}}\ue^{-(t_b-s)}\ud{s}\\
&\leq&\frac{C}{\e^{\frac{1}{m}}\d^{\frac{1}{2m}}}\int_{0}^{t_b}\bigg(\int_{\Omega}\abs{\bar u(\vec
y)}^{2m}\ud{\vec y}\bigg)^{\frac{1}{2m}}\ue^{-(t_b-s)}\ud{s}\no\\
&\leq&\frac{C}{\e^{\frac{1}{m}}\d^{\frac{1}{2m}}}\nm{
\bar u}_{L^{2m}(\Omega\times\s^1)}.\no
\end{eqnarray}
Hence, for fixed $0<\delta<1$, we have
\begin{eqnarray}
\im{u}{\Omega\times\s^1}\leq \delta
\im{u}{\Omega\times\s^1}+\frac{C}{\e^{\frac{1}{m}}\d^{\frac{1}{2m}}}\nm{
\bar u}_{L^{2m}(\Omega\times\s^1)}+C
\bigg(\im{f}{\Omega\times\s^1}+\im{h}{\Gamma^-}\bigg).
\end{eqnarray}
Then taking $\d$ small to absorb $\delta
\im{u}{\Omega\times\s^1}$ into the left-hand side to get
\begin{eqnarray}
\im{u}{\Omega\times\s^1}\leq
\frac{C}{\e^{\frac{1}{m}}\d^{\frac{1}{2m}}}\nm{
\bar u}_{L^{2m}(\Omega\times\s^1)}+C
\bigg(\im{f}{\Omega\times\s^1}+\im{h}{\Gamma^-}\bigg).
\end{eqnarray}
Using Theorem \ref{LN estimate}, we get
\begin{eqnarray}
\im{u}{\Omega\times\s^1}&\leq& C\bigg(\frac{1}{\e^{1+\frac{1}{m}}}\tm{f}{\Omega\times\s^1}+
\frac{1}{\e^{2+\frac{1}{m}}}\nm{f}_{L^{\frac{2m}{2m-1}}(\Omega\times\s^1)}+\im{f}{\Omega\times\s^1}\\
&&+\frac{1}{\e^{\frac{1}{2}+\frac{1}{m}}}\nm{h}_{L^2(\Gamma^-)}+\frac{1}{\e^{\frac{1}{m}}}\nm{h}_{L^{m}(\Gamma^-)}+\im{h}{\Gamma^-}\bigg)+o(1)\im{u}{\Omega\times\s^1}.\no
\end{eqnarray}
Absorbing $\im{u}{\Omega\times\s^1}$ into the left-hand side, we obtain
\begin{eqnarray}
\im{u}{\Omega\times\s^1}&\leq& C\bigg(\frac{1}{\e^{1+\frac{1}{m}}}\tm{f}{\Omega\times\s^1}+
\frac{1}{\e^{2+\frac{1}{m}}}\nm{f}_{L^{\frac{2m}{2m-1}}(\Omega\times\s^1)}+\im{f}{\Omega\times\s^1}\\
&&+\frac{1}{\e^{\frac{1}{2}+\frac{1}{m}}}\nm{h}_{L^2(\Gamma^-)}+\frac{1}{\e^{\frac{1}{m}}}\nm{h}_{L^{m}(\Gamma^-)}+\im{h}{\Gamma^-}\bigg).\no
\end{eqnarray}

\end{proof}

\newpage

\section{Well-Posedness of $\e$-Milne Problem with Geometric Correction}

We consider the $\e$-Milne problem with geometric correction for $f^{\e}(\eta,\tau,\phi)$ in
the domain $(\eta,\tau,\phi)\in[0,L]\times[-\pi,\pi)\times[-\pi,\pi)$ where $L=\e^{-\frac{1}{2}}$ as
\begin{eqnarray}\label{Milne problem.}
\left\{ \begin{array}{l}\displaystyle
\sin\phi\frac{\p
f^{\e}}{\p\eta}+F(\e;\eta,\tau)\cos\phi\frac{\p
f^{\e}}{\p\phi}+f^{\e}-\bar f^{\e}=S^{\e}(\eta,\tau,\phi),\\\rule{0ex}{1.5em}
f^{\e}(0,\tau,\phi)= h^{\e}(\tau,\phi)\ \ \text{for}\
\ \sin\phi>0,\\\rule{0ex}{1.5em}
f^{\e}(L,\tau,\phi)=f^{\e}(L,\tau,\rr[\phi]),
\end{array}
\right.
\end{eqnarray}
where $\rr[\phi]=-\phi$ and
\begin{eqnarray}
F(\e;\eta,\tau)=-\frac{\e}{\rk(\tau)-\e\eta},
\end{eqnarray}
for the radius of curvature $\rk$. In this section, for convenience, we temporarily ignore the superscript on $\e$ and $\tau$. In other words, we will study
\begin{eqnarray}\label{Milne problem}
\left\{ \begin{array}{l}\displaystyle
\sin\phi\frac{\p
f}{\p\eta}+F(\eta)\cos\phi\frac{\p
f}{\p\phi}+f-\bar f=S(\eta,\phi),\\\rule{0ex}{1.5em}
f(0,\phi)= h(\phi)\ \ \text{for}\
\ \sin\phi>0,\\\rule{0ex}{1.5em}
f(L,\phi)=f(L,\rr[\phi]).
\end{array}
\right.
\end{eqnarray}
Define potential function $V(\eta)$ satisfying $V(0)=0$ and $\dfrac{\p V}{\p\eta}=-F(\eta)$. Then we can direct compute
\begin{eqnarray}
V(\eta)=\ln\left(\frac{\rk}{\rk-\e\eta}\right).
\end{eqnarray}
Define the weight function
\begin{eqnarray}
\zeta(\eta,\phi)=\Bigg(1-\bigg(\frac{\rk-\e\eta}{\rk}\cos\phi\bigg)^2\Bigg)^{\frac{1}{2}}.
\end{eqnarray}
We can easily show that
\begin{eqnarray}
\sin\phi\frac{\p \zeta}{\p\eta}+F(\eta)\cos\phi\frac{\p
\zeta}{\p\phi}=0.
\end{eqnarray}
We define the norms in the
space $(\eta,\phi)\in[0,\infty)\times[-\pi,\pi)$ as follows:
\begin{eqnarray}
\tnnm{f}&=&\bigg(\int_0^{L}\int_{-\pi}^{\pi}\abs{f(\eta,\phi)}^2\ud{\phi}\ud{\eta}\bigg)^{\frac{1}{2}},\\
\lnnm{f}&=&\sup_{(\eta,\phi)\in[0,L]\times[-\pi,\pi)}\abs{f(\eta,\phi)}.
\end{eqnarray}
Similarly,
\begin{eqnarray}
\tnm{f(\eta)}&=&\bigg(\int_{-\pi}^{\pi}\abs{f(\eta,\phi)}^2\ud{\phi}\bigg)^{\frac{1}{2}},\\
\lnm{f(\eta)}&=&\sup_{\phi\in[-\pi,\pi)}\abs{f(\eta,\phi)}.
\end{eqnarray}
Also, we define the weighted norms at in-flow boundary as
\begin{eqnarray}
\tss{h}{-}&=&\bigg(\int_{\sin\phi>0}\abs{h(\phi)}^2\sin\phi\ud{\phi}\bigg)^{\frac{1}{2}},\\
\lss{h}{-}&=&\sup_{\sin\phi>0}\abs{h(\phi)}.
\end{eqnarray}
Also define
\begin{eqnarray}
\br{f,g}_{\phi}(\eta)=\int_{-\pi}^{\pi}f(\eta,\phi)g(\eta,\phi)\ud{\phi},
\end{eqnarray}
as the $L^2$ inner product in $\phi$.

In the following, we will always assume that for some $K>0$,
\begin{eqnarray}
\lss{h}{-}+\lnnm{\ue^{K\eta}S}\leq C.
\end{eqnarray}
The well-posedness, exponential decay and maximum principle of the equation (\ref{Milne problem}) has been well studied in \cite{AA003}. Here we will focus on the a priori estimates and present detail structure of the dependence of the boundary data $h$ and the source term $S$.

\subsection{$L^2$ Estimates}

\subsubsection{$\bar S=0$ Case}

Assume that $S$ satisfies $\bar S(\eta)=0$ for any
$\eta$. We may decompose the solution
\begin{eqnarray}
f(\eta,\phi)=q_f(\eta)+r_f(\eta,\phi),
\end{eqnarray}
where the hydrodynamical part $q_f$ is in the null space of the
operator $f-\bar f$, and the microscopic part $r_f$ is
the orthogonal complement, i.e.
\begin{eqnarray}\label{hydro}
q_f(\eta)=\frac{1}{2\pi}\int_{-\pi}^{\pi}f(\eta,\phi)\ud{\phi}=\bar f,\quad
r_f(\eta,\phi)=f(\eta,\phi)-q_f(\eta).
\end{eqnarray}
In the following, when there is no confusion, we simply write
$f=q+r$.
\begin{lemma}\label{Milne finite LT}
Assume $\bar S(\eta)=0$ for any $\eta\in[0,L]$. Then the unique solution $f(\eta,\phi)$ to the equation
(\ref{Milne problem}) satisfies
\begin{eqnarray}
&&\tnnm{r}\leq C\bigg(\tss{h}{-}+\tnnm{S}\bigg)\label{Milne temp 8},
\end{eqnarray}
and there exists $q_L\in\r$ such that
\begin{eqnarray}
\abs{q_L}&\leq&C\bigg(\tss{h}{-}+\tnnm{S}\bigg)+C\abs{\int_0^{L}\br{\sin\phi,S}_{\phi}(y)\ud{y}}\label{Milne temp 17},\\
\tnnm{q-q_L}&\leq&C\bigg(\tss{h}{-}+\tnnm{S}\bigg)+C\Bigg(\int_0^{L}\bigg(\int_{\eta}^{L}\br{\sin\phi,S}_{\phi}(y)\ud{y}\bigg)^2\ud{\eta}\Bigg)^{\frac{1}{2}}.\label{Milne temp 10}
\end{eqnarray}
Also, for any $\eta\in[0,L]$,
\begin{eqnarray}
&&\br{\sin\phi,r}_{\phi}(\eta)=0\label{Milne temp 19}.
\end{eqnarray}
\end{lemma}
\begin{proof}
We divide the proof into several steps:\\
\ \\
Step 1: Estimate of $r$.\\
Multiplying $f$
on both sides of (\ref{Milne problem}) and
integrating over $\phi\in[-\pi,\pi)$, we get the energy estimate
\begin{eqnarray}\label{Milne temp 31}
\half\frac{\ud{}}{\ud{\eta}}\br{f
,f\sin\phi}_{\phi}(\eta)+F(\eta)\br{\frac{\p
f}{\p\phi},f\cos\phi}_{\phi}(\eta)+\tnm{r(\eta)}^2=\br{S,f}_{\phi}(\eta).
\end{eqnarray}
An integration by parts reveals
\begin{eqnarray}
F(\eta)\br{\frac{\p
f}{\p\phi},f\cos\phi}_{\phi}(\eta)&=&\half
F(\eta)\br{f,f\sin\phi}_{\phi}(\eta).
\end{eqnarray}
Also, the assumption $\bar S(\eta)=0$ leads to
\begin{eqnarray}
\br{S,f}_{\phi}(\eta)&=&\br{S,q}_{\phi}(\eta)+\br{S,r}_{\phi}(\eta)=\br{S,r}_{\phi}(\eta).
\end{eqnarray}
Hence, we have the simplified form of (\ref{Milne temp 31}) as
follows:
\begin{eqnarray}\label{Milne temp 32}
\half\frac{\ud{}}{\ud{\eta}}\br{
f,f\sin\phi}_{\phi}(\eta)+\half
F(\eta)\br{
f,f\sin\phi}_{\phi}(\eta)+\tnm{r(\eta)}^2=\br{S,r}_{\phi}(\eta).
\end{eqnarray}
Define
\begin{eqnarray}
\alpha(\eta)=\half\br{f,f\sin\phi }_{\phi}(\eta).
\end{eqnarray}
Then (\ref{Milne temp 32}) can be rewritten as follows:
\begin{eqnarray}
\frac{\ud{\alpha}}{\ud{\eta}}+F(\eta)\alpha(\eta)+\tnm{r(\eta)}^2=\br{S,r}_{\phi}(\eta).
\end{eqnarray}
We can solve this differential equation for $\alpha$ on $[\eta,L]$ and $[0,\eta]$ respectively to obtain
\begin{eqnarray}
\label{Milne temp 4}
\alpha(\eta)&=&\alpha(L)\exp\bigg(\int_{\eta}^LF(y)\ud(y)\bigg)+\int_{\eta}^L\exp\bigg(\int_{\eta}^yF(z)\ud{z}\bigg)
\bigg(\tnm{r(y)}^2-\br{S,r}_{\phi}(y)\bigg)\ud{y},\\
\label{Milne temp 5}
\alpha(\eta)&=&\alpha(0)\exp\bigg(-\int_{0}^{\eta}F(y)\ud(y)\bigg)+\int_{0}^{\eta}\exp\bigg(-\int_{y}^{\eta}F(z)\ud{z}\bigg)
\bigg(-\tnm{r(y)}^2+\br{S,r}_{\phi}(y)\bigg)\ud{y}.
\end{eqnarray}
The specular reflexive boundary $f(L,\phi)=f(L,\rr[\phi])$
ensures $\alpha(L)=0$. Hence, based on (\ref{Milne temp 4}), we have
\begin{eqnarray}\label{Milne temp 36}
\alpha(\eta)\geq\int_{\eta}^L\exp\bigg(\int_{\eta}^yF(z)\ud{z}\bigg)\bigg(-\br{S,r}_{\phi}(y)\bigg)\ud{y}\geq -C\int_{\eta}^L\br{S,r}_{\phi}(y)\ud{y}.
\end{eqnarray}
Also, (\ref{Milne temp 5}) implies
\begin{eqnarray}
\alpha(\eta)&\leq&\alpha(0)\exp\bigg(-\int_{0}^{\eta}F(y)\ud(y)\bigg)+\int_{0}^{\eta}\exp\bigg(-\int_{y}^{\eta}F(z)\ud{z}\bigg)
\bigg(\br{S,r}_{\phi}(y)\bigg)\ud{y}\\
&\leq&C\tss{h}{-}^2+C\int_{0}^{\eta}\bigg(\br{S,r}_{\phi}(y)\bigg)\ud{y}\nonumber,
\end{eqnarray}
due to the fact
\begin{eqnarray}
\alpha(0)=\half\br{\sin\phi
f,f}_{\phi}(0)\leq\half\bigg(\int_{\sin\phi>0}h^2(\phi)\sin\phi
\ud{\phi}\bigg)\leq C\tss{h}{-}^2.
\end{eqnarray}
Then in (\ref{Milne temp 5}) taking $\eta=L$, from $\alpha(L)=0$, we have
\begin{eqnarray}\label{Milne temp 6}
\int_{0}^L\exp\bigg(\int_{0}^yF(z)\ud{z}\bigg)\tnm{r(y)}^2\ud{y}
&\leq&\alpha(0)+\int_{0}^L\exp\bigg(\int_{0}^yF(z)\ud{z}\bigg)\br{S,r}_{\phi}(y)\ud{y}\\
&\leq&
C\tss{h}{-}^2+C\int_{0}^L\br{S,r}_{\phi}(y)\ud{y}\nonumber.
\end{eqnarray}
On the other hand, we can directly
estimate as follows:
\begin{eqnarray}\label{Milne temp 7}
\int_{0}^L\exp\bigg(\int_{0}^yF(z)\ud{z}\bigg)\tnm{r(y)}^2\ud{y}\geq
C\int_{0}^L\tnm{r(y)}^2\ud{y}.
\end{eqnarray}
Combining (\ref{Milne temp 6}) and (\ref{Milne temp 7}) yields
\begin{eqnarray}\label{Milne temp 33}
\int_{0}^L\tnm{r(\eta)}^2\ud{\eta}&\leq&
C\tss{h}{-}^2+C\int_{0}^L\br{S,r}_{\phi}(y)\ud{y}.
\end{eqnarray}
By Cauchy's inequality, we have
\begin{eqnarray}\label{Milne temp 34}
\abs{\int_0^L\br{S,r}_{\phi}(y)\ud{y}}\leq
C_0\int_{0}^L\tnm{r(\eta)}^2\ud{\eta}+\frac{4}{C_0}\int_{0}^L\tnm{S(\eta)}^2\ud{\eta},
\end{eqnarray}
for $C_0>0$ small. Therefore, absorbing $\ds\int_{0}^L\tnm{r(\eta)}^2\ud{\eta}$ and summarizing (\ref{Milne temp 33}) and (\ref{Milne temp
34}), we deduce
\begin{eqnarray}\label{Milne temp 37}
\int_{0}^L\tnm{r(\eta)}^2\ud{\eta}&\leq&
C\bigg(\tss{h}{-}^2+\int_{0}^L\tnm{S(\eta)}^2\ud{\eta}\bigg).
\end{eqnarray}
\ \\
Step 2: Orthogonality relation.\\
A direct integration over $\phi\in[-\pi,\pi)$ in
(\ref{Milne problem}) implies
\begin{eqnarray}
\frac{\ud{}}{\ud{\eta}}\br{\sin\phi,f}_{\phi}(\eta)=-F\br{\cos\phi,\frac{\ud{f}}{\ud{\phi}}}_{\phi}(\eta)
+\bar S(\eta)=-F\br{\sin\phi,f}_{\phi}(\eta),
\end{eqnarray}
due to $\bar S=0$. The specular reflexive boundary
$f(L,\phi)=f(L,\rr[\phi])$ implies
$\br{\sin\phi,f}_{\phi}(L)=0$. Then we have
\begin{eqnarray}
\br{\sin\phi,f}_{\phi}(\eta)=0.
\end{eqnarray}
It is easy to see
\begin{eqnarray}
\br{\sin\phi,q}_{\phi}(\eta)=0.
\end{eqnarray}
Hence, we may derive
\begin{eqnarray}
\br{\sin\phi,r}_{\phi}(\eta)=0.
\end{eqnarray}
This leads to orthogonal relation (\ref{Milne temp 19}).\\
\ \\
Step 3: Estimate of $q$.\\
Multiplying $\sin\phi$ on
both sides of (\ref{Milne problem}) and
integrating over $\phi\in[-\pi,\pi)$ lead to
\begin{eqnarray}
\label{Milne temp 18}
\frac{\ud{}}{\ud{\eta}}\br{\sin^2\phi,f}_{\phi}(\eta)=
-\br{\sin\phi,r}_{\phi}(\eta)-F(\eta)\br{\sin\phi\cos\phi,\frac{\p
f}{\p\phi}}_{\phi}(\eta)+\br{\sin\phi,S}_{\phi}(\eta).
\end{eqnarray}
We can further integrate by parts as follows:
\begin{eqnarray}
-F(\eta)\br{\sin\phi\cos\phi,\frac{\p
f}{\p\phi}}_{\phi}(\eta)=F(\eta)\br{\cos(2\phi),f}_{\phi}(\eta)=F(\eta)\br{\cos(2\phi),r}_{\phi}(\eta).
\end{eqnarray}
Using the orthogonal relation (\ref{Milne temp 19}), we obtain
\begin{eqnarray}
\frac{\ud{}}{\ud{\eta}}\br{\sin^2\phi,f}_{\phi}(\eta)&=&F(\eta)\br{\cos(2\phi),r}_{\phi}(\eta)+\br{\sin\phi,S}_{\phi}(\eta).\nonumber
\end{eqnarray}
Define
\begin{eqnarray}
\beta(\eta)=\br{\sin^2\phi,f}_{\phi}(\eta),
\end{eqnarray}
and
\begin{eqnarray}\label{Milne t 21}
\frac{\ud{\beta}}{\ud{\eta}}=D(\eta,\phi),
\end{eqnarray}
where
\begin{eqnarray}
D(\eta,\phi)=F(\eta)\br{\cos(2\phi),r}_{\phi}+\br{\sin\phi,S}_{\phi}(\eta).
\end{eqnarray}
Hence, we can integrate (\ref{Milne t 21}) over $[0,\eta]$ to get that
\begin{eqnarray}
\beta(\eta)-\beta(0)=\int_0^{\eta}F(y)\br{\cos(2\phi),r}_{\phi}(y)\ud{y}+\int_0^{\eta}\br{\sin\phi,S}_{\phi}(y)\ud{y}.
\end{eqnarray}
Then the initial data
\begin{eqnarray}
\beta(0)=\br{\sin^2\phi,f}_{\phi}(0)\leq \bigg(\br{
f,f\abs{\sin\phi}}_{\phi}(0)\bigg)^{\frac{1}{2}}\tnm{\sin\phi}^{3/2}\leq
C\bigg(\br{ f,f\abs{\sin\phi}}_{\phi}(0)\bigg)^{\frac{1}{2}}.
\end{eqnarray}
Obviously, we have
\begin{eqnarray}
\br{f, f\abs{\sin\phi}
}_{\phi}(0)=\int_{\sin\phi>0}h^2(\phi)\sin\phi
\ud{\phi}-\int_{\sin\phi<0}\bigg(f(0,\phi)\bigg)^2\sin\phi\ud{\phi}.
\end{eqnarray}
However, based on the definition of $\alpha(\eta)$ and (\ref{Milne temp 36}), we can obtain
\begin{eqnarray}
\int_{\sin\phi>0} h^2(\phi)\sin\phi\ud{\phi}+\int_{\sin\phi<0}
\bigg(f(0,\phi)\bigg)^2\sin\phi\ud{\phi}=2\alpha(0)
&\geq&-C\int_0^L\br{S,r}_{\phi}(y)\ud{y}\nonumber.
\end{eqnarray}
Hence, we can deduce
\begin{eqnarray}
-\int_{\sin\phi<0}
\bigg(f(0,\phi)\bigg)^2\sin\phi\ud{\phi}&\leq&\int_{\sin\phi>0}
h^2(\phi)\sin\phi\ud{\phi}+C\int_{0}^L\br{S,r}_{\phi}(y)\ud{y}\\
&\leq&
C\bigg(\tss{h}{-}^2+\int_{0}^L\tnm{S(\eta)}^2\ud{\eta}\bigg)\nonumber.
\end{eqnarray}
From (\ref{Milne temp 37}), we can deduce
\begin{eqnarray}\label{Milne t 05}
\beta(0)&\leq& C\bigg(\tss{h}{-}+\tnnm{S}\bigg).
\end{eqnarray}
Since $F\in L^1[0,L]\cap
L^2[0,L]$, $r\in L^2([0,L]\times[-\pi,\pi))$, by (\ref{Milne t 05}) and (\ref{Milne temp
8}), we have
\begin{eqnarray}
\abs{\beta(L)}&\leq&\abs{\beta(0)}+\abs{\int_0^{L}F(y)\br{\cos(2\phi),r}_{\phi}(y)\ud{y}}
+\abs{\int_0^{L}\br{\sin\phi,S}_{\phi}(y)\ud{y}}\\
&\leq&C\bigg(\tss{h}{-}+\tnnm{S}\bigg)+C\tnnm{F}
\tnnm{r}+\abs{\int_0^{L}\br{\sin\phi,S}_{\phi}(y)\ud{y}}\no\\
&\leq&C\bigg(\tss{h}{-}+\tnnm{S}\bigg)+\abs{\int_0^{L}\br{\sin\phi,S}_{\phi}(y)\ud{y}}.\no
\end{eqnarray}
We define
\begin{eqnarray}
q_L=\frac{\beta(L)}{\tnm{\sin\phi}^2}.
\end{eqnarray}
Naturally, we have
\begin{eqnarray}
\abs{q_L}&\leq&C\bigg(\tss{h}{-}+\tnnm{S}\bigg)+C\abs{\int_0^{L}\br{\sin\phi,S}_{\phi}(y)\ud{y}}.
\end{eqnarray}
Note that $q_L$ is not necessarily $q(L)$. Moreover,
\begin{eqnarray}
\beta(L)-\beta(\eta)=\int_{\eta}^{L}D(y)\ud{y}=\int_{\eta}^{L}F(y)\br{\cos(2\phi),r}_{\phi}(y)\ud{y}+
\int_{\eta}^{L}\br{\sin\phi,S}_{\phi}(y)\ud{y}.
\end{eqnarray}
Note
\begin{eqnarray}
\beta(\eta)=\br{\sin^2\phi,f}_{\phi}(\eta)=\br{\sin^2\phi,q}_{\phi}(\eta)+\br{\sin^2\phi,r}_{\phi}(\eta)
=q(\eta)\tnm{\sin\phi}^2+\br{\sin^2\phi,r}_{\phi}(\eta).
\end{eqnarray}
Thus we can estimate
\begin{eqnarray}\label{Milne t 22}
&&\tnm{\sin\phi}^2\tnm{q(\eta)-q_L}\\
&=&\beta(L)-\beta(\eta)+\br{\sin^2\phi,r}_{\phi}(\eta)\no\\
&\leq&
C\bigg(\int_{\eta}^{L}\abs{F(y)\br{\cos(2\phi),r(y)}_{\phi}\ud{y}}\ud{\eta}
+\abs{\int_{\eta}^{L}\br{\sin\phi,S}_{\phi}(y)\ud{y}}+\abs{\br{\sin^2\phi,r}_{\phi}(\eta)}\bigg)\nonumber\\
&\leq&C\bigg(\tnm{r(\eta)}+\int_{\eta}^{L}\abs{F(y)}\tnm{r(y)}\ud{y}+\abs{\int_{\eta}^{L}\br{\sin\phi,S}_{\phi}(y)\ud{y}}\bigg)\nonumber.
\end{eqnarray}
Then we
integrate (\ref{Milne t 22}) over $\eta\in[0,L]$.
Cauchy's inequality implies
\begin{eqnarray}
&&\int_0^{L}\bigg(\int_{\eta}^{L}\abs{F(y)}\tnm{r(y)}\ud{y}\bigg)^2\ud{\eta}\leq
\tnnm{r}^2\int_0^{L}\int_{\eta}^{L}\abs{F(y)}^2\ud{y}\ud{\eta}\leq
C\tnnm{r}^2.
\end{eqnarray}
Hence, we have
\begin{eqnarray}
\tnnm{q-q_L}\leq C\bigg(\tss{h}{-}+\tnnm{S}\bigg)+C\Bigg(\int_0^{L}\bigg(\int_{\eta}^{L}\br{\sin\phi,S}_{\phi}(y)\ud{y}\bigg)^2\ud{\eta}\Bigg)^{\frac{1}{2}}.
\end{eqnarray}

\end{proof}

\subsubsection{$\bar S\neq0$ Case}

For general $S$, we define $S=\bar S+(S-\bar S)=S_Q+S_R$.
\begin{lemma}\label{Milne infinite LT general}
The unique solution $f(\eta,\phi)$ to the equation
(\ref{Milne problem}) satisfies
\begin{eqnarray}
&&\tnnm{r}\leq C\bigg(\tss{h}{-}+\tnnm{S}\bigg)+C\Bigg(\int_0^{L}\bigg(\int_{\eta}^{L}\abs{S_Q(y)}\ud{y}\bigg)^2\ud{\eta}\Bigg)^{\frac{1}{2}}\label{Milne temp 39},
\end{eqnarray}
and there exists $q_L\in\r$ such that
\begin{eqnarray}
\abs{q_L}&\leq&C\bigg(\tss{h}{-}+\tnnm{S}\bigg)\label{Milne temp 41}\\
&&+C\abs{\int_0^{L}\br{\sin\phi,S_R}_{\phi}(y)\ud{y}}+C\abs{\int_0^{L}\int_{\eta}^{L}\abs{S_Q(y)}\ud{y}\ud{\eta}}\no,\\
\tnnm{q-q_L}&\leq&C\bigg(\tss{h}{-}+\tnnm{S}\bigg)\label{Milne temp 43}\\
&&+C\Bigg(\int_0^{L}\bigg(\int_{\eta}^{L}\br{\sin\phi,S_R}_{\phi}(y)\ud{y}\bigg)^2\ud{\eta}\Bigg)^{\frac{1}{2}}
+C\Bigg(\int_0^{L}\bigg(\int_{\eta}^{L}\int_{y}^{L}\abs{S_Q(z)}\ud{z}\ud{y}\bigg)^2\ud{\eta}\Bigg)^{\frac{1}{2}}.\no
\end{eqnarray}
Also, for any $\eta\in[0,L]$,
\begin{eqnarray}
&&\br{\sin\phi,r}_{\phi}(\eta)=-\int_{\eta}^{L}\ue^{V(\eta)-V(y)}S_Q(y)\ud{y}\label{Milne temp 40}.
\end{eqnarray}
\end{lemma}
\begin{proof}
We can apply superposition property for this linear problem. For simplicity, we just above estimates as the $L^2$ estimates. \\
\ \\
Step 1: Construction of auxiliary function $f^1$.\\
We first solve $f^1$ as the solution to
\begin{eqnarray}
\left\{
\begin{array}{l}\displaystyle
\sin\phi\frac{\p f^1}{\p\eta}+F(\eta)\cos\phi\frac{\p
f^1}{\p\phi}+f^1-\bar f^1=S_R(\eta,\phi),\\\rule{0ex}{1.5em}
f^1(0,\phi)=h(\phi)\ \ \text{for}\ \ \sin\phi>0,\\\rule{0ex}{1.5em}
f^1(L,\phi)=f^1(L,\rr[\phi]).
\end{array}
\right.
\end{eqnarray}
Since $\bar S_R=0$, by Lemma \ref{Milne finite LT}, we know there
exists a unique solution $f^1$
satisfying the $L^2$ estimate.\\
\ \\
Step 2: Construction of auxiliary function $f^2$.\\
We seek a function $f^{2}$ satisfying
\begin{eqnarray}\label{Milne temp 83}
-\frac{1}{2\pi}\int_{-\pi}^{\pi}\bigg(\sin\phi\frac{\p
f^{2}}{\p\eta}+F(\eta)\cos\phi\frac{\p
f^{2}}{\p\phi}\bigg)\ud{\phi}+S_Q=0.
\end{eqnarray}
The following analysis shows this type of function can always be
found. An integration by parts transforms the equation (\ref{Milne temp 83}) into
\begin{eqnarray}\label{Milne t 07}
-\int_{-\pi}^{\pi}\sin\phi\frac{\p
f^{2}}{\p\eta}\ud{\phi}-\int_{-\pi}^{\pi}F(\eta)\sin\phi
f^{2}\ud{\phi}+2\pi S_Q=0.
\end{eqnarray}
Setting
\begin{eqnarray}
f^{2}(\phi,\eta)=a(\eta)\sin\phi.
\end{eqnarray}
and plugging this ansatz into (\ref{Milne t 07}), we have
\begin{eqnarray}
-\frac{\ud{a}}{\ud{\eta}}\int_{-\pi}^{\pi}\sin^2\phi\ud{\phi}-F(\eta)a(\eta)\int_{-\pi}^{\pi}\sin^2\phi\ud{\phi}+2\pi
S_Q=0.
\end{eqnarray}
Hence, we have
\begin{eqnarray}
-\frac{\ud{a}}{\ud{\eta}}-F(\eta)a(\eta)+2S_Q=0.
\end{eqnarray}
This is a first order linear ordinary differential equation, which
possesses infinite solutions. We can directly solve it to obtain
\begin{eqnarray}
a(\eta)=\exp\bigg(-\int_0^{\eta}F(y)\ud{y}\bigg)\bigg(a(0)+\int_0^{\eta}\exp\bigg(\int_0^yF(z)\ud{z}\bigg)2S_Q(y)\ud{y}\bigg).
\end{eqnarray}
We may take
\begin{eqnarray}
a(0)=-\int_0^{L}\exp\bigg(\int_0^yF(z)\ud{z}\bigg)2S_Q(y)\ud{y}.
\end{eqnarray}
Then, we can directly verify
\begin{eqnarray}
\abs{a(\eta)}\leq C\int_{\eta}^L\abs{S_Q(y)}\ud{y},
\end{eqnarray}
and $f^2$
satisfies the $L^2$ estimate.\\
\ \\
Step 3: Construction of auxiliary function $f^3$.\\
Based on above construction, we can directly verify that
\begin{eqnarray}\label{Milne temp 84}
\int_{-\pi}^{\pi}\bigg(-\sin\phi\frac{\p
f^{2}}{\p\eta}-F(\eta)\cos\phi\frac{\p f^{2}}{\p\phi}-f^{2}+\bar
f^{2}+S_Q\bigg)\ud{\phi}=0.
\end{eqnarray}
Then we can solve $f^3$ as the solution to
\begin{eqnarray}
\left\{
\begin{array}{l}\displaystyle
\sin\phi\frac{\p f^3}{\p\eta}+F(\eta)\cos\phi\frac{\p
f^3}{\p\phi}+f^3-\bar f^3=-\sin\phi\dfrac{\p
f^{2}}{\p\eta}-F(\eta)\cos\phi\dfrac{\p
f^{2}}{\p\phi}-f^{2}+\bar f^{2}+S_Q,\\\rule{0ex}{1.5em}
f^3(0,\phi)=-a(0)\sin\phi\ \ \text{for}\ \ \sin\phi>0,\\\rule{0ex}{1.5em}
f^3(L,\phi)=f^3(L,\rr[\phi]).
\end{array}
\right.
\end{eqnarray}
By (\ref{Milne temp 84}), we can apply Lemma \ref{Milne finite LT}
to obtain a unique solution $f^3$
satisfying the $L^2$ estimate.\\
\ \\
Step 4: Construction of auxiliary function $f^4$.\\
We now define $f^4=f^2+f^3$ and an explicit verification shows
\begin{eqnarray}
\left\{
\begin{array}{l}\displaystyle
\sin\phi\frac{\p f^4}{\p\eta}+F(\eta)\cos\phi\frac{\p
f^4}{\p\phi}+f^4-\bar f^4=S_Q(\eta,\phi),\\\rule{0ex}{1.5em}
f^4(0,\phi)=0\ \ \text{for}\ \ \sin\phi>0,\\\rule{0ex}{1.5em}
f^4(L,\phi)=f^4(L,\rr[\phi]),
\end{array}
\right.
\end{eqnarray}
and $f^4$
satisfies the $L^2$ estimate.\\
\ \\
In summary, we deduce that $f^1+f^4$ is the solution of (\ref{Milne
problem}) and satisfies the $L^2$ estimate.
\end{proof}
Combining all above, we have the following theorem.
\begin{theorem}\label{Milne lemma 6}
The unique solution $f(\eta,\phi)$ to the equation
(\ref{Milne problem}) satisfies
\begin{eqnarray}
\tnnm{f-f_L}&\leq& C\bigg(\tss{h}{-}+\tnnm{S}\bigg)+C\Bigg(\int_0^{L}\bigg(\int_{\eta}^{L}\br{\sin\phi,S_R}_{\phi}(y)\ud{y}\bigg)^2\ud{\eta}\Bigg)^{\frac{1}{2}}\\
&&+C\Bigg(\int_0^{L}\bigg(\int_{\eta}^{L}\abs{S_Q(y)}\ud{y}\bigg)^2\ud{\eta}\Bigg)^{\frac{1}{2}}
+C\Bigg(\int_0^{L}\bigg(\int_{\eta}^{L}\int_{y}^{L}\abs{S_Q(z)}\ud{z}\ud{y}\bigg)^2\ud{\eta}\Bigg)^{\frac{1}{2}},\no
\end{eqnarray}
for some $f_L\in\r$ satisfying
\begin{eqnarray}
\abs{f_L}&\leq&C\bigg(\tss{h}{-}+\tnnm{S}\bigg)+C\abs{\int_0^{L}\br{\sin\phi,S_R}_{\phi}(y)\ud{y}}+C\abs{\int_0^{L}\int_{\eta}^{L}\abs{S_Q(y)}\ud{y}\ud{\eta}}.
\end{eqnarray}
\end{theorem}

\subsection{$L^{\infty}$ Estimates}

\subsubsection{Formulation}

Consider the $\e$-transport problem for $f(\eta,\phi)$ in $(\eta,\phi)\in[0,L]\times[-\pi,\pi)$
\begin{eqnarray}\label{Milne finite problem LI}
\left\{
\begin{array}{l}\displaystyle
\sin\phi\frac{\p f}{\p\eta}+F(\eta)\cos\phi\frac{\p
f}{\p\phi}+f=H(\eta,\phi),\\\rule{0ex}{1.5em}
f(0,\phi)=h(\phi)\ \ \text{for}\ \ \sin\phi>0,\\\rule{0ex}{1.5em}
f(L,\phi)=f(L,\rr[\phi]).
\end{array}
\right.
\end{eqnarray}
Define the energy as follows:
\begin{eqnarray}
E(\eta,\phi)=\ue^{-V(\eta)}\cos\phi.
\end{eqnarray}
Along the characteristics, this energy is conserved and the equation can be simplified as follows:
\begin{eqnarray}
\sin\phi\frac{\ud{f}}{\ud{\eta}}+f=H.
\end{eqnarray}
An implicit function
$\eta^+(\eta,\phi)$ can be determined through
\begin{eqnarray}
\abs{E(\eta,\phi)}=\ue^{-V(\eta^+)}.
\end{eqnarray}
which means $(\eta^+,\phi_0)$ with $\sin\phi_0=0$ is on the same characteristics as $(\eta,\phi)$.
Define the quantities for $0\leq\eta'\leq\eta^+$ as follows:
\begin{eqnarray}
\phi'(\eta,\phi;\eta')&=&\cos^{-1}\Big(\ue^{V(\eta')-V(\eta)}\cos\phi\Big),\\
\rr[\phi'(\eta,\phi;\eta')]&=&-\cos^{-1}\Big(\ue^{V(\eta')-V(\eta)}\cos\phi\Big)=-\phi'(\eta,\phi;\eta'),
\end{eqnarray}
where the inverse trigonometric function can be defined
single-valued in the domain $[0,\pi)$ and the quantities are always well-defined due to the monotonicity of $V$.
Finally we put
\begin{eqnarray}
G_{\eta,\eta'}(\phi)&=&\int_{\eta'}^{\eta}\frac{1}{\sin\Big(\phi'(\eta,\phi;\xi)\Big)}\ud{\xi}.
\end{eqnarray}
We can rewrite the solution to the equation (\ref{Milne finite problem LI}) along
the characteristics as
\begin{eqnarray}
f(\eta,\phi)=\k[h](\phi)+\t[H](\eta,\phi),
\end{eqnarray}
where\\
\ \\
Region I:\\
For $\sin\phi>0$,
\begin{eqnarray}\label{Milne t 08}
\k[h](\phi)&=&h\Big(\phi'(\eta,\phi;0)\Big)\exp(-G_{\eta,0}),\\
\t[H](\eta,\phi)&=&\int_0^{\eta}\frac{H\Big(\eta',\phi'(\eta,\phi;\eta')\Big)}{\sin\Big(\phi'(\eta,\phi;\eta')\Big)}\exp(-G_{\eta,\eta'})\ud{\eta'}.
\end{eqnarray}
\ \\
Region II:\\
For $\sin\phi<0$ and $\abs{E(\eta,\phi)}\leq \ue^{-V(L)}$,
\begin{eqnarray}\label{Milne t 09}
\k[h](\phi)&=&h(\phi'(\eta,\phi;0))\exp(-G_{L,0}-G_{L,\eta})\\
\t[H](\eta,\phi)&=&\int_0^{L}\frac{H\Big(\eta',\phi'(\eta,\phi;\eta')\Big)}{\sin\Big(\phi'(\eta,\phi;\eta')\Big)}
\exp(-G_{L,\eta'}-G_{L,\eta})\ud{\eta'}\\
&&+\int_{\eta}^{L}\frac{H\Big(\eta',\rr[\phi'(\eta,\phi;\eta')]\Big)}{\sin\Big(\phi'(\eta,\phi;\eta')\Big)}\exp(G_{\eta,\eta'})\ud{\eta'}\nonumber.
\end{eqnarray}
\ \\
Region III:\\
For $\sin\phi<0$ and $\abs{E(\eta,\phi)}\geq \ue^{-V(L)}$,
\begin{eqnarray}\label{Milne t 10}
\k[h](\phi)&=&h\Big(\phi'(\eta,\phi;0)\Big)\exp(-G_{\eta^+,0}-G_{\eta^+,\eta})\\
\t[H](\eta,\phi)&=&\int_0^{\eta^+}\frac{H\Big(\eta',\phi'(\eta,\phi;\eta')\Big)}{\sin\Big(\phi'(\eta,\phi;\eta')\Big)}
\exp(-G_{\eta^+,\eta'}-G_{\eta^+,\eta})\ud{\eta'}\\
&&+\int_{\eta}^{\eta^+}\frac{H\Big(\eta',\rr[\phi'(\eta,\phi;\eta')]\Big)}{\sin\Big(\phi'(\eta,\phi;\eta')\Big)}\exp(G_{\eta,\eta'})\ud{\eta'}\nonumber.
\end{eqnarray}
In order to achieve the estimate
of $f$, we need to control $\k[h]$ and $\t[H]$.

\subsubsection{Preliminaries}

We first give several technical lemmas to be used for proving
$L^{\infty}$ estimates of $f$.
\begin{lemma}\label{Milne lemma 1}
For any $0\leq\beta\leq1$, we have
\begin{eqnarray}\label{Milne temp 51}
\lnm{\ue^{\beta\eta}\k[h]}\leq \lnm{h}.
\end{eqnarray}
In particular,
\begin{eqnarray}
\lnm{\k[h]}\leq \lnm{h}\label{Milne temp 52}.
\end{eqnarray}
\end{lemma}
\begin{proof}
Since $\phi'$ is always in the domain $[0,\pi)$, we naturally have
\begin{eqnarray}
0\leq\sin\Big(\phi'(\eta,\phi;\xi)\Big)\leq 1,
\end{eqnarray}
which further implies
\begin{eqnarray}
\frac{1}{\sin\Big(\phi'(\eta,\phi;\xi)\Big)}\geq 1.
\end{eqnarray}
Combined with the fact $L\geq\eta^+\geq\eta$, this implies
\begin{eqnarray}
\exp(-G_{\eta,0})&\leq&\ue^{-\eta},\\
\exp(-G_{L,0}-G_{L,\eta})&\leq&\ue^{-\eta},\\
\exp(-G_{\eta^+,0}-G_{\eta^+,\eta})&\leq&\exp(-G_{\eta^+,0})\leq\exp(-G_{\eta,0})\leq
\ue^{-\eta}.
\end{eqnarray}
Hence, our result easily follows.
\end{proof}
\begin{lemma}\label{Milne lemma 2}
The integral operator $\t$ satisfies
\begin{eqnarray}\label{Milne temp 53}
\lnnm{\t[H]}\leq \lnnm{H},
\end{eqnarray}
and for any $0\leq\beta\leq\dfrac{1}{2}$
\begin{eqnarray}\label{Milne temp 54}
\lnnm{\ue^{\beta\eta}\t[H]}\leq \lnnm{\ue^{\beta\eta}H}.
\end{eqnarray}
\end{lemma}
\begin{proof}
For (\ref{Milne temp 53}), when $\sin\phi>0$
\begin{eqnarray}
\abs{\t[H]}&\leq&\int_0^{\eta}\abs{H\Big(\eta',\phi'(\eta,\phi;\eta')\Big)}\frac{1}{\sin\Big(\phi'(\eta,\phi;\eta')\Big)}\exp(-G_{\eta,\eta'})\ud{\eta'}\\
&\leq&\lnnm{H}\int_0^{\eta}\frac{1}{\sin\Big(\phi'(\eta,\phi;\eta')\Big)}\exp(-G_{\eta,\eta'})\ud{\eta'}\nonumber.
\end{eqnarray}
We can directly estimate
\begin{eqnarray}\label{Milne t 11}
\int_0^{\eta}\frac{1}{\sin\Big(\phi'(\eta,\phi;\eta')\Big)}\exp(-G_{\eta,\eta'})\ud{\eta'}\leq\int_0^{\infty}\ue^{-z}\ud{z}=1,
\end{eqnarray}
and (\ref{Milne temp 53}) naturally follows. For $\sin\phi<0$ and
$\abs{E(\eta,\phi)}\leq \ue^{-V(L)}$,
\begin{eqnarray}
\abs{\t[H]}&\leq&\int_{\eta}^{\infty}\abs{H\Big(\eta',\phi'(\eta,\phi;\eta')\Big)}\frac{1}{\sin\Big(\phi'(\eta,\phi;\eta')\Big)}\exp(G_{\eta,\eta'})\ud{\eta'}\\
&\leq&\lnnm{H}\int_{\eta}^{\infty}\frac{1}{\sin\Big(\phi'(\eta,\phi;\eta')\Big)}\exp(G_{\eta,\eta'})\ud{\eta'}\nonumber.
\end{eqnarray}
we have
\begin{eqnarray}
\int_{\eta}^{\infty}\frac{1}{\sin\Big(\phi'(\eta,\phi;\eta')\Big)}\exp(G_{\eta,\eta'})\ud{\eta'}\leq\int_{-\infty}^0\ue^{z}\ud{z}=1,
\end{eqnarray}
and (\ref{Milne temp 53}) easily follows. The region $\sin\phi<0$ and
$\abs{E(\eta,\phi)}\geq \ue^{-V(L)}$ can be proved combining
above two
techniques, so we omit it here.

For (\ref{Milne temp 54}), when $\sin\phi>0$, $\eta\geq\eta'$ and
$\beta<\dfrac{1}{2}$, since $G_{\eta,\eta'}\geq\eta-\eta'$, we have
\begin{eqnarray}
\beta(\eta-\eta')-G_{\eta,\eta'}\leq
\beta(\eta-\eta')-\half(\eta-\eta')-\half G_{\eta,\eta'}\leq -\half
G_{\eta,\eta'}.
\end{eqnarray}
Then it is natural that
\begin{eqnarray}
\int_0^{\eta}\frac{1}{\sin\Big(\phi'(\eta,\phi;\eta')\Big)}\exp\Big(\beta(\eta-\eta')-G_{\eta,\eta'}\Big)\ud{\eta'}
&\leq&\int_0^{\eta}\frac{1}{\sin\Big(\phi'(\eta,\phi;\eta')\Big)}\exp\left(-\frac{G_{\eta,\eta'}}{2}\right)\ud{\eta'}\\
&\leq&\int_0^{\infty}\ue^{-\frac{z}{2}}\ud{z}=2\nonumber.
\end{eqnarray}
This leads to
\begin{eqnarray}
\abs{\ue^{\beta\eta}\t[H]}&\leq&
\ue^{\beta\eta}\int_0^{\eta}\abs{H\Big(\eta',\phi'(\eta,\phi;\eta')\Big)}\frac{1}{\sin\Big(\phi'(\eta,\phi;\eta')\Big)}\exp(-G_{\eta,\eta'})\ud{\eta'}\\
&\leq&\lnnm{\ue^{\beta\eta}H}\int_0^{\eta}\frac{1}{\sin\Big(\phi'(\eta,\phi;\eta')\Big)}\exp\Big(\beta(\eta-\eta')-G_{\eta,\eta'}\Big)\ud{\eta'}\nonumber\\
&\leq&C\lnnm{\ue^{\beta\eta}H},\nonumber
\end{eqnarray}
and (\ref{Milne temp 54}) naturally follows. For $\sin\phi<0$ and
$\abs{E(\eta,\phi)}\leq \ue^{-V(L)}$, note that $-G_{L,\eta'}-G_{L,\eta}\leq -G_{\eta,\eta'}$ and for $\eta'\geq\eta$
\begin{eqnarray}
\beta(\eta-\eta')+G_{\eta,\eta'}\leq
\beta(\eta-\eta')+\half(\eta-\eta')+\half G_{\eta,\eta'}\leq \half
G_{\eta,\eta'}.
\end{eqnarray}
Then (\ref{Milne temp 54}) holds by obvious modifications of
$\sin\phi>0$ region. The case $\sin\phi<0$ and $\abs{E(\eta,\phi)}\geq
\ue^{-V(L)}$ can be shown by combining above two regions, so we
omit it here.
\end{proof}
\begin{lemma}\label{Milne lemma 3}
For any $\delta>0$ there is a constant $C(\delta)>0$ independent of
data such that
\begin{eqnarray}
\ltnm{\t[H]}\leq C(\delta)\tnnm{H}+\delta\lnnm{H}\label{Milne temp
11}.
\end{eqnarray}
\end{lemma}
\begin{proof}
In the following, we use $\chi_i$ to represent certain indicator functions. Also, we let $m>0$ and $\sigma>0$ be some constants that are determined later.\\
\ \\
Region I: $\sin\phi>0$.\\
We have
\begin{eqnarray}
\t[H](\eta,\phi)&=&\int_0^{\eta}\frac{H\Big(\eta',\phi'(\eta,\phi;\eta')\Big)}{\sin\Big(\phi'(\eta,\phi;\eta')\Big)}\exp(-G_{\eta,\eta'})\ud{\eta'}.
\end{eqnarray}
We consider
\begin{eqnarray}
I&=&\int_{\sin\phi>0}\abs{\t
[H](\eta,\phi)}^2\ud{\phi}=\int_{\sin\phi>0}\bigg(\int_0^{\eta}\frac{H\Big(\eta',\phi'(\eta,\phi;\eta')\Big)}{\sin\Big(\phi'(\eta,\phi;\eta')\Big)}
\exp(-G_{\eta,\eta'})\ud{\eta'}\bigg)^2\ud{\phi}\\
&=&I_1+I_2\nonumber.
\end{eqnarray}
Region I - Case I: $\chi_1: \sin(\phi'(\eta,\phi;\eta'))\geq m$.\\
By Cauchy's inequality and (\ref{Milne t 11}), we get
\begin{eqnarray}\label{Milne temp 55}
I_1&\leq&\int_{\sin\phi>0}\bigg(\int_0^{\eta}\abs{H\Big(\eta',\phi'(\eta,\phi;\eta')\Big)}^2\ud{\eta'}\bigg)
\bigg(\int_0^{\eta}\chi_1\frac{\exp(-2G_{\eta,\eta'})}{\sin^2\Big(\phi'(\eta,\phi;\eta')\Big)}
\ud{\eta'}\bigg)\ud{\phi}\\
&\leq&\frac{1}{m}\int_{\sin\phi>0}\bigg(\int_0^{\eta}\abs{H\Big(\eta',\phi'(\eta,\phi;\eta')\Big)}^2\ud{\eta'}\bigg)
\bigg(\int_0^{\eta}\chi_1\frac{\exp(-2G_{\eta,\eta'})}{\sin\Big(\phi'(\eta,\phi;\eta')\Big)}
\ud{\eta'}\bigg)\ud{\phi}\no\\
&\leq&\frac{1}{m}\tnnm{H}^2\Bigg(\int_{\sin\phi>0}
\bigg(\int_0^{\eta}\chi_1\frac{\exp(-2G_{\eta,\eta'})}{\sin\Big(\phi'(\eta,\phi;\eta')\Big)}
\ud{\eta'}\bigg)^2\ud{\phi}\Bigg)^{\frac{1}{2}}\nonumber\\
&\leq&\frac{C}{m}\tnnm{H}^2\nonumber.
\end{eqnarray}
due to
\begin{eqnarray}
\int_0^{\eta}
\frac{1}{\sin\Big(\phi'(\eta,\phi;\eta')\Big)}\exp(-2G_{\eta,\eta'})\ud{\eta'}\leq\int_{0}^{\infty}\ue^{-2z}\ud{z}=\frac{1}{2}.
\end{eqnarray}
\ \\
Region I - Case II: $\chi_2: \sin\Big(\phi'(\eta,\phi;\eta')\Big)\leq m$.\\
For $\eta'\leq\eta$, we can directly estimate
$\phi'(\eta,\phi;\eta')\geq\phi$. Hence, we have the relation
\begin{eqnarray}
\sin\phi\leq\sin\Big(\phi'(\eta,\phi;\eta')\Big).
\end{eqnarray}
Therefore, we can directly estimate $I_2$ as follows:
\begin{eqnarray}\label{Milne temp 56}
I_2&\leq&\lnnm{H}^2\int_{\sin\phi>0}
\bigg(\int_0^{\eta}\chi_2
\frac{1}{\sin\Big(\phi'(\eta,\phi;\eta')\Big)}\exp(-G_{\eta,\eta'})\ud{\eta'}\bigg)^2\ud{\phi}\\
&\leq&\lnnm{H}^2\int_{\sin\phi>0}\chi_2\ud{\phi}\no\\
&\leq&Cm\lnnm{H}^2,\no
\end{eqnarray}
due to
\begin{eqnarray}
\int_0^{\eta}
\frac{1}{\sin\Big(\phi'(\eta,\phi;\eta')\Big)}\exp(-G_{\eta,\eta'})\ud{\eta'}\leq\int_{0}^{\infty}\ue^{-z}\ud{z}=1.
\end{eqnarray}
Summing up (\ref{Milne temp 55}) and (\ref{Milne temp 56}), for $m$
sufficiently small, we deduce (\ref{Milne temp 11}).\\
\ \\
Region II: $\sin\phi<0$ and $\abs{E(\eta,\phi)}\leq \ue^{-V(L)}$.\\
We have
\begin{eqnarray}
\t[H](\eta,\phi)&=&\int_0^{L}\frac{H\Big(\eta',\phi'(\eta,\phi;\eta')\Big)}{\sin\Big(\phi'(\eta,\phi;\eta')\Big)}
\exp(-G_{L,\eta'}-G_{L,\eta})\ud{\eta'}\\
&&+\int_{\eta}^{L}\frac{H\Big(\eta',\rr[\phi'(\eta,\phi;\eta')]\Big)}{\sin\Big(\phi'(\eta,\phi;\eta')\Big)}\exp(G_{\eta,\eta'})\ud{\eta'}\nonumber.
\end{eqnarray}
Since $-G_{L,\eta'}-G_{L,\eta}\leq -G_{\eta,\eta'}$, it suffices to estimate
\begin{eqnarray}
II&=&\int_{\sin\phi<0}{\bf{1}}_{\{\abs{E(\eta,\phi)}\leq
\ue^{-V(L)}\}}\bigg(\int_{\eta}^{L}
\frac{H\Big(\eta',\rr[\phi'(\eta,\phi;\eta')]\Big)}{\sin\Big(\phi'(\eta,\phi;\eta')\Big)}
\exp(G_{\eta,\eta'})\ud{\eta'}\bigg)^2\ud{\phi}\no\\
&=&II_1+II_2+II_3\nonumber.
\end{eqnarray}
\ \\
Region II - Case I: $\chi_1: \sin(\phi'(\eta,\phi;\eta'))>m$.\\
We can directly estimate $II_1$ as follows:
\begin{eqnarray}\label{Milne temp 57}
II_1&\leq&\int_{\sin\phi<0}\bigg(\int_{\eta}^{L}\abs{H\Big(\eta',\rr[\phi'(\eta,\phi;\eta')]\Big)}^2\ud{\eta'}\bigg)
\bigg(\int_{\eta}^{L}\chi_1\frac{\exp(2G_{\eta,\eta'})}{\sin^2\Big(\phi'(\eta,\phi;\eta')\Big)}
\ud{\eta'}\bigg)\ud{\phi}\\
&\leq&\frac{1}{m}\int_{\sin\phi<0}\bigg(\int_{\eta}^{L}\abs{H\Big(\eta',\rr[\phi'(\eta,\phi;\eta')]\Big)}^2\ud{\eta'}\bigg)
\bigg(\int_{\eta}^{L}\chi_1\frac{\exp(2G_{\eta,\eta'})}{\sin\Big(\phi'(\eta,\phi;\eta')\Big)}
\ud{\eta'}\bigg)\ud{\phi}\no\\
&\leq&\frac{1}{m}\tnnm{H}^2\int_{\sin\phi<0}
\bigg(\int_{\eta}^{L}\chi_1\frac{\exp(2G_{\eta,\eta'})}{\sin\Big(\phi'(\eta,\phi;\eta')\Big)}
\ud{\eta'}\bigg)\ud{\phi}\nonumber\\
&\leq&\frac{C}{m}\tnnm{H}^2\nonumber,
\end{eqnarray}
due to
\begin{eqnarray}
\int_{\eta}^L
\frac{1}{\sin\Big(\phi'(\eta,\phi;\eta')\Big)}\exp(2G_{\eta,\eta'})\ud{\eta'}\leq\int^{0}_{-\infty}\ue^{2z}\ud{z}=\frac{1}{2}.
\end{eqnarray}
\ \\
Region II - Case II: $\chi_2: \sin\Big(\phi'(\eta,\phi;\eta')\Big)>m,\ \eta'-\eta\geq\sigma$.\\
We have
\begin{eqnarray}
II_2&\leq&\lnnm{H}^2\int_{\sin\phi<0}
\bigg(\int_{\eta}^{L}\chi_2
\frac{\exp(G_{\eta,\eta'})}{\sin\Big(\phi'(\eta,\phi;\eta')\Big)}\ud{\eta'}\bigg)^2\ud{\phi}.
\end{eqnarray}
Note
\begin{eqnarray}
G_{\eta,\eta'}=\int_{\eta'}^{\eta}\frac{1}{\sin\Big(\phi'(\eta,\phi;y)\Big)}\ud{y}\leq-\frac{\eta'-\eta}{m}=-\frac{\sigma}{m}.
\end{eqnarray}
Then we can obtain
\begin{eqnarray}\label{Milne temp 58}
II_2&\leq&\lnnm{H}^2\int_{\sin\phi<0}
\bigg(\int^{-\frac{\sigma}{m}}_{-\infty}\ue^z\ud{z}\bigg)^2\ud{\phi}\leq C
\ue^{-\frac{\sigma}{m}}\lnnm{H}^2.
\end{eqnarray}
\ \\
Region II - Case III: $\chi_3: \sin\Big(\phi'(\eta,\phi;\eta')\Big)>m,\ \eta'-\eta\leq\sigma$\\
For $II_3$, we can estimate as follows:
\begin{eqnarray}
I_3&\leq&\lnnm{H}^2\int_{\sin\phi<0}
\bigg(\int_{\eta}^{L}\chi_3
\frac{\exp(G_{\eta,\eta'})}{\sin\Big(\phi'(\eta,\phi;\eta')\Big)}\ud{\eta'}\bigg)^2\ud{\phi}\\
&\leq&\lnnm{H}^2\int_{\sin\phi<0}\chi_3
\bigg(\int_{\eta}^{\eta+\sigma}
\frac{\exp(G_{\eta,\eta'})}{\sin\Big(\phi'(\eta,\phi;\eta')\Big)}\ud{\eta'}\bigg)^2\ud{\phi}\nonumber.
\end{eqnarray}
Note that
\begin{eqnarray}
\int_{\eta}^{L}
\frac{1}{\sin\Big(\phi'(\eta,\phi;\eta')\Big)}\exp(G_{\eta,\eta'})\ud{\eta'}\leq\int_{-\infty}^{0}\ue^{z}\ud{z}=1.
\end{eqnarray}
Then $1\leq\alpha=\ue^{V(\eta')-V(\eta)}\leq
\ue^{V(\eta+\sigma)-V(\eta)}\leq 1+4\sigma$, and for $\eta'\in[\eta,\eta+\sigma]$,
$\sin\Big(\phi'(\eta,\phi;\eta')\Big)=\sin\bigg(\cos^{-1}(\alpha\cos\phi)\bigg)$,
$\sin\Big(\phi'(\eta,\phi;\eta')\Big)<m$ lead to
\begin{eqnarray}
\abs{\sin\phi}&=&\sqrt{1-\cos^2\phi}=\sqrt{1-\frac{\cos^2\Big(\phi'(\eta,\phi;\eta')\Big)}{\alpha^2}}
=\frac{\sqrt{\alpha^2-\bigg(1-\sin^2\Big(\phi'(\eta,\phi;\eta')\Big)\bigg)}}{\alpha}\\
&\leq&\frac{\sqrt{\alpha^2-1+m^2}}{\alpha}\leq\frac{\sqrt{(1+4\sigma)^2-1+m^2}}{\alpha}\leq
\sqrt{9\sigma+m^2}.\nonumber
\end{eqnarray}
Hence, we can obtain
\begin{eqnarray}\label{Milne temp 59}
I_3&\leq&\lnnm{H}^2\int_{\sin\phi<0}\chi_3\ud{\phi}\leq\lnnm{H}^2
\int_{\sin\phi<0}{\bf{1}}_{\{\abs{\sin\phi}\leq \sqrt{9\sigma+m^2}\}}\ud{\phi}\leq C\sqrt{\sigma+m^2}.
\end{eqnarray}
Summarizing (\ref{Milne temp 57}), (\ref{Milne temp 58}) and
(\ref{Milne temp 59}), for sufficiently small $\sigma$,
we can always choose $m<<\sigma$ small enough to guarantee the relation (\ref{Milne temp 11}).\\
\ \\
Region III: $\sin\phi<0$ and $\abs{E(\eta,\phi)}\geq \ue^{-V(L)}$.\\
We have
\begin{eqnarray}
\t[H](\eta,\phi)&=&\int_0^{\eta^+}\frac{H\Big(\eta',\phi'(\eta,\phi;\eta')\Big)}{\sin\Big(\phi'(\eta,\phi;\eta')\Big)}
\exp(-G_{\eta^+,\eta'}-G_{\eta^+,\eta})\ud{\eta'}\\
&&+\int_{\eta}^{\eta^+}\frac{H\Big(\eta',\rr[\phi'(\eta,\phi;\eta')]\Big)}{\sin\Big(\phi'(\eta,\phi;\eta')\Big)}\exp(G_{\eta,\eta'})\ud{\eta'}\nonumber.
\end{eqnarray}
We can decompose $\t[H]$.
For the integral on $[0,\eta]$, we can apply a similar argument as in Region 1 and for
the integral on $[\eta,\eta^+]$, a similar argument as in Region 2 concludes the proof.
\end{proof}

\subsubsection{Estimates of $\e$-Milne Equation with Geometric Correction}

Consider the equation satisfied by $\v=f-f_L$ as follows:
\begin{eqnarray}\label{difference equation}
\left\{
\begin{array}{l}\displaystyle
\sin\phi\frac{\p \v}{\p\eta}+F(\eta)\cos\phi\frac{\p
\v}{\p\phi}+\v=\bar \v+S,\\\rule{0ex}{1.5em}
\v(0,\phi)=p(\phi)=h(\phi)-f_L\ \ \text{for}\ \ \sin\phi>0,\\\rule{0ex}{1.5em}
\v(L,\phi)=\v(L,\rr[\phi]).
\end{array}
\right.
\end{eqnarray}
\begin{theorem}\label{Milne lemma 5}
The unique solution $f(\eta,\phi)$ to the equation
(\ref{Milne problem}) satisfies
\begin{eqnarray}
\lnnm{f-f_L}\leq C\bigg(\abs{f_L}+\lnm{h}+\lnnm{S}+\tnnm{f-f_L}\bigg).
\end{eqnarray}
\end{theorem}
\begin{proof}
We first show the following important facts:
\begin{eqnarray}
\tnnm{\bar \v}&\leq&\tnnm{\v}\label{Milne temp 61},\\
\lnnm{\bar \v}&\leq&\ltnm{\v}\label{Milne temp 62}.
\end{eqnarray}
We can directly derive them by Cauchy's inequality as follows:
\begin{eqnarray}
\\
\tnnm{\bar
\v}^2&=&\int_0^{\infty}\int_{-\pi}^{\pi}\bigg(\frac{1}{2\pi}\bigg)^2\bigg(\int_{-\pi}^{\pi}\v(\eta,\phi)\ud{\phi}\bigg)^2\ud{\phi}\ud{\eta}
\leq\int_0^{\infty}\int_{-\pi}^{\pi}\bigg(\frac{1}{2\pi}\bigg)\bigg(\int_{-\pi}^{\pi}\v^2(\eta,\phi)\ud{\phi}\bigg)\ud{\phi}\ud{\eta}\nonumber\\
&=&\int_0^{\infty}\bigg(\int_{-\pi}^{\pi}\v^2(\eta,\phi)\ud{\phi}\bigg)\ud{\eta}=\tnnm{\v}^2\nonumber.\\
\\
\lnnm{\bar \v}^2&=&\sup_{\eta}\bar
\v^2(\eta)=\sup_{\eta}\bigg(\frac{1}{2\pi}\int_{-\pi}^{\pi}\v(\eta,\phi)\ud{\phi}\bigg)^2
\leq\sup_{\eta}\bigg(\frac{1}{2\pi}\bigg)^2\bigg(\int_{-\pi}^{\pi}\v^2(\eta,\phi)\ud{\phi}\bigg)\bigg(\int_{-\pi}^{\pi}1^2\ud{\phi}\bigg)\nonumber\\
&=&\sup_{\eta}\bigg(\int_{-\pi}^{\pi}\v^2(\eta,\phi)\ud{\phi}\bigg)=\ltnm{\v}^2\nonumber.
\end{eqnarray}
By (\ref{difference equation}), $\v=\k[p]+\t[\bar \v]+\t[S]$ leads to
\begin{eqnarray}
\t[\bar \v]=\v-\k[p]-\t[S],
\end{eqnarray}
Then by Lemma \ref{Milne lemma 3}, (\ref{Milne temp 61}) and
(\ref{Milne temp 62}), we can show
\begin{eqnarray}\label{Milne temp 63}
\ltnm{\v-\k[p]-\t[S]}&\leq& C(\delta)\tnnm{\bar \v}+\delta\lnnm{\bar \v}\leq
C(\delta)\tnnm{\v}+\delta\ltnm{\v}.
\end{eqnarray}
Therefore, based on Lemma \ref{Milne lemma 1}, Lemma \ref{Milne lemma 2} and (\ref{Milne temp
63}), we can directly estimate
\begin{eqnarray}
\ltnm{\v}&\leq&\lnm{\k[p]}+\ltnm{\t[S]}+C(\delta)\tnnm{\v}+\delta\ltnm{\v}\\
&\leq&\lnm{p}+\lnnm{S}+C(\delta)\tnnm{\v}+\delta\ltnm{\v}\nonumber.
\end{eqnarray}
We can take $\delta=\dfrac{1}{2}$ to obtain
\begin{eqnarray}\label{Milne temp 64}
\ltnm{\v}\leq C\bigg(\tnnm{\v}+\lnm{p}+\lnnm{S}\bigg).
\end{eqnarray}
Therefore, based on Lemma \ref{Milne lemma 2}, (\ref{Milne temp 64})
and (\ref{Milne temp 62}), we can achieve
\begin{eqnarray}
\lnnm{\v}&\leq&\lnnm{\k[p]}+\lnnm{\t[S]}+\lnnm{\t[\bar \v]}\\
&\leq&
C\bigg(\lnm{p}+\lnnm{S}+\lnnm{\bar \v}\bigg)\no\\
&\leq&C\bigg(\lnm{p}+\lnnm{S}+\ltnm{\v}\bigg)\no\\
&\leq&
C\bigg(\lnm{p}+\lnnm{S}+\tnnm{\v}\bigg)\nonumber.
\end{eqnarray}
\end{proof}

Combining Theorem \ref{Milne lemma 5} and Theorem \ref{Milne lemma
6}, we deduce the main theorem.
\begin{theorem}\label{Milne theorem 1}
The unique solution $f(\eta,\phi)$ to the equation
(\ref{Milne problem}) satisfies
\begin{eqnarray}
\lnnm{f-f_L}&\leq& C\bigg(\tss{h}{-}+\tnnm{S}+\lnm{h}+\lnnm{S}\bigg)\\
&&+C\Bigg(\int_0^{L}\bigg(\int_{\eta}^{L}\br{\sin\phi,S_R}_{\phi}(y)\ud{y}\bigg)^2\ud{\eta}\Bigg)^{\frac{1}{2}}
+C\abs{\int_0^{L}\br{\sin\phi,S_R}_{\phi}(y)\ud{y}}\no\\
&&+C\abs{\int_0^{L}\int_{\eta}^{L}\abs{S_Q(y)}\ud{y}\ud{\eta}}+C\Bigg(\int_0^{L}\bigg(\int_{\eta}^{L}\abs{S_Q(y)}\ud{y}\bigg)^2\ud{\eta}\Bigg)^{\frac{1}{2}}\no\\
&&+C\Bigg(\int_0^{L}\bigg(\int_{\eta}^{L}\int_{y}^{L}\abs{S_Q(z)}\ud{z}\ud{y}\bigg)^2\ud{\eta}\Bigg)^{\frac{1}{2}},\no
\end{eqnarray}
for some $f_L\in\r$ satisfying
\begin{eqnarray}
\abs{f_L}&\leq&C\bigg(\tss{h}{-}+\tnnm{S}\bigg)+C\abs{\int_0^{L}\br{\sin\phi,S_R}_{\phi}(y)\ud{y}}+C\abs{\int_0^{L}\int_{\eta}^{L}\abs{S_Q(y)}\ud{y}\ud{\eta}}.
\end{eqnarray}
\end{theorem}

\subsection{Exponential Decay}

In this section, we prove the spatial decay of the solution to the
Milne problem.
\begin{theorem}\label{Milne theorem 2}
For $K_0>0$ sufficiently small, the unique solution $f(\eta,\phi)$ to the equation
\begin{eqnarray}
\lnnm{f-f_L}&\leq& C\bigg(\tss{h}{-}+\tnnm{\ue^{K_0\eta}S}+\lnm{h}+\lnnm{\ue^{K_0\eta}S}\bigg)\\
&&+C\Bigg(\int_0^{L}\ue^{2K_0\eta}\bigg(\int_{\eta}^{L}\br{\sin\phi,S_R}_{\phi}(y)\ud{y}\bigg)^2\ud{\eta}\Bigg)^{\frac{1}{2}}
+C\abs{\int_0^{L}\br{\sin\phi,S_R}_{\phi}(y)\ud{y}}\no\\
&&+C\abs{\int_0^{L}\int_{\eta}^{L}\abs{S_Q(y)}\ud{y}\ud{\eta}}+C\Bigg(\int_0^{L}\ue^{2K_0\eta}\bigg(\int_{\eta}^{L}\abs{S_Q(y)}\ud{y}\bigg)^2\ud{\eta}\Bigg)^{\frac{1}{2}}\no\\
&&+C\Bigg(\int_0^{L}\ue^{2K_0\eta}\bigg(\int_{\eta}^{L}\int_{y}^{L}\abs{S_Q(z)}\ud{z}\ud{y}\bigg)^2\ud{\eta}\Bigg)^{\frac{1}{2}},\no
\end{eqnarray}
for some $f_L\in\r$ satisfying
\begin{eqnarray}
\abs{f_L}&\leq&C\bigg(\tss{h}{-}+\tnnm{S}\bigg)+C\abs{\int_0^{L}\br{\sin\phi,S_R}_{\phi}(y)\ud{y}}+C\abs{\int_0^{L}\int_{\eta}^{L}\abs{S_Q(y)}\ud{y}\ud{\eta}}.
\end{eqnarray}
\end{theorem}
\begin{proof}
Define $Z=\ue^{K_0\eta}\v$ for $\v=f-f_L$.
We divide the analysis into several steps:\\
\ \\
Step 1: $L^2$ estimates.\\
The orthogonal property
reveals
\begin{eqnarray}
\br{f,f\sin\phi}_{\phi}(\eta)=\br{r,r\sin\phi}_{\phi}(\eta).
\end{eqnarray}
Multiplying $\ue^{2K_0\eta}f$ on both sides of equation (\ref{Milne
problem}) and integrating over $\phi\in[-\pi,\pi)$, we obtain
\begin{eqnarray}\label{Milne temp 15}
\half\frac{\ud{}}{\ud{\eta}}\bigg(\ue^{2K_0\eta}\br{r,r\sin\phi}_{\phi}(\eta)\bigg)
+\half
F(\eta)\bigg(\ue^{2K_0\eta}\br{r,r\sin\phi}_{\phi}(\eta)\bigg)\\
-\ue^{2K_0\eta}\bigg(K_0\br{r,r\sin\phi}_{\phi}(\eta)-\br{r,
r}_{\phi}(\eta)\bigg)
&=&\ue^{2K_0\eta}\br{S,f}_{\phi}(\eta)\nonumber.
\end{eqnarray}
For $K_0<\min\left\{\dfrac{1}{2},K\right\}$, we have
\begin{eqnarray}\label{Milne temp 16}
\frac{3}{2}\tnm{r(\eta)}^2\geq-K_0\br{r,r\sin\phi}_{\phi}(\eta)+\br{r,r}_{\phi}(\eta)\geq
\half\tnm{r(\eta)}^2.
\end{eqnarray}
Similar to the proof of Lemma \ref{Milne finite LT}, formula as (\ref{Milne temp 15}) and
(\ref{Milne temp 16}) imply
\begin{eqnarray}\label{Milne temp 71}
\tnnm{\ue^{K_0\eta}r}^2=\int_0^{L}\ue^{2K_0\eta}\br{r,r}_{\phi}(\eta)\ud{\eta}\leq
C\bigg(\tss{h}{-}^2+\tnnm{\ue^{K_0\eta}S}^2\bigg).
\end{eqnarray}
From the proof of Lemma \ref{Milne finite LT} and Cauchy's inequality, we can deduce
\begin{eqnarray}
&&\int_0^{L}\ue^{2K_0\eta}\bigg(\int_{-\pi}^{\pi}(f(\eta,\phi)-f_L)^2\ud{\phi}\bigg)\ud{\eta}\\
&\leq&\int_0^{L}\ue^{2K_0\eta}\bigg(\int_{-\pi}^{\pi}r^2(\eta,\phi)\ud{\phi}\bigg)\ud{\eta}+
\int_0^{L}\ue^{2K_0\eta}\bigg(\int_{-\pi}^{\pi}\Big(q(\eta)-q_L\Big)^2\ud{\phi}\bigg)\ud{\eta}
\nonumber\\
&\leq&\int_0^{L}\ue^{2K_0\eta}\tnm{r(\eta)}^2\ud{\eta}\nonumber\\
&&+\int_0^{L}\ue^{2K_0\eta}\bigg(\int_{\eta}^{L}\abs{F(y)}\tnm{r(y)}\ud{y}\bigg)^2\ud{\eta}
+\int_0^{L}\ue^{2K_0\eta}\bigg(\int_{\eta}^{L}\br{\sin\phi,S}_{\phi}(y)\ud{y}\bigg)^2\ud{\eta}
\nonumber\\
&\leq&C\bigg(\tss{h}{-}^2+\tnnm{\ue^{K_0\eta}S}^2\bigg)\no\\
&&+C\bigg(\int_0^{L}\ue^{2K_0\eta}\tnm{r(\eta)}^2\ud{\eta}\bigg)
\bigg(\int_0^{L}\int_{\eta}^{L}\ue^{2K_0(\eta-y)}F^2(y)\ud{y}\ud{\eta}\bigg)
+\int_0^{L}\ue^{2K_0\eta}\bigg(\int_{\eta}^{L}\br{\sin\phi,S}_{\phi}(y)\ud{y}\bigg)^2\ud{\eta}
\nonumber\\
&\leq&C\bigg(\tss{h}{-}^2+\tnnm{\ue^{K_0\eta}S}^2\bigg)\no\\
&&+C\bigg(\int_0^{L}\ue^{2K_0\eta}\tnm{r(\eta)}^2\ud{\eta}\bigg)
\bigg(\int_0^{L}\int_{\eta}^{L}F^2(y)\ud{y}\ud{\eta}\bigg)
+\int_0^{L}\ue^{2K_0\eta}\bigg(\int_{\eta}^{L}\br{\sin\phi,S}_{\phi}(y)\ud{y}\bigg)^2\ud{\eta}
\nonumber\\
&\leq&C\bigg(\tss{h}{-}^2+\tnnm{\ue^{K_0\eta}S}^2\bigg)+\int_0^{L}\ue^{2K_0\eta}\bigg(\int_{\eta}^{L}\br{\sin\phi,S}_{\phi}(y)\ud{y}\bigg)^2\ud{\eta}\nonumber.
\end{eqnarray}
This completes the proof of $L^2$ estimate when $\bar S=0$.
By the method introduced in Lemma \ref{Milne infinite LT general},
we can extend above $L^2$ estimates to the general $S$ case. Note
all the auxiliary functions
constructed in Lemma \ref{Milne infinite LT general} satisfy the estimates.
We have
\begin{eqnarray}
\tnnm{Z}&\leq& C\bigg(\tss{h}{-}+\tnnm{S}\bigg)+C\Bigg(\int_0^{L}\ue^{2K_0\eta}\bigg(\int_{\eta}^{L}\br{\sin\phi,S_R}_{\phi}(y)\ud{y}\bigg)^2\ud{\eta}\Bigg)^{\frac{1}{2}}\\
&&+C\Bigg(\int_0^{L}\ue^{2K_0\eta}\bigg(\int_{\eta}^{L}\abs{S_Q(y)}\ud{y}\bigg)^2\ud{\eta}\Bigg)^{\frac{1}{2}}
+C\Bigg(\int_0^{L}\ue^{2K_0\eta}\bigg(\int_{\eta}^{L}\int_{y}^{L}\abs{S_Q(z)}\ud{z}\ud{y}\bigg)^2\ud{\eta}\Bigg)^{\frac{1}{2}},\no
\end{eqnarray}
\ \\
Step 2: $L^{\infty}$ estimates.\\
$Z$ satisfies the equation
\begin{eqnarray}\label{decay equation}
\left\{
\begin{array}{l}\displaystyle
\sin\phi\frac{\p Z}{\p\eta}+F(\eta)\cos\phi\frac{\p
Z}{\p\phi}+Z=\bar Z+\ue^{K_0\eta}S+K_0\sin\phi Z,\\\rule{0ex}{1.5em}
Z(0,\phi)=p(\phi)=h(\phi)-f_L\ \ \text{for}\ \ \sin\phi>0\\\rule{0ex}{1.5em}
Z(L,\phi)=Z(L,\rr[\phi]).
\end{array}
\right.
\end{eqnarray}
Since we know $Z=\k[p]+\t[\bar Z+\ue^{K_0\eta}S+K_0\sin\phi Z]$
leads to
\begin{eqnarray}
\t[\bar Z]=Z-\k[p]-\t[\ue^{K_0\eta}S]-\t[K_0\sin\phi Z],
\end{eqnarray}
then by Lemma \ref{Milne lemma 3}, (\ref{Milne temp 61}) and
(\ref{Milne temp 62}), we can show
\begin{eqnarray}\label{Milne temp 73}
\ltnm{Z-\k[p]-\t[\ue^{K_0\eta}S]-\t[K_0\sin\phi Z]}
&\leq& C(\delta)\tnnm{\bar
Z}+\delta\lnnm{\bar Z}\\
&\leq&
C(\delta)\tnnm{Z}+\delta\ltnm{Z}\nonumber.
\end{eqnarray}
Therefore, based on Lemma \ref{Milne lemma 1} and (\ref{Milne temp
63}), we can directly estimate
\begin{eqnarray}
\\
\ltnm{Z}
&\leq&\lnm{\k[p]}+\lnnm{\t[\ue^{K_0\eta}S]}+\lnnm{\t[K_0\sin\phi Z]}+C(\delta)\tnnm{Z}+\delta\ltnm{Z}\no\\
&\leq&\lnm{p}+\lnnm{\ue^{K_0\eta}S}+K_0\lnnm{Z}+C(\delta)\tnnm{Z}+\delta\ltnm{Z}\nonumber.
\end{eqnarray}
We can take $\delta=\dfrac{1}{2}$ to obtain
\begin{eqnarray}\label{Milne temp 74}
\ltnm{Z}\leq
C\bigg(\lnm{p}+\lnnm{\ue^{K_0\eta}S}+K_0\lnnm{Z}+\tnnm{Z}\bigg).
\end{eqnarray}
Then based on Lemma \ref{Milne lemma 1}, Lemma \ref{Milne lemma 2}
and Lemma \ref{Milne lemma 3}, we can deduce
\begin{eqnarray}
\lnnm{Z}&\leq& \lnm{\ue^{K_0\eta}\k[p]}+\lnnm{\ue^{K_0\eta}\t[S]}+\lnnm{\bar Z}+\lnnm{K_0\sin\phi Z}\\
&\leq& \lnm{p}+\lnnm{\ue^{K_0\eta}S}+\lnnm{\bar Z}+K_0\lnnm{Z}\nonumber\\
&\leq& \lnm{p}+\lnnm{\ue^{K_0\eta}S}+\ltnm{Z}+K_0\lnnm{Z}\nonumber\\
&\leq&
C\bigg(\tnnm{Z}+\tnnm{\ue^{K_0\eta}S}+\lnnm{\ue^{K_0\eta}S}+\lnm{p}+K_0\lnnm{Z}\bigg)\nonumber.
\end{eqnarray}
Taking $K_0$ sufficiently small, we absorb $K_0\lnnm{Z}$ to the left-hand side and obtain
\begin{eqnarray}
\lnnm{Z}
&\leq&
C\bigg(\tnnm{Z}+\tnnm{\ue^{K_0\eta}S}+\lnnm{\ue^{K_0\eta}S}+\lnm{p}\bigg).
\end{eqnarray}
Then the final result is obvious.
\end{proof}

\subsection{Maximum Principle}

In \cite{AA003}, the author proved the maximum principle.
\begin{theorem}\label{Milne theorem 3}
The unique solution $f(\eta,\phi)$ to the equation with $S=0$ satisfies the maximum principle, i.e.
\begin{eqnarray}
\min_{\sin\phi>0}h(\phi)\leq f(\eta,\phi)\leq
\max_{\sin\phi>0}h(\phi).
\end{eqnarray}
\end{theorem}

\newpage

\section{Regularity of $\e$-Milne Problem with Geometric Correction}

We consider the $\e$-Milne problem with geometric correction for $f^{\e}(\eta,\tau,\phi)$ in
the domain $(\eta,\tau,\phi)\in[0,L]\times[-\pi,\pi)\times[-\pi,\pi)$ where $L=\e^{-\frac{1}{2}}$ as
\begin{eqnarray}\label{Milne problem.,}
\left\{ \begin{array}{l}\displaystyle
\sin\phi\frac{\p
f^{\e}}{\p\eta}+F(\e;\eta,\tau)\cos\phi\frac{\p
f^{\e}}{\p\phi}+f^{\e}-\bar f^{\e}=S^{\e}(\eta,\tau,\phi),\\\rule{0ex}{1.5em}
f^{\e}(0,\tau,\phi)= h^{\e}(\tau,\phi)\ \ \text{for}\
\ \sin\phi>0,\\\rule{0ex}{1.5em}
f^{\e}(L,\tau,\phi)=f^{\e}(L,\tau,\rr[\phi]),
\end{array}
\right.
\end{eqnarray}
where $\rr[\phi]=-\phi$ and
\begin{eqnarray}
F(\e;\eta,\tau)=-\frac{\e}{\rk(\tau)-\e\eta},
\end{eqnarray}
for the radius of curvature $\rk$. In this section, for convenience, we temporarily ignore the superscript on $\e$ and $\tau$. In other words, we will study
\begin{eqnarray}\label{Milne problem,}
\left\{ \begin{array}{l}\displaystyle
\sin\phi\frac{\p
f}{\p\eta}+F(\eta)\cos\phi\frac{\p
f}{\p\phi}+f-\bar f=S(\eta,\phi),\\\rule{0ex}{1.5em}
f(0,\phi)= h(\phi)\ \ \text{for}\
\ \sin\phi>0,\\\rule{0ex}{1.5em}
f(L,\phi)=f(L,\rr[\phi]).
\end{array}
\right.
\end{eqnarray}
Define potential function $V(\eta)$ satisfying $V(0)=0$ and $\dfrac{\p V}{\p\eta}=-F(\eta)$. Then we can direct compute
\begin{eqnarray}
V(\eta)=\ln\left(\frac{\rk}{\rk-\e\eta}\right).
\end{eqnarray}
Define the weight function
\begin{eqnarray}
\zeta(\eta,\phi)=\Bigg(1-\bigg(\frac{\rk-\e\eta}{\rk}\cos\phi\bigg)^2\Bigg)^{\frac{1}{2}}.
\end{eqnarray}
We can easily show that
\begin{eqnarray}
\sin\phi\frac{\p \zeta}{\p\eta}+F(\eta)\cos\phi\frac{\p
\zeta}{\p\phi}=0.
\end{eqnarray}
It is easy to see $\v(\eta,\tau,\phi)=f(\eta,\tau,\phi)-f_L(\tau)$ satisfies the equation
\begin{eqnarray}\label{Milne difference problem}
\left\{
\begin{array}{l}\displaystyle
\sin\phi\frac{\p \v}{\p\eta}+F(\eta)\cos\phi\frac{\p
\v}{\p\phi}+\v=\bar \v+S,\\\rule{0ex}{1.5em}
\v(0,\phi)=p(\phi)=h(\phi)-f_L\ \ \text{for}\ \ \sin\phi>0,\\\rule{0ex}{1.5em}
\v(L,\phi)=\v(L,\rr[\phi]).
\end{array}
\right.
\end{eqnarray}
The regularity has been thoroughly studied in \cite{AA007}. However, here we will focus on the a priori estimates and prove an improved version of the regularity theorem. The major upshot is that we can avoid using the information of $\dfrac{\p S}{\p\phi}$.

\subsection{Mild Formulation}

Consider the $\e$-transport problem for $\a=\zeta\dfrac{\p\v}{\p\eta}$ as
\begin{eqnarray}
\left\{
\begin{array}{l}\displaystyle
\sin\phi\frac{\p\a}{\p\eta}+F(\eta)\cos\phi\frac{\p
\a}{\p\phi}+\a=\tilde\a+S_{\a},\\\rule{0ex}{1.5em}
\a(0,\phi)=p_{\a}(\phi)\ \ \text{for}\ \ \sin\phi>0,\\\rule{0ex}{1.5em}
\a(L,\phi)=\a(L,R\phi),
\end{array}
\right.
\end{eqnarray}
where $p_{\a}$ and $S_{\a}$ will be specified later with
\begin{eqnarray}
\tilde\a(\eta,\phi)=\frac{1}{2\pi}\int_{-\pi}^{\pi}\frac{\zeta(\eta,\phi)}{\zeta(\eta,\phi_{\ast})}\a(\eta,\phi_{\ast})\ud{\phi_{\ast}}.
\end{eqnarray}
Define the energy as before
\begin{eqnarray}
E(\eta,\phi)=\ue^{-V(\eta)}\cos\phi=\cos\phi\frac{\rk-\e\eta}{\rk}.
\end{eqnarray}
Along the characteristics, where this energy is conserved and $\zeta$ is a constant, the equation can be simplified as follows:
\begin{eqnarray}
\sin\phi\frac{\ud{\a}}{\ud{\eta}}+\a=\tilde\a+S_{\a}.
\end{eqnarray}
An implicit function
$\eta^+(\eta,\phi)$ can be determined through
\begin{eqnarray}
\abs{E(\eta,\phi)}=\ue^{-V(\eta^+)}.
\end{eqnarray}
which means $(\eta^+,\phi_0)$ with $\sin\phi_0=0$ is on the same characteristics as $(\eta,\phi)$.
Define the quantities for $0\leq\eta'\leq\eta^+$ as follows:
\begin{eqnarray}
\phi'(\eta,\phi;\eta')&=&\cos^{-1}\Big(\ue^{V(\eta')-V(\eta)}\cos\phi \Big),\\
\rr[\phi'(\eta,\phi;\eta')]&=&-\cos^{-1}\Big(\ue^{V(\eta')-V(\eta)}\cos\phi \Big)=-\phi'(\eta,\phi;\eta'),
\end{eqnarray}
where the inverse trigonometric function can be defined
single-valued in the domain $[0,\pi)$ and the quantities are always well-defined due to the monotonicity of $V$. Note that $\sin\phi'\geq0$, even if $\sin\phi<0$.
Finally we put
\begin{eqnarray}
G_{\eta,\eta'}(\phi)&=&\int_{\eta'}^{\eta}\frac{1}{\sin\Big(\phi'(\eta,\phi;\xi)\Big)}\ud{\xi}.
\end{eqnarray}
Similar to $\e$-Milne problem, we can define the solution along the characteristics as follows:
\begin{eqnarray}
\a(\eta,\phi)=\k[p_{\a}]+\t[\tilde\a+S_{\a}],
\end{eqnarray}
where\\
\ \\
Region I:\\
For $\sin\phi>0$,
\begin{eqnarray}
\k[p_{\a}]&=&p_{\a}\Big(\phi'(\eta,\phi;0)\Big)\exp(-G_{\eta,0})\\
\t[\tilde\a+S_{\a}]&=&\int_0^{\eta}\frac{(\tilde\a+S_{\a})\Big(\eta',\phi'(\eta,\phi;\eta')\Big)}{\sin\Big(\phi'(\eta,\phi;\eta')\Big)}\exp(-G_{\eta,\eta'})\ud{\eta'}.
\end{eqnarray}
\ \\
Region II:\\
For $\sin\phi<0$ and $\abs{E(\eta,\phi)}\leq \ue^{-V(L)}$,
\begin{eqnarray}
\k[p_{\a}]&=&p_{\a}\Big(\phi'(\eta,\phi;0)\Big)\exp(-G_{L,0}-G_{L,\eta})\\
\t[\tilde\a+S_{\a}]&=&\int_0^{L}\frac{(\tilde\a+S)\Big(\eta',\phi'(\eta,\phi;\eta')\Big)}{\sin\Big(\phi'(\eta,\phi;\eta')\Big)}
\exp(-G_{L,\eta'}-G_{L,\eta})\ud{\eta'}\\
&&+\int_{\eta}^{L}\frac{(\tilde\a+S)\Big(\eta',\rr[\phi'(\eta,\phi;\eta')]\Big)}{\sin\Big(\phi'(\eta,\phi;\eta')\Big)}\exp(-G_{\eta',\eta})\ud{\eta'}\nonumber.
\end{eqnarray}
\ \\
Region III:\\
For $\sin\phi<0$ and $\abs{E(\eta,\phi)}\geq \ue^{-V(L)}$,
\begin{eqnarray}
\k[p_{\a}]&=&p_{\a}\Big(\phi'(\eta,\phi;0)\Big)\exp(-G_{\eta^+,0}-G_{\eta^+,\eta})\\
\t[\tilde\a+S_{\a}]&=&\int_0^{\eta^+}\frac{(\tilde\a+S_{\a})\Big(\eta',\phi'(\eta,\phi;\eta')\Big)}{\sin\Big(\phi'(\eta,\phi;\eta')\Big)}
\exp(-G_{\eta^+,\eta'}-G_{\eta^+,\eta})\ud{\eta'}\\
&&+
\int_{\eta}^{\eta^+}\frac{(\tilde\a+S_{\a})\Big(\eta',\rr[\phi'(\eta,\phi;\eta')]\Big)}{\sin\Big(\phi'(\eta,\phi;\eta')\Big)}\exp(-G_{\eta',\eta})\ud{\eta'}\nonumber.
\end{eqnarray}
Then we need to estimate $\k[p_{\a}]$ and $\t[\tilde\a+S_{\a}]$ in each region. We assume $0<\d<<1$ and $0<\d_0<<1$ are small quantities which will be determined later.
Since we always assume that $(\eta,\phi)$ and $(\eta',\phi')$ are on the same characteristics, when there is no confusion, we simply write $\phi'$ or $\phi'(\eta')$ instead of $\phi'(\eta,\phi;\eta')$.

\subsection{Region I: $\sin\phi>0$}

We consider
\begin{eqnarray}
\k[p_{\a}]&=&p_{\a}\Big(\phi'(\eta,\phi;0)\Big)\exp(-G_{\eta,0})\\
\t[\tilde\a+S_{\a}]&=&\int_0^{\eta}\frac{(\tilde\a+S_{\a})\Big(\eta',\phi'(\eta,\phi;\eta')\Big)}{\sin\Big(\phi'(\eta,\phi;\eta')\Big)}\exp(-G_{\eta,\eta'})\ud{\eta'}.
\end{eqnarray}
Based on \cite[Lemma 4.7, Lemma 4.8]{AA003},
we can directly obtain
\begin{eqnarray}
\lnm{\k[p_{\a}]}&\leq&\lnm{p_{\a}},\\
\lnnm{\t[S_{\a}]}&\leq&\lnnm{S_{\a}}.
\end{eqnarray}
Hence, we only need to estimate
\begin{eqnarray}\label{mild 1}
I=\t[\tilde\a]=\int_0^{\eta}\frac{\tilde\a\Big(\eta',\phi'(\eta,\phi;\eta')\Big)}{\sin\Big(\phi'(\eta,\phi;\eta')\Big)}\exp(-G_{\eta,\eta'})\ud{\eta'}.
\end{eqnarray}
We divide it into several steps:\\
\ \\
Step 0: Preliminaries.\\
We have
\begin{eqnarray}
E(\eta',\phi')=\frac{\rk-\e\eta'}{\rk}\cos\phi'.
\end{eqnarray}
We can directly obtain
\begin{eqnarray}\label{pt 01}
\zeta(\eta',\phi')&=&\frac{1}{\rk}\sqrt{\rk^2-\bigg((\rk-\e\eta')\cos\phi'\bigg)^2}
=\frac{1}{\rk}\sqrt{\rk^2-(\rk-\e\eta')^2+(\rk-\e\eta')^2\sin^2\phi'},\\
&\leq& \sqrt{\rk^2-(\rk-\e\eta')^2}+\sqrt{(\rk-\e\eta')^2\sin^2\phi'}\leq C\bigg(\sqrt{\e\eta'}+\sin\phi'\bigg),\no
\end{eqnarray}
and
\begin{eqnarray}\label{pt 02}
\zeta(\eta',\phi')\geq\frac{1}{\rk}\sqrt{\rk^2-(\rk-\e\eta')^2}\geq C\sqrt{\e\eta'}.
\end{eqnarray}
Also, we know for $0\leq\eta'\leq\eta$,
\begin{eqnarray}
\sin\phi'&=&\sqrt{1-\cos^2\phi'}=\sqrt{1-\bigg(\frac{\rk-\e\eta}{\rk-\e\eta'}\bigg)^2\cos^2\phi}\\
&=&\frac{\sqrt{(\rk-\e\eta')^2\sin^2\phi+(2\rk-\e\eta-\e\eta')(\e\eta-\e\eta')\cos^2\phi}}{\rk-\e\eta'}.
\end{eqnarray}
Since
\begin{eqnarray}
0\leq(2\rk-\e\eta-\e\eta')(\e\eta-\e\eta')\cos^2\phi\leq 2\rk\e(\eta-\eta'),
\end{eqnarray}
we have
\begin{eqnarray}
\sin\phi\leq\sin\phi'
\leq2\sqrt{\sin^2\phi+\e(\eta-\eta')},
\end{eqnarray}
which means
\begin{eqnarray}
\frac{1}{2\sqrt{\sin^2\phi+\e(\eta-\eta')}}\leq\frac{1}{\sin\phi'}
\leq\frac{1}{\sin\phi}.
\end{eqnarray}
Therefore,
\begin{eqnarray}\label{pt 03}
-\int_{\eta'}^{\eta}\frac{1}{\sin\phi'(y)}\ud{y}&\leq& -\int_{\eta'}^{\eta}\frac{1}{2\sqrt{\sin^2\phi+\e(\eta-y)}}\ud{y}\\
&=&\frac{1}{\e}\bigg(\sin\phi-\sqrt{\sin^2\phi+\e(\eta-\eta')}\bigg)\no\\
&=&-\frac{\eta-\eta'}{\sin\phi+\sqrt{\sin^2\phi+\e(\eta-\eta')}}\no\\
&\leq&-\frac{\eta-\eta'}{2\sqrt{\sin^2\phi+\e(\eta-\eta')}}.\no
\end{eqnarray}
Define a cut-off function $\chi\in C^{\infty}[-\pi,\pi]$ satisfying
\begin{eqnarray}
\chi(\phi)=\left\{
\begin{array}{ll}
1&\text{for}\ \ \abs{\sin\phi}\leq\d,\\
0&\text{for}\ \ \abs{\sin\phi}\geq2\d,
\end{array}
\right.
\end{eqnarray}
In the following, we will divide the estimate of $I$ into several cases based on the value of $\sin\phi$, $\abs{\cos\phi}$, $\sin\phi'$, $\e\eta'$ and $\e(\eta-\eta')$. Let $\id$ denote the indicator function. We write
\begin{eqnarray}
I&=&\int_0^{\eta}\id_{\{\sin\phi\geq\d_0\}}\id_{\{\abs{\cos\phi}\geq\d_0\}}+\int_0^{\eta}\id_{\{0\leq\sin\phi\leq\d_0\}}\id_{\{\chi(\phi_{\ast})<1\}}\\
&&+\int_0^{\eta}\id_{\{0\leq\sin\phi\leq\d_0\}}\id_{\{\chi(\phi_{\ast})=1\}}\id_{\{ \sqrt{\e\eta'}\geq\sin\phi'\}}\no\\
&&+\int_0^{\eta}\id_{\{0\leq\sin\phi\leq\d_0\}}\id_{\{\chi(\phi_{\ast})=1\}}\id_{\{\sqrt{\e\eta'}\leq\sin\phi'\}}\id_{\{\sin^2\phi\leq\e(\eta-\eta')\}}\no\\
&&+\int_0^{\eta}\id_{\{0\leq\sin\phi\leq\d_0\}}\id_{\{\chi(\phi_{\ast})=1\}}\id_{\{\sqrt{\e\eta'}\leq\sin\phi'\}}\id_{\{\sin^2\phi\geq\e(\eta-\eta')\}}\no\\
&&+\int_0^{\eta}\id_{\{\abs{\cos\phi}\leq\d_0\}}\no\\
&=&I_1+I_2+I_3+I_4+I_5+I_6.\no
\end{eqnarray}
\ \\
Step 1: Estimate of $I_1$ for $\sin\phi\geq\d_0$ and $\abs{\cos\phi}\geq\d_0$.\\
For $\sin\phi\geq\d_0$ and $\abs{\cos\phi}\geq\d_0$, we do not need the mild formulation of $\a$. Instead, we directly estimate
\begin{eqnarray}
\abs{\a}\leq \abs{\frac{\p\v}{\p\eta}}.
\end{eqnarray}
We will estimate $I_1$ based on the characteristics of $\v$ itself instead of the derivative.
Here, we will use two formulations of the equation (\ref{Milne difference problem}) along the characteristics
\begin{itemize}
\item
Formulation I: $\eta$ is the principal variable, $\phi=\phi(\eta)$, and the equation can be rewritten as
\begin{eqnarray}
\sin\phi\frac{\ud{\v}}{\ud{\eta}}+\v=\bar\v+S.
\end{eqnarray}
\item
Formulation II: $\phi$ is the principal variable, $\eta=\eta(\phi)$ and the equation can be rewritten as
\begin{eqnarray}
F(\eta)\cos\phi\frac{\ud{\v}}{\ud{\phi}}+\v=\bar\v+ S.
\end{eqnarray}
\end{itemize}
These two formulations are equivalent and can be applied to different regions of the domain.

We may decompose $\v=\v_1+\v_2$ where $\v_1$ satisfies
\begin{eqnarray}\label{Milne difference problem 1}
\left\{
\begin{array}{l}\displaystyle
\sin\phi\frac{\p \v_1}{\p\eta}+F(\eta)\cos\phi\frac{\p
\v_1}{\p\phi}+\v_1=\bar \v,\\\rule{0ex}{1.5em}
\v_1(0,\phi)=p(\phi)\ \ \text{for}\ \ \sin\phi>0,\\\rule{0ex}{1.5em}
\v_1(L,\phi)=\v_1(L,\rr[\phi]),
\end{array}
\right.
\end{eqnarray}
and $\v_2$ satisfies
\begin{eqnarray}\label{Milne difference problem 2}
\left\{
\begin{array}{l}\displaystyle
\sin\phi\frac{\p \v_2}{\p\eta}+F(\eta)\cos\phi\frac{\p
\v_2}{\p\phi}+\v_2=S,\\\rule{0ex}{1.5em}
\v_2(0,\phi)=0\ \ \text{for}\ \ \sin\phi>0,\\\rule{0ex}{1.5em}
\v_2(L,\phi)=\v_2(L,\rr[\phi]).
\end{array}
\right.
\end{eqnarray}
Assume $\v$ is well-defined. Then we can easily see that $\v_1$ and $\v_2$ are well-defined.

Using Formulation I, we rewrite the equation (\ref{Milne difference problem 1}) along the characteristics as
\begin{eqnarray}\label{mt 01}
\v_1(\eta,\phi)&=&\exp\left(-G_{\eta,0}\right)\Bigg(p\Big(\phi'(0)\Big)
+\int_0^{\eta}\frac{\bar\v(\eta')}{\sin\Big(\phi'(\eta')\Big)}
\exp\left(G_{\eta',0}\right)\ud{\eta'}\Bigg)
\end{eqnarray}
where $(\eta',\phi')$, $(0,\phi'(0))$ and $(\eta,\phi)$ are on the same characteristic with $\sin\phi'\geq0$, and
\begin{eqnarray}
G_{t,s}&=&\int_{s}^{t}\frac{1}{\sin\Big(\phi'(\xi)\Big)}\ud{\xi}.
\end{eqnarray}
Taking $\eta$ derivative on both sides of (\ref{mt 01}), we have
\begin{eqnarray}\label{mt 04}
\frac{\p\v_1}{\p\eta}&=&X_1+X_2+X_3+X_4+X_5,
\end{eqnarray}
where
\begin{eqnarray}
X_1&=&-\exp\left(-G_{\eta,0}\right)\frac{\p G_{\eta,0}}{\p\eta}\Bigg(p\Big(\phi'(0)\Big)
+\int_0^{\eta}\frac{\bar\v(\eta')}{\sin\Big(\phi'(\eta')\Big)}
\exp\left(G_{\eta',0}\right)\ud{\eta'}\Bigg),\no\\
X_2&=&\exp\left(-G_{\eta,0}\right)\frac{\p p\Big(\phi'(0)\Big)}{\p\eta},\\
X_3&=&\frac{\bar\v(\eta)}{\sin\phi},\\
X_4&=&-\exp\left(-G_{\eta,0}\right)\int_0^{\eta}\bar\v(\eta')
\exp\left(G_{\eta',0}\right)
\frac{\cos\Big(\phi'(\eta')\Big)}{\sin^2\Big(\phi'(\eta')\Big)}\frac{\p\phi'(\eta')}{\p\eta}\ud{\eta'},\no\\
X_5&=&\exp\left(-G_{\eta,0}\right)\int_0^{\eta}\frac{\bar\v(\eta')}{\sin\Big(\phi'(\eta')\Big)}
\exp\left(G_{\eta',0}\right)\frac{\p G_{\eta',0}}{\p\eta}\ud{\eta'}.
\end{eqnarray}
Then we need to estimate each term. This procedure is standard, so we omit the details. Note that fact that for $0\leq\eta'\leq\eta$, we have $\sin\phi'\geq\sin\phi\geq\d_0$ and
\begin{eqnarray}
\int_0^{\eta}\frac{1}{\sin\Big(\phi'(\eta')\Big)}
\exp\left(-G_{\eta,\eta'}\right)\ud{\eta'}\leq \int_0^{\infty}\ue^{-y}\ud{y}=1,
\end{eqnarray}
with the substitution $y=G_{\eta,\eta'}$. The estimates can be listed as below:
\begin{eqnarray}
\abs{X_1}&\leq&\frac{C}{\d_0}\lnnm{\v},\\
\abs{X_2}&\leq&\frac{C}{\d_0}\lnm{\frac{\p p}{\p\phi}},\\
\abs{X_3}&\leq&\frac{C}{\d_0}\lnnm{\v},\\
\abs{X_4}&\leq&\frac{C}{\d_0}\lnnm{\v},\\
\abs{X_5}&\leq&\frac{C}{\d_0}\lnnm{\v}.
\end{eqnarray}
In total, we have
\begin{eqnarray}
\abs{\frac{\p\v_1}{\p\eta}}\leq \frac{C}{\d_0}\bigg(\lnm{\frac{\p p}{\p\phi}}+\lnnm{\v}\bigg).
\end{eqnarray}

Using Formulation II, we rewrite the equation (\ref{Milne difference problem 2}) along the characteristics as
\begin{eqnarray}\label{mt 01'}
\v_2(\eta,\phi)&=&\exp\left(-H_{\phi,\phi_{\ast}}\right)\int_{\phi_{\ast}}^{\phi}\frac{S\Big(\eta'(\phi'),\phi'\Big)}{F\Big(\eta'(\phi')\Big)\cos\phi'}
\exp\left(H_{\phi',\phi_{\ast}}\right)\ud{\phi'}.
\end{eqnarray}
where $(\eta',\phi')$, $(0,\phi_{\ast})$ and $(\eta,\phi)$ are on the same characteristic with $\sin\phi'\geq0$, and
\begin{eqnarray}
H_{t,s}&=&\int_{s}^{t}\frac{1}{F\Big(\eta'(\xi)\Big)\cos\xi'}\ud{\xi}.
\end{eqnarray}
Taking $\eta$ derivative on both sides of (\ref{mt 01'}), we have
\begin{eqnarray}\label{mt 04'}
\frac{\p\v_2}{\p\eta}&=&Y_1+Y_2+Y_3+Y_4+Y_5,
\end{eqnarray}
where
\begin{eqnarray}
Y_1&=&-\exp\left(-H_{\phi,\phi_{\ast}}\right)\frac{\p H_{\phi,\phi_{\ast}}}{\p\eta}\int_{\phi_{\ast}}^{\phi}\frac{S\Big(\eta'(\phi'),\phi'\Big)}{F\Big(\eta'(\phi')\Big)\cos\phi'}
\exp\left(H_{\phi',\phi_{\ast}}\right)\ud{\phi'},\no\\
Y_2&=&\frac{S(0,\phi_{\ast})}{F(0)\cos\phi_{\ast}}\frac{\p\phi_{\ast}}{\p\eta},\\
Y_3&=&-\exp\left(-H_{\phi,\phi_{\ast}}\right)\int_{\phi_{\ast}}^{\phi}S\Big(\eta'(\phi'),\phi'\Big)\frac{1}{F^2\Big(\eta'(\phi')\Big)\cos\phi'}\frac{\p F\Big(\eta'(\phi')\Big)}{\p\eta}
\exp\left(H_{\phi',\phi_{\ast}}\right)\ud{\phi'},\no\\
Y_4&=&\exp\left(-H_{\phi,\phi_{\ast}}\right)\int_{\phi_{\ast}}^{\phi}\frac{S\Big(\eta'(\phi'),\phi'\Big)}{F\Big(\eta'(\phi')\Big)\cos\phi'}
\exp\left(H_{\phi',\phi_{\ast}}\right)\frac{\p H_{\phi',\phi_{\ast}}}{\p\eta}\ud{\phi'},\\
Y_5&=&\exp\left(-H_{\phi,\phi_{\ast}}\right)\int_{\phi_{\ast}}^{\phi}\frac{\p_{\eta'}S\Big(\eta'(\phi'),\phi'\Big)}{F\Big(\eta'(\phi')\Big)\cos\phi'}\frac{\p\eta'(\phi')}{\p\eta}
\exp\left(H_{\phi',\phi_{\ast}}\right)\ud{\phi'}.
\end{eqnarray}
Then we just need to estimate each term. Along the characteristics, we know
\begin{eqnarray}
\ue^{-V(\eta')}\cos\phi'=\ue^{-V(\eta)}\cos\phi,
\end{eqnarray}
which implies
\begin{eqnarray}
\cos\phi'&=&\ue^{V(\eta')-V(\eta)}\cos\phi\geq \ue^{V(0)-V(L)}\cos\phi\geq \ue^{V(0)-V(L)}\d_0.
\end{eqnarray}
We can further deduce that
\begin{eqnarray}
\cos\phi'\geq\bigg(1-\frac{\e^{\frac{1}{2}}}{\rk}\bigg)\d_0\geq \frac{\d_0}{2},
\end{eqnarray}
when $\e$ is sufficiently small. Also, we have
\begin{eqnarray}
\int_{\phi_{\ast}}^{\phi}\frac{1}{F\Big(\eta'(\phi')\Big)\cos\phi'}
\exp\left(H_{\phi,\phi'}\right)\ud{\phi'}\leq \int_0^{\infty}\ue^{-y}\ud{y}=1,
\end{eqnarray}
with the substitution $y=H_{\phi,\phi'}$. Similar to $X_i$ estimates, we may directly obtain
\begin{eqnarray}
\abs{Y_1}&\leq&\frac{C}{\d_0}\lnnm{S},\\
\abs{Y_2}&\leq&\frac{C}{\d_0}\lnnm{S},\\
\abs{Y_3}&\leq&\frac{C}{\d_0}\lnnm{S},\\
\abs{Y_4}&\leq&\frac{C}{\d_0}\lnnm{S},\\
\abs{Y_5}&\leq&\frac{C}{\d_0}\lnnm{\frac{\p S}{\p\eta}}.
\end{eqnarray}
In total, we have
\begin{eqnarray}
\abs{\frac{\p\v_2}{\p\eta}}\leq \frac{C}{\d_0}\bigg(\lnnm{S}+\lnnm{\frac{\p S}{\p\eta}}\bigg).
\end{eqnarray}
Combining all above, we have
\begin{eqnarray}
\abs{\frac{\p\v}{\p\eta}}&\leq&\abs{\frac{\p\v_1}{\p\eta}}+\abs{\frac{\p\v_2}{\p\eta}}\leq \frac{C}{\d_0}\bigg(\lnm{\frac{\p p}{\p\phi}}+\lnnm{S}+\lnnm{\frac{\p S}{\p\eta}}+\lnnm{\v}\bigg).
\end{eqnarray}
Hence, noting that $\zeta\geq\d_0$, we know
\begin{eqnarray}
I_1&\leq&\frac{C}{\d_0^2}\bigg(\lnm{\zeta\frac{\p p}{\p\phi}}+\lnnm{S}+\lnnm{\zeta\frac{\p S}{\p\eta}}+\lnnm{\v}\bigg).
\end{eqnarray}
\ \\
Step 2: Estimate of $I_2$ for $0\leq\sin\phi\leq\d_0$ and $\chi(\phi_{\ast})<1$.\\
We have
\begin{eqnarray}
I_2&=&\frac{1}{2\pi}\int_0^{\eta}\bigg(\int_{-\pi}^{\pi}\frac{\zeta(\eta',\phi')}{\zeta(\eta',\phi_{\ast})}\Big(1-\chi(\phi_{\ast})\Big)
\a(\eta',\phi_{\ast})\ud{\phi_{\ast}}\bigg)
\frac{1}{\sin\phi'}\exp(-G_{\eta,\eta'})\ud{\eta'}\\
&=&\frac{1}{2\pi}\int_0^{\eta}\bigg(\int_{-\pi}^{\pi}\zeta(\eta',\phi')\Big(1-\chi(\phi_{\ast})\Big)
\frac{\v(\eta',\phi_{\ast})}{\p\eta'}\ud{\phi_{\ast}}\bigg)\frac{1}{\sin\phi'}\exp(-G_{\eta,\eta'})\ud{\eta'}.\no
\end{eqnarray}
Based on the $\e$-Milne problem of $\v$ as
\begin{eqnarray}
\sin\phi_{\ast}\frac{\p\v(\eta',\phi_{\ast})}{\p\eta'}+F(\eta')\cos\phi_{\ast}\frac{\p\v(\eta',\phi_{\ast})}{\p\phi_{\ast}}+\v(\eta',\phi_{\ast})-\bar\v(\eta')=S(\eta',\phi_{\ast}),
\end{eqnarray}
we have
\begin{eqnarray}
\frac{\p\v(\eta',\phi_{\ast})}{\p\eta'}=-\frac{1}{\sin\phi_{\ast}}
\bigg(F(\eta')\cos\phi_{\ast}\frac{\p\v(\eta',\phi_{\ast})}{\p\phi_{\ast}}+\v(\eta',\phi_{\ast})-\bar\v(\eta')-S(\eta',\phi_{\ast})\bigg)
\end{eqnarray}
Hence, we have
\begin{eqnarray}
\tilde\a&=&\int_{-\pi}^{\pi}\zeta(\eta',\phi')\Big(1-\chi(\phi_{\ast})v)
\frac{\p\v(\eta',\phi_{\ast})}{\p\eta'}\ud{\phi_{\ast}}\\
&=&-\int_{-\pi}^{\pi}\zeta(\eta',\phi')\Big(1-\chi(\phi_{\ast})\Big)
\frac{1}{\sin\phi_{\ast}}
\bigg(\v(\eta',\phi_{\ast})-\bar\v(\eta')-S(\eta',\phi_{\ast})\bigg)\ud{\phi_{\ast}}\no\\
&&-\int_{-\pi}^{\pi}\zeta(\eta',\phi')\Big(1-\chi(\phi_{\ast})\Big)
\frac{1}{\sin\phi_{\ast}}
F(\eta')\cos\phi_{\ast}\frac{\p\v(\eta',\phi_{\ast})}{\p\phi_{\ast}}\ud{\phi_{\ast}}\no\\
&=&\tilde\a_1+\tilde\a_2.\no
\end{eqnarray}
We may directly obtain
\begin{eqnarray}
\abs{\tilde\a_1}&\leq&\int_{-\pi}^{\pi}\zeta(\eta',\phi')\Big(1-\chi(\phi_{\ast})\Big)
\frac{1}{\sin\phi_{\ast}}
\bigg(\v(\eta',\phi_{\ast})-\bar\v(\eta')-S(\eta',\phi_{\ast})\bigg)\ud{\phi_{\ast}}\\
&\leq&\frac{\rk}{\d}\abs{\int_{-\pi}^{\pi}
\bigg(\v(\eta',\phi_{\ast})-\bar\v(\eta')-S(\eta',\phi_{\ast})\bigg)\ud{\phi_{\ast}}}\no\\
&\leq&\frac{C}{\d}\bigg(\lnnm{\v}+\lnnm{S}\bigg).\no
\end{eqnarray}
On the other hand, an integration by parts yields
\begin{eqnarray}
\tilde\a_2&=&\int_{-\pi}^{\pi}\frac{\p}{\p\phi_{\ast}}\bigg(\zeta(\eta',\phi')\Big(1-\chi(\phi_{\ast})\Big)
\frac{1}{\sin\phi_{\ast}}
F(\eta')\cos\phi_{\ast}\bigg)\v(\eta',\phi_{\ast})\ud{\phi_{\ast}},
\end{eqnarray}
which further implies
\begin{eqnarray}
\abs{\tilde\a_2}&\leq&\frac{C\e}{\d^2}\lnnm{\v}.
\end{eqnarray}
Since we can use substitution to show
\begin{eqnarray}
\int_0^{\eta}\frac{1}{\sin\phi'}\exp(-G_{\eta,\eta'})\ud{\eta'}\leq 1,
\end{eqnarray}
we have
\begin{eqnarray}
\abs{I_2}&\leq&C\bigg(\frac{1}{\d}+\frac{\e}{\d^2}\bigg)\bigg(\lnnm{\v}+\lnnm{S}\bigg)\int_0^{\eta}\frac{1}{\sin\phi'}\exp(-G_{\eta,\eta'})\ud{\eta'}\\
&\leq&C\bigg(\frac{1}{\d}+\frac{\e}{\d^2}\bigg)\bigg(\lnnm{\v}+\lnnm{S}\bigg).\no
\end{eqnarray}
\ \\
Step 3: Estimate of $I_3$ for $0\leq\sin\phi\leq\d_0$, $\chi(\phi_{\ast})=1$ and $\sqrt{\e\eta'}\geq\sin\phi'$.\\
Based on (\ref{pt 01}), this implies
\begin{eqnarray}
\zeta(\eta',\phi')\leq C\sqrt{\e\eta'}.\no
\end{eqnarray}
Then combining this with (\ref{pt 02}), we can directly obtain
\begin{eqnarray}
\int_{-\pi}^{\pi}\frac{\zeta(\eta',\phi')}{\zeta(\eta',\phi_{\ast})}\chi(\phi_{\ast})
\a(\eta',\phi_{\ast})\ud{\phi_{\ast}}&\leq&C\int_{-\d}^{\d}
\a(\eta',\phi_{\ast})\ud{\phi_{\ast}}\leq C\d\lnnm{\a}.
\end{eqnarray}
Hence, we have
\begin{eqnarray}
\abs{I_3}&\leq&C\d\lnnm{\a}\int_0^{\eta}\frac{1}{\sin\phi'}\exp(-G_{\eta,\eta'})\ud{\eta'}\leq C\d\lnnm{\a}.
\end{eqnarray}
\ \\
Step 4: Estimate of $I_4$ for $0\leq\sin\phi\leq\d_0$, $\chi(\phi_{\ast})=1$, $\sqrt{\e\eta'}\leq\sin\phi'$ and $\sin^2\phi\leq\e(\eta-\eta')$.\\
Based on (\ref{pt 01}), this implies
\begin{eqnarray}
\zeta(\eta',\phi')\leq C\sin\phi'.
\end{eqnarray}
Based on (\ref{pt 03}), we have
\begin{eqnarray}
-G_{\eta,\eta'}=-\int_{\eta'}^{\eta}\frac{1}{\sin\phi'(y)}\ud{y}&\leq&-\frac{\eta-\eta'}{2\sqrt{\e(\eta-\eta')}}\leq-C\sqrt{\frac{\eta-\eta'}{\e}}.
\end{eqnarray}
Hence, considering $\zeta(\eta',\phi_{\ast})\geq\sqrt{\e\eta'}$, we know
\begin{eqnarray}
\abs{I_4}&\leq&C\int_0^{\eta}\bigg(\int_{-\pi}^{\pi}\frac{\zeta(\eta',\phi')}{\zeta(\eta',\phi_{\ast})}\chi(\phi_{\ast})
\a(\eta',\phi_{\ast})\ud{\phi_{\ast}}\bigg)
\frac{1}{\sin\phi'}\exp(-G_{\eta,\eta'})\ud{\eta'}\\
&\leq&C\int_0^{\eta}\bigg(\int_{-\d}^{\d}\frac{1}{\zeta(\eta',\phi_{\ast})}
\a(\eta',\phi_{\ast})\ud{\phi_{\ast}}\bigg)
\frac{\zeta(\eta',\phi')}{\sin\phi'}\exp(-G_{\eta,\eta'})\ud{\eta'}\no\\
&\leq&C\lnnm{\a}\int_0^{\eta}\bigg(\int_{-\d}^{\d}\frac{1}{\zeta(\eta',\phi_{\ast})}
\ud{\phi_{\ast}}\bigg)
\frac{\sin\phi'}{\sin\phi'}\exp(-G_{\eta,\eta'})\ud{\eta'}\no\\
&\leq&C\d\lnnm{\a}\int_0^{\eta}\frac{1}{\sqrt{\e\eta'}}\exp(-G_{\eta,\eta'})\ud{\eta'}\no\\
&\leq&C\d\lnnm{\a}\int_0^{\eta}\frac{1}{\sqrt{\e\eta'}}\exp\bigg(-C\sqrt{\frac{\eta-\eta'}{\e}}\bigg)\ud{\eta'}\no
\end{eqnarray}
Define $z=\dfrac{\eta'}{\e}$, which implies $\ud{\eta'}=\e\ud{z}$. Substituting this into above integral, we have
\begin{eqnarray}
\abs{I_4}&\leq&C\d\lnnm{\a}\int_0^{\frac{\eta}{\e}}\frac{1}{\sqrt{z}}\exp\bigg(-C\sqrt{\frac{\eta}{\e}-z}\bigg)\ud{z}\\
&=&C\d\lnnm{\a}\Bigg(\int_0^{1}\frac{1}{\sqrt{z}}\exp\bigg(-C\sqrt{\frac{\eta}{\e}-z}\bigg)\ud{z}
+\int_1^{\frac{\eta}{\e}}\frac{1}{\sqrt{z}}\exp\bigg(-C\sqrt{\frac{\eta}{\e}-z}\bigg)\ud{z}\Bigg).\no
\end{eqnarray}
We can estimate these two terms separately.
\begin{eqnarray}
\int_0^{1}\frac{1}{\sqrt{z}}\exp\bigg(-C\sqrt{\frac{\eta}{\e}-z}\bigg)\ud{z}&\leq&\int_0^{1}\frac{1}{\sqrt{z}}\ud{z}=2.
\end{eqnarray}
\begin{eqnarray}
\int_1^{\frac{\eta}{\e}}\frac{1}{\sqrt{z}}\exp\bigg(-C\sqrt{\frac{\eta}{\e}-z}\bigg)\ud{z}&\leq&\int_1^{\frac{\eta}{\e}}\exp\bigg(-C\sqrt{\frac{\eta}{\e}-z}\bigg)\ud{z}
\overset{t^2=\frac{\eta}{\e}-z}{\leq}2\int_0^{\infty}t\ue^{-Ct}\ud{t}<\infty.
\end{eqnarray}
Hence, we know
\begin{eqnarray}
\abs{I_4}&\leq&C\d\lnnm{\a}.
\end{eqnarray}
\ \\
Step 5: Estimate of $I_5$ for $0\leq\sin\phi\leq\d_0$, $\chi(\phi_{\ast})=1$, $\sqrt{\e\eta'}\leq\sin\phi'$ and $\sin^2\phi\geq\e(\eta-\eta')$.\\
Based on (\ref{pt 01}), this implies
\begin{eqnarray}
\zeta(\eta',\phi')\leq C\sin\phi'.\no
\end{eqnarray}
Based on (\ref{pt 03}), we have
\begin{eqnarray}
-G_{\eta,\eta'}=-\int_{\eta'}^{\eta}\frac{1}{\sin\phi'(y)}\ud{y}&\leq&-\frac{C(\eta-\eta')}{\sin\phi}.
\end{eqnarray}
Hence, we have
\begin{eqnarray}
\abs{I_5}&\leq& C\lnnm{\a}\int_0^{\eta}\bigg(\int_{-\d}^{\d}\frac{1}{\zeta(\eta',\phi_{\ast})}
\ud{\phi_{\ast}}\bigg)
\exp\left(-\frac{C(\eta-\eta')}{\sin\phi}\right)\ud{\eta'}
\end{eqnarray}
Here, we use a different way to estimate the inner integral. We use substitution to find
\begin{eqnarray}
\int_{-\d}^{\d}\frac{1}{\zeta(\eta',\phi_{\ast})}
\ud{\phi_{\ast}}
&=&\int_{-\d}^{\d}\frac{1}{\bigg(\rk^2-(\rk-\e\eta')^2\cos\phi_{\ast}^2\bigg)^{1/2}}
\ud{\phi_{\ast}}\\
&\overset{\sin\phi_{\ast}\ small}{\leq}&C\int_{-\d}^{\d}\frac{\cos\phi_{\ast}}{\bigg(\rk^2-(\rk-\e\eta')^2\cos\phi_{\ast}^2\bigg)^{1/2}}
\ud{\phi_{\ast}}\no\\
&=&C\int_{-\d}^{\d}\frac{\cos\phi_{\ast}}{\bigg(\rk^2-(\rk-\e\eta')^2+(\rk-\e\eta')^2\sin\phi_{\ast}^2\bigg)^{1/2}}
\ud{\phi_{\ast}}\no\\
&\overset{y=\sin\phi_{\ast}}{=}&C\int_{-\d}^{\d}\frac{1}{\bigg(\rk^2-(\rk-\e\eta')^2+(\rk-\e\eta')^2y^2\bigg)^{1/2}}
\ud{y}.\no
\end{eqnarray}
Define
\begin{eqnarray}
p&=&\sqrt{\rk^2-(\rk-\e\eta')^2}=\sqrt{2\rk\e\eta'-\e^2\eta'^2}\leq C\sqrt{\e\eta'},\\
q&=&\rk-\e\eta'\geq C,\\
r&=&\frac{p}{q}\leq C\sqrt{\e\eta'}.
\end{eqnarray}
Then we have
\begin{eqnarray}
\int_{-\d}^{\d}\frac{1}{\zeta(\eta',\phi_{\ast})}\ud{\phi_{\ast}}&\leq&C\int_{-\d}^{\d}\frac{1}{(p^2+q^2y^2)^{1/2}}\ud{y}\\
&\leq&C\int_{-2}^{2}\frac{1}{(p^2+q^2y^2)^{1/2}}\ud{y}\leq C\int_{-2}^{2}\frac{1}{(r^2+y^2)^{1/2}}\ud{y}\no\\
&\leq&C\int_{0}^{2}\frac{1}{(r^2+y^2)^{1/2}}\ud{y}=\bigg(\ln(y+\sqrt{r^2+y^2})-\ln(r)\bigg)\bigg|_0^{2}\no\\
&\leq&C\bigg(\ln(2+\sqrt{r^2+4})-\ln{r}\bigg)\leq C\bigg(1+\ln(r)\bigg)\no\\
&\leq&C\bigg(1+\abs{\ln(\e)}+\abs{\ln(\eta')}\bigg).\no
\end{eqnarray}
Hence, we know
\begin{eqnarray}
\abs{I_5}&\leq&C\lnnm{\a}\int_0^{\eta}\bigg(1+\abs{\ln(\e)}+\abs{\ln(\eta')}\bigg)
\exp\left(-\frac{C(\eta-\eta')}{\sin\phi}\right)\ud{\eta'}
\end{eqnarray}
We may directly compute
\begin{eqnarray}
\abs{\int_0^{\eta}\bigg(1+\abs{\ln(\e)}\bigg)
\exp\left(-\frac{C(\eta-\eta')}{\sin\phi}\right)\ud{\eta'}}\leq C\sin\phi(1+\abs{\ln(\e)}).
\end{eqnarray}
Hence, we only need to estimate
\begin{eqnarray}
\abs{\int_0^{\eta}\abs{\ln(\eta')}
\exp\left(-\frac{C(\eta-\eta')}{\sin\phi}\right)\ud{\eta'}}.
\end{eqnarray}
If $\eta\leq 2$, using Cauchy's inequality, we have
\begin{eqnarray}
\abs{\int_0^{\eta}\abs{\ln(\eta')}
\exp\left(-\frac{C(\eta-\eta')}{\sin\phi}\right)\ud{\eta'}}
&\leq&\bigg(\int_0^{\eta}\ln^2(\eta')\ud{\eta'}\bigg)^{\frac{1}{2}}\bigg(\int_0^{\eta}
\exp\left(-\frac{2C(\eta-\eta')}{\sin\phi}\right)\ud{\eta'}\bigg)^{\frac{1}{2}}\\
&\leq&\bigg(\int_0^{2}\ln^2(\eta')\ud{\eta'}\bigg)^{\frac{1}{2}}\bigg(\int_0^{\eta}
\exp\left(-\frac{2C(\eta-\eta')}{\sin\phi}\right)\ud{\eta'}\bigg)^{\frac{1}{2}}\no\\
&\leq&\sqrt{\sin\phi}.\no
\end{eqnarray}
If $\eta\geq 2$, we decompose and apply Cauchy's inequality to obtain
\begin{eqnarray}
&&\abs{\int_0^{\eta}\abs{\ln(\eta')}
\exp\left(-\frac{C(\eta-\eta')}{\sin\phi}\right)\ud{\eta'}}\\
&\leq&\abs{\int_0^{2}\abs{\ln(\eta')}
\exp\left(-\frac{C(\eta-\eta')}{\sin\phi}\right)\ud{\eta'}}+\abs{\int_2^{\eta}\ln(\eta')
\exp\left(-\frac{C(\eta-\eta')}{\sin\phi}\right)\ud{\eta'}}\no\\
&\leq&\bigg(\int_0^{2}\ln^2(\eta')\ud{\eta'}\bigg)^{\frac{1}{2}}\bigg(\int_0^{2}
\exp\left(-\frac{2C(\eta-\eta')}{\sin\phi}\right)\ud{\eta'}\bigg)^{\frac{1}{2}}+\ln(L)\abs{\int_2^{\eta}
\exp\left(-\frac{C(\eta-\eta')}{\sin\phi}\right)\ud{\eta'}}\no\\
&\leq&C\bigg(\sqrt{\sin\phi}+\abs{\ln(\e)}\sin\phi\bigg)\leq C\Big(1+\abs{\ln(\e)}\Big)\sqrt{\sin\phi}.\no
\end{eqnarray}
Hence, we have
\begin{eqnarray}
\abs{I_5}\leq C\Big(1+\abs{\ln(\e)}\Big)\sqrt{\d_0}\lnnm{\a}.
\end{eqnarray}
\ \\
Step 6: Estimate of $I_6$ for $\abs{\cos\phi}<\d_0$.\\
We have
\begin{eqnarray}
I_6&=&\frac{1}{2\pi}\int_0^{\eta}\bigg(\int_{-\pi}^{\pi}\frac{\zeta(\eta',\phi')}{\zeta(\eta',\phi_{\ast})}
\a(\eta',\phi_{\ast})\ud{\phi_{\ast}}\bigg)
\frac{1}{\sin\phi'}\exp(-G_{\eta,\eta'})\ud{\eta'}\\
&=&\frac{1}{2\pi}\int_0^{\eta}\bigg(\int_{-\pi}^{\pi}\zeta(\eta',\phi')\Big(1-\chi(\phi_{\ast})\Big)
\frac{\v(\eta',\phi_{\ast})}{\p\eta'}\ud{\phi_{\ast}}\bigg)\frac{1}{\sin\phi'}\exp(-G_{\eta,\eta'})\ud{\eta'}\no\\
&&+\frac{1}{2\pi}\int_0^{\eta}\bigg(\int_{-\pi}^{\pi}\chi(\phi_{\ast})
\frac{\zeta(\eta',\phi')}{\zeta(\eta',\phi_{\ast})}
\a(\eta',\phi_{\ast})\ud{\phi_{\ast}}\bigg)\frac{1}{\sin\phi'}\exp(-G_{\eta,\eta'})\ud{\eta'}.\no
\end{eqnarray}
The first term can be estimated as $I_2$.
\begin{eqnarray}
&&\frac{1}{2\pi}\int_0^{\eta}\bigg(\int_{-\pi}^{\pi}\zeta(\eta',\phi')\Big(1-\chi(\phi_{\ast})\Big)
\frac{\v(\eta',\phi_{\ast})}{\p\eta'}\ud{\phi_{\ast}}\bigg)\frac{1}{\sin\phi'}\exp(-G_{\eta,\eta'})\ud{\eta'}\\
&\leq&C\bigg(\frac{1}{\d}+\frac{\e}{\d^2}\bigg)\bigg(\lnnm{\v}+\lnnm{S}\bigg).\no
\end{eqnarray}
It is easy to check that $\sqrt{\e\eta'}\leq\sin\phi\leq\sin\phi'$ and $\sin^2\phi\geq\e(\eta-\eta')$, so the second term can be estimated as $I_5$.
\begin{eqnarray}
&&\frac{1}{2\pi}\int_0^{\eta}\bigg(\int_{-\pi}^{\pi}\chi(\phi_{\ast})
\frac{\zeta(\eta',\phi')}{\zeta(\eta',\phi_{\ast})}
\a(\eta',\phi_{\ast})\ud{\phi_{\ast}}\bigg)\frac{1}{\sin\phi'}\exp(-G_{\eta,\eta'})\ud{\eta'}\\
&\leq&C\Big(1+\abs{\ln(\e)}\Big)\sqrt{\sin\phi}\sup_{\abs{\sin\phi_{\ast}}\leq\d}\abs{\a(\eta,\phi_{\ast})}\leq C\Big(1+\abs{\ln(\e)}\Big)\sup_{\abs{\sin\phi_{\ast}}\leq\d}\abs{\a(\eta,\phi_{\ast})}.\no
\end{eqnarray}
Note that now we lose the smallness since $\sin\phi\geq\dfrac{1}{2}$, so we need a more detailed analysis. Actually, the value of $\abs{\a}$ for $\abs{\sin\phi}\leq\d$, is covered in $I_2,I_3,I_4,I_5$ and the following $II_2,II_3,II_4$. Therefore, in fact, we get the estimate
\begin{eqnarray}
&&\frac{1}{2\pi}\int_0^{\eta}\bigg(\int_{-\pi}^{\pi}\chi(\phi_{\ast})
\frac{\zeta(\eta',\phi')}{\zeta(\eta',\phi_{\ast})}
\a(\eta',\phi_{\ast})\ud{\phi_{\ast}}\bigg)\frac{1}{\sin\phi'}\exp(-G_{\eta,\eta'})\ud{\eta'}\\
&\leq&C\Big(1+\abs{\ln(\e)}\Big)\bigg(\lnm{p_{\a}}+\lnnm{S_{\a}}\bigg)+C\Big(1+\abs{\ln(\e)}\Big)\bigg(\frac{1}{\d}+\frac{\e}{\d^2}\bigg)\bigg(\lnnm{\v}+\lnnm{S}\bigg)\no\\
&&+C\Big(1+\abs{\ln(\e)}\Big)\bigg(\d+\Big(1+\abs{\ln(\e)}\Big)\sqrt{\d_0}\bigg)\lnnm{\a}.\no
\end{eqnarray}
Therefore, we have
\begin{eqnarray}
\\
\abs{I_6}&\leq&C\Big(1+\abs{\ln(\e)}\Big)\bigg(\lnm{p_{\a}}+\lnnm{S_{\a}}\bigg)+C\Big(1+\abs{\ln(\e)}\Big)\bigg(\frac{1}{\d}+\frac{\e}{\d^2}\bigg)\bigg(\lnnm{\v}+\lnnm{S}\bigg)\no\\
&&+C\Big(1+\abs{\ln(\e)}\Big)\bigg(\d+\Big(1+\abs{\ln(\e)}\Big)\sqrt{\d_0}\bigg)\lnnm{\a}.\no
\end{eqnarray}
\ \\
Step 7: Synthesis.\\
Collecting all the terms in previous steps, we have proved
\begin{eqnarray}
\abs{I}&\leq&C\Big(1+\abs{\ln(\e)}\Big)\bigg(\lnm{p_{\a}}+\lnnm{S_{\a}}\bigg)\\
&&+\frac{C}{\d_0^2}\bigg(\zeta\lnm{\frac{\p p}{\p\phi}}+\lnnm{S}+\lnnm{\zeta\frac{\p S}{\p\eta}}+\lnnm{\v}\bigg)\no\\
&&+C\Big(1+\abs{\ln(\e)}\Big)\bigg(\frac{1}{\d}+\frac{\e}{\d^2}\bigg)\bigg(\lnnm{\v}+\lnnm{S}\bigg)\no\\
&&+C\Big(1+\abs{\ln(\e)}\Big)\bigg(\d+\Big(1+\abs{\ln(\e)}\Big)\sqrt{\d_0}\bigg)\lnnm{\a}.\no
\end{eqnarray}

\subsection{Region II: $\sin\phi<0$ and $\abs{E(\eta,\phi)}\leq \ue^{-V(L)}$}

We consider
\begin{eqnarray}
\k[p_{\a}]&=&p_{\a}\Big(\phi'(\eta,\phi;0)\Big)\exp(-G_{L,0}-G_{L,\eta})\\
\t[\tilde\a+S_{\a}]&=&\int_0^{L}\frac{(\tilde\a+S)\Big(\eta',\phi'(\eta,\phi;\eta')\Big)}{\sin\Big(\phi'(\eta,\phi;\eta')\Big)}
\exp(-G_{L,\eta'}-G_{L,\eta})\ud{\eta'}\\
&&+\int_{\eta}^{L}\frac{(\tilde\a+S)\Big(\eta',\rr[\phi'(\eta,\phi;\eta')]\Big)}{\sin\Big(\phi'(\eta,\phi;\eta')\Big)}\exp(-G_{\eta',\eta})\ud{\eta'}\nonumber.
\end{eqnarray}
Based on \cite[Lemma 4.7, Lemma 4.8]{AA003},
we can directly obtain
\begin{eqnarray}
\abs{\k[p_{\a}]}&\leq&\lnm{p_{\a}},\\
\abs{\t[S_{\a}]}&\leq&\lnnm{S_{\a}}.
\end{eqnarray}
Hence, we only need to estimate
\begin{eqnarray}\label{mild 2}
II=\t[\tilde\a+S_{\a}]&=&\int_0^{L}\frac{\tilde\a\Big(\eta',\phi'(\eta,\phi;\eta')\Big)}{\sin\Big(\phi'(\eta,\phi;\eta')\Big)}
\exp(-G_{L,\eta'}-G_{L,\eta})\ud{\eta'}\\
&&+\int_{\eta}^{L}\frac{\tilde\a\Big(\eta',\rr[\phi'(\eta,\phi;\eta')]\Big)}{\sin\Big(\phi'(\eta,\phi;\eta')\Big)}\exp(-G_{\eta',\eta})\ud{\eta'}\nonumber.
\end{eqnarray}
In particular,
since the integral $\displaystyle\int_0^{\eta}\cdots$ can be estimated as in Region I, so we only need to estimate the integral $\displaystyle\int_{\eta}^L\cdots$. Also, noting that fact that \begin{eqnarray}
\exp(-G_{L,\eta'}-G_{L,\eta})\leq \exp(-G_{\eta',\eta}),
\end{eqnarray}
we only need to estimate
\begin{eqnarray}
\int_{\eta}^{L}\frac{\tilde\a\Big(\eta',\rr[\phi'(\eta,\phi;\eta')]\Big)}{\sin\Big(\phi'(\eta,\phi;\eta')\Big)}\exp(-G_{\eta',\eta})\ud{\eta'}.
\end{eqnarray}
Here the proof is almost identical to that in Region I, so we only point out the key differences.\\
\ \\
Step 0: Preliminaries.\\
We need to update one key result. For $0\leq\eta\leq\eta'$,
\begin{eqnarray}
\sin\phi'&=&\sqrt{1-\cos^2\phi'}=\sqrt{1-\bigg(\frac{\rk-\e\eta}{\rk-\e\eta'}\bigg)^2\cos^2\phi}\\
&=&\frac{\sqrt{(\rk-\e\eta')^2\sin^2\phi+(2\rk-\e\eta-\e\eta')(\e\eta'-\e\eta)\cos^2\phi}}{\rk-\e\eta'}\no\\
&\leq&\abs{\sin\phi}.\no
\end{eqnarray}
Then we have
\begin{eqnarray}\label{pt 04}
-\int_{\eta}^{\eta'}\frac{1}{\sin\phi'(y)}\ud{y}&\leq&-\frac{\eta'-\eta}{\abs{\sin\phi}}.
\end{eqnarray}
In the following, we will divide the estimate of $II$ into several cases based on the value of $\sin\phi$, $\abs{\cos\phi}$, $\sin\phi'$ and $\e\eta'$. We write
\begin{eqnarray}
II&=&\int_{\eta}^L\id_{\{\sin\phi\leq-\d_0\}}\id_{\{\abs{\cos\phi}\geq\d_0\}}+\int_{\eta}^L\id_{\{-\d_0\leq\sin\phi\leq0\}}\id_{\{\chi(\phi_{\ast})<1\}}\\
&&+\int_{\eta}^L\id_{\{-\d_0\leq\sin\phi\leq0\}}\id_{\{\chi(\phi_{\ast})=1\}}\id_{\{\sqrt{\e\eta'}\geq\sin\phi'\}}
+\int_{\eta}^L\id_{\{-\d_0\leq\sin\phi\leq0\}}\id_{\{\chi(\phi_{\ast})=1\}}\id_{\{\sqrt{\e\eta'}\leq\sin\phi'\}}\no\\
&&+\int_{\eta}^L\id_{\{\abs{\cos\phi}\leq\d_0\}}\no\\
&=&II_1+II_2+II_3+II_4+II_5.\no
\end{eqnarray}
\ \\
Step 1: Estimate of $II_1$ for $\sin\phi\leq-\d_0$.\\
We first estimate $\sin\phi'$. Along the characteristics, we know
\begin{eqnarray}
\ue^{-V(\eta')}\cos\phi'=\ue^{-V(\eta)}\cos\phi,
\end{eqnarray}
which implies
\begin{eqnarray}
\cos\phi'&=&\ue^{V(\eta')-V(\eta)}\cos\phi\leq \ue^{V(L)-V(0)}\cos\phi= \ue^{V(L)-V(0)}\sqrt{1-\d_0^2}.
\end{eqnarray}
We can further deduce that
\begin{eqnarray}
\cos\phi'\leq \bigg(1-\frac{\e^{\frac{1}{2}}}{\rk}\bigg)^{-1}\sqrt{1-\d_0^2}.
\end{eqnarray}
Then we have
\begin{eqnarray}
\sin\phi'\geq\sqrt{1-\bigg(1-\frac{\e^{\frac{1}{2}}}{\rk}\bigg)^{-2}(1-\d_0^2)}\geq \d_0-\e^{\frac{1}{4}}>\frac{\d_0}{2},
\end{eqnarray}
when $\e$ is sufficiently small.

Similar to Region I, we will use two formulations to handle different terms and we will decompose $\v=\v_1+\v_2$.

Using Formulation I, we rewrite the $\v_1$ equation along the characteristics as
\begin{eqnarray}\label{mt 03}
\v_1(\eta,\phi)&=&p\Big(\phi'(0)\Big)\exp(-G_{L,0}-G_{L,\eta})\\
&&+\int_0^{L}\frac{\bar\v(\eta')}{\sin\Big(\phi'(\eta')\Big)}
\exp(-G_{L,\eta'}-G_{L,\eta})\ud{\eta'}+\int_{\eta}^{L}\frac{\bar\v(\eta')}{\sin\Big(\phi'(\eta')\Big)}\exp(-G_{\eta',\eta})\ud{\eta'}\no
\end{eqnarray}
where $(\eta',\phi')$ and $(\eta,\phi)$ are on the same characteristic with $\sin\phi'\geq0$.
Then taking $\eta$ derivative on both sides of (\ref{mt 03}) yields
\begin{eqnarray}
\frac{\p\v_1}{\p\eta}&=&X_1+X_2+X_3+X_4+X_5+X_6+X_7,
\end{eqnarray}
where
\begin{eqnarray}
X_1&=&\frac{\p p\Big(\phi'(0)\Big)}{\p\eta}\exp(-G_{L,0}-G_{L,\eta}),\\
X_2&=&-p\Big(\phi'(0)\Big)\exp(-G_{L,0}-G_{L,\eta})\bigg(\frac{\p G_{L,0}}{\p\eta}+\frac{\p G_{L,\eta}}{\p\eta}\bigg),\\
X_3&=&-\int_0^{L}\bar\v(\eta')\frac{\cos\Big(\phi'(\eta')\Big)}{\sin^2\Big(\phi'(\eta')\Big)}\frac{\p\phi'(\eta')}{\p\eta}\exp(-G_{L,\eta'}-G_{L,\eta})\ud{\eta'},\\
X_4&=&-\int_0^{L}\frac{\bar\v(\eta')}{\sin\Big(\phi'(\eta')\Big)}\exp(-G_{L,\eta'}-G_{L,\eta})\bigg(\frac{\p G_{L,\eta'}}{\p\eta}+\frac{\p G_{L,\eta}}{\p\eta}\bigg)\ud{\eta'},\\
X_5&=&-\int_{\eta}^{L}\bar\v(\eta')\frac{\cos\Big(\phi'(\eta')\Big)}{\sin^2\Big(\phi'(\eta')\Big)}\frac{\p\phi'(\eta')}{\p\eta}\exp(-G_{\eta',\eta})\ud{\eta'},\\
X_6&=&-\int_{\eta}^{L}\frac{\bar\v(\eta')}{\sin\Big(\phi'(\eta')\Big)}\exp(-G_{\eta',\eta})\frac{\p G_{\eta',\eta}}{\p\eta}\ud{\eta'},\\
X_7&=&-\frac{\bar\v(\eta)}{\sin(\phi)}.
\end{eqnarray}
We need to estimate each term. The estimates are standard, so we only list the results:
\begin{eqnarray}
\abs{X_1}&\leq&\frac{C}{\d_0}\lnm{\frac{\p p}{\p\phi}},\\
\abs{X_2}&\leq&\frac{C}{\d_0}\lnm{p},\\
\abs{X_3}&\leq&\frac{C}{\d_0}\lnnm{\v},\\
\abs{X_4}&\leq&\frac{C}{\d_0}\lnnm{\v},\\
\abs{X_5}&\leq&\frac{C}{\d_0}\lnnm{\v},\\
\abs{X_6}&\leq&\frac{C}{\d_0}\lnnm{\v},\\
\abs{X_7}&\leq&\frac{C}{\d_0}\lnnm{\v}.
\end{eqnarray}
In total, we have
\begin{eqnarray}
\abs{\frac{\p\v_1}{\p\eta}}\leq \frac{C}{\d_0}\bigg(\lnm{p}+\lnm{\frac{\p p}{\p\phi}}+\lnnm{\v}\bigg).
\end{eqnarray}

Using Formulation II, we rewrite the $\v_2$ equation along the characteristics as
\begin{eqnarray}\label{mt 03'}
\\
\v_2(\eta,\phi)&=&\int_{\phi_{\ast}}^{\phi^{\ast}}\frac{S\Big(\eta'(\phi'),\phi'\Big)}{F\Big(\eta'(\phi')\Big)\cos(\phi')}
\exp(-H_{\phi^{\ast},\phi'}-H_{-\phi^{\ast},\phi})\ud{\phi'}
+\int_{\phi}^{-\phi^{\ast}}\frac{S\Big(\eta'(\phi'),\phi'\Big)}{F\Big(\eta'(\phi')\Big)\cos(\phi')}\exp(-H_{\phi',\phi})\ud{\phi'}\no
\end{eqnarray}
where $(\eta',\phi')$, $(0,\phi_{\ast})$, $(L,\phi^{\ast})$, $(L,-\phi^{\ast})$ and $(\eta,\phi)$ are on the same characteristic with $\sin\phi'\geq0$ and $\phi^{\ast}\geq0$.
Then taking $\eta$ derivative on both sides of (\ref{mt 03'}) yields
\begin{eqnarray}
\frac{\p\v_2}{\p\eta}&=&Y_1+Y_2+Y_3+Y_4+Y_5+Y_6+Y_7+Y_8,
\end{eqnarray}
where
\begin{eqnarray}
Y_1&=&\frac{S(L,\phi^{\ast})}{F(L)\cos(\phi^{\ast})}\exp(-H_{-\phi^{\ast},\phi})\frac{\p\phi^{\ast}}{\p\eta}-\frac{S(0,\phi_{\ast})}{F(0)\cos(\phi_{\ast})}
\exp(-H_{\phi^{\ast},\phi_{\ast}}-H_{-\phi^{\ast},\phi})\frac{\p\phi_{\ast}}{\p\eta},\no\\
Y_2&=&-\int_{\phi_{\ast}}^{\phi^{\ast}}S\Big(\eta'(\phi'),\phi'\Big)\frac{1}{F^2\Big(\eta'(\phi')\Big)\cos(\phi')}\frac{\p F\Big(\eta'(\phi')\Big)}{\p\eta}
\exp(-H_{\phi^{\ast},\phi'}-H_{-\phi^{\ast},\phi})\ud{\phi'},\no\\
Y_3&=&-\int_{\phi_{\ast}}^{\phi^{\ast}}\frac{S\Big(\eta'(\phi'),\phi'\Big)}{F\Big(\eta'(\phi')\Big)\cos(\phi')}
\exp(-H_{\phi^{\ast},\phi'}-H_{-\phi^{\ast},\phi})\bigg(\frac{\p H_{\phi^{\ast},\phi'}}{\p\eta}+\frac{\p H_{-\phi^{\ast},\phi}}{\p\eta}\bigg)\ud{\phi'},\\
Y_4&=&\int_{\phi_{\ast}}^{\phi^{\ast}}\frac{\p_{\eta'}S\Big(\eta'(\phi'),\phi'\Big)}{F\Big(\eta'(\phi')\Big)\cos(\phi')}\frac{\p \eta'(\phi')}{\p\eta}
\exp(-H_{\phi^{\ast},\phi'}-H_{-\phi^{\ast},\phi})\ud{\phi'},\\
Y_5&=&-\frac{S(L,-\phi^{\ast})}{F(L)\cos(-\phi^{\ast})}\exp(-H_{-\phi^{\ast},\phi})\frac{\p \phi^{\ast}}{\p\eta},\\
Y_6&=&-\int_{\phi}^{-\phi^{\ast}}S\Big(\eta'(\phi'),\phi'\Big)\frac{1}{F^2\Big(\eta'(\phi')\Big)\cos(\phi')}\frac{\p F\Big(\eta'(\phi')\Big)}{\p\eta}\exp(-H_{\phi',\phi})\ud{\phi'},\\
Y_7&=&-\int_{\phi}^{-\phi^{\ast}}\frac{S\Big(\eta'(\phi'),\phi'\Big)}{F\Big(\eta'(\phi')\Big)\cos(\phi')}\exp(-H_{\phi',\phi})\frac{\p H_{\phi',\phi}}{\p\eta}\ud{\phi'},\\
Y_8&=&\int_{\phi}^{-\phi^{\ast}}\frac{\p_{\eta'}S\Big(\eta'(\phi'),\phi'\Big)}{F\Big(\eta'(\phi')\Big)\cos(\phi')}\frac{\p \eta'(\phi')}{\p\eta}\exp(-H_{\phi',\phi})\ud{\phi'}.
\end{eqnarray}
We need to estimate each term. The estimates are standard, so we only list the results:
\begin{eqnarray}
\abs{Y_1}&\leq&\frac{C}{\d_0}\lnnm{S},\\
\abs{Y_2}&\leq&\frac{C}{\d_0}\lnnm{S},\\
\abs{Y_3}&\leq&\frac{C}{\d_0}\lnnm{S},\\
\abs{Y_4}&\leq&\frac{C}{\d_0}\lnnm{\frac{\p S}{\p\eta}},\\
\abs{Y_5}&\leq&\frac{C}{\d_0}\lnnm{S},\\
\abs{Y_6}&\leq&\frac{C}{\d_0}\lnnm{S},\\
\abs{Y_7}&\leq&\frac{C}{\d_0}\lnnm{S},\\
\abs{Y_8}&\leq&\frac{C}{\d_0}\lnnm{\frac{\p S}{\p\eta}},\\
\end{eqnarray}
In total, we have
\begin{eqnarray}
\abs{\frac{\p\v_2}{\p\eta}}\leq \frac{C}{\d_0}\bigg(\lnnm{S}+\lnnm{\frac{\p S}{\p\eta}}\bigg).
\end{eqnarray}
Combining all above, noting that $\zeta\geq\d_0$, we have
\begin{eqnarray}
\abs{II_1}\leq \frac{C}{\d_0^2}\bigg(\lnm{p}+\lnm{\zeta\frac{\p p}{\p\phi}}+\lnnm{S}+\lnnm{\zeta\frac{\p S}{\p\eta}}+\lnnm{\v}\bigg).
\end{eqnarray}
\ \\
Step 2: Estimate of $II_2$ for $-\d_0\leq\sin\phi\leq0$ and $\chi(\phi_{\ast})<1$.\\
This is similar to the estimate of $I_2$ based on the integral
\begin{eqnarray}
\int_{\eta}^{L}\frac{1}{\sin\phi'}\exp(-G_{\eta',\eta})\ud{\eta'}\leq 1.
\end{eqnarray}
Then we have
\begin{eqnarray}
\abs{II_2}
&\leq&\bigg(\frac{1}{\d}+\frac{\e}{\d^2}\bigg)\bigg(\lnnm{\v}+\lnnm{S}\bigg).
\end{eqnarray}
\ \\
Step 3: Estimate of $II_3$ for $-\d_0\leq\sin\phi\leq0$, $\chi(\phi_{\ast})=1$ and $\sqrt{\e\eta'}\geq\sin\phi'$.\\
This is identical to the estimate of $I_4$, we have
\begin{eqnarray}
\abs{II_3}&\leq&C\d\lnnm{\a}.
\end{eqnarray}
\ \\
Step 4: Estimate of $II_4$ for $-\d_0\leq\sin\phi\leq0$, $\chi(\phi_{\ast})=1$ and $\sqrt{\e\eta'}\leq\sin\phi'$.\\
This step is different. We do not need to further decompose the cases.
Based on (\ref{pt 04}), we have,
\begin{eqnarray}
-G_{\eta,\eta'}&\leq&-\frac{\eta'-\eta}{\abs{\sin\phi}}.
\end{eqnarray}
Then following the same argument in estimating $I_5$, we obtain
\begin{eqnarray}
\abs{II_4}&\leq&C\lnnm{\a}\int_{\eta}^{L}\bigg(1+\abs{\ln(\e)}+\abs{\ln(\eta')}\bigg)
\exp\left(-\frac{\eta'-\eta}{\abs{\sin\phi}}\right)\ud{\eta'}
\end{eqnarray}
If $\eta\geq 2$, we directly obtain
\begin{eqnarray}
\abs{\int_{\eta}^{L}\abs{\ln(\eta')}
\exp\left(-\frac{\eta'-\eta}{\abs{\sin\phi}}\right)\ud{\eta'}}&\leq& \abs{\int_{2}^{L}\ln(\eta')
\exp\left(-\frac{\eta'-\eta}{\abs{\sin\phi}}\right)\ud{\eta'}}\\
&\leq&\ln(2)\abs{\int_{2}^{L}
\exp\left(-\frac{\eta'-\eta}{\abs{\sin\phi}}\right)\ud{\eta'}}\no\\
&\leq&C\sqrt{\abs{\sin\phi}}.\no
\end{eqnarray}
If $\eta\leq 2$, we decompose as
\begin{eqnarray}
&&\abs{\int_{\eta}^{L}\abs{\ln(\eta')}
\exp\left(-\frac{\eta'-\eta}{\abs{\sin\phi}}\right)\ud{\eta'}}\\
&\leq&\abs{\int_{\eta}^{2}\abs{\ln(\eta')}
\exp\left(-\frac{\eta'-\eta}{\abs{\sin\phi}}\right)\ud{\eta'}}+\abs{\int_{2}^{L}\abs{\ln(\eta')}
\exp\left(-\frac{\eta'-\eta}{\abs{\sin\phi}}\right)\ud{\eta'}}.\no
\end{eqnarray}
The second term is identical to the estimate in $\eta\geq2$. We apply Cauchy's inequality to the first term
\begin{eqnarray}
\abs{\int_{\eta}^{2}\abs{\ln(\eta')}
\exp\left(-\frac{\eta'-\eta}{\abs{\sin\phi}}\right)\ud{\eta'}}
&\leq&\bigg(\int_{\eta}^{2}\ln^2(\eta')\ud{\eta'}\bigg)^{\frac{1}{2}}\bigg(\int_{\eta}^{2}
\exp\left(-\frac{2(\eta'-\eta)}{\abs{\sin\phi}}\right)\ud{\eta'}\bigg)^{\frac{1}{2}}\\
&\leq&\bigg(\int_0^{2}\ln^2(\eta')\ud{\eta'}\bigg)^{\frac{1}{2}}\bigg(\int_{\eta}^{2}
\exp\left(-\frac{2(\eta'-\eta)}{\abs{\sin\phi}}\right)\ud{\eta'}\bigg)^{\frac{1}{2}}\no\\
&\leq&C\sqrt{\abs{\sin\phi}}.\no
\end{eqnarray}
Hence, we have
\begin{eqnarray}
\abs{II_4}\leq C(1+\abs{\ln(\e)})\sqrt{\d_0}\lnnm{\a}.
\end{eqnarray}
Step 5: Estimate of $II_5$ for $\abs{\cos\phi}<\d_0$.\\
This is identical to the estimate of $I_6$, we have
\begin{eqnarray}
\\
\abs{II_5}&\leq&C\Big(1+\abs{\ln(\e)}\Big)\bigg(\lnm{p_{\a}}+\lnnm{S_{\a}}\bigg)+C\Big(1+\abs{\ln(\e)}\Big)\bigg(\frac{1}{\d}+\frac{\e}{\d^2}\bigg)\bigg(\lnnm{\v}+\lnnm{S}\bigg)\no\\
&&+C\Big(1+\abs{\ln(\e)}\Big)\bigg(\d+\Big(1+\abs{\ln(\e)}\Big)\sqrt{\d_0}\bigg)\lnnm{\a}.\no
\end{eqnarray}
\ \\
Step 6: Synthesis.\\
\begin{eqnarray}
\abs{II}&\leq&C\Big(1+\abs{\ln(\e)}\Big)\bigg(\lnm{p_{\a}}+\lnnm{S_{\a}}\bigg)\\
&&+\frac{C}{\d_0^2}\bigg(\lnm{p}+\lnm{\zeta\frac{\p p}{\p\phi}}+\lnnm{S}+\lnnm{\zeta\frac{\p S}{\p\eta}}+\lnnm{\v}\bigg)\no\\
&&+C\Big(1+\abs{\ln(\e)}\Big)\bigg(\frac{1}{\d}+\frac{\e}{\d^2}\bigg)\bigg(\lnnm{\v}+\lnnm{S}\bigg)\no\\
&&+C\Big(1+\abs{\ln(\e)}\Big)\bigg(\d+\Big(1+\abs{\ln(\e)}\Big)\sqrt{\d_0}\bigg)\lnnm{\a}.\no
\end{eqnarray}

\subsection{Region III: $\sin\phi<0$ and $\abs{E(\eta,\phi)}\geq \ue^{-V(L)}$}

We consider
\begin{eqnarray}
\k[p_{\a}]&=&p_{\a}\Big(\phi'(\eta,\phi;0)\Big)\exp(-G_{\eta^+,0}-G_{\eta^+,\eta})\\
\t[\tilde\a+S_{\a}]&=&\int_0^{\eta^+}\frac{(\tilde\a+S_{\a})\Big(\eta',\phi'(\eta,\phi;\eta')\Big)}{\sin\Big(\phi'(\eta,\phi;\eta')\Big)}
\exp(-G_{\eta^+,\eta'}-G_{\eta^+,\eta})\ud{\eta'}\\
&&+
\int_{\eta}^{\eta^+}\frac{(\tilde\a+S_{\a})\Big(\eta',\rr[\phi'(\eta,\phi;\eta')]\Big)}{\sin\Big(\phi'(\eta,\phi;\eta')\Big)}\exp(-G_{\eta',\eta})\ud{\eta'}\nonumber.
\end{eqnarray}
Based on \cite[Lemma 4.7, Lemma 4.8]{AA003}, we still have
\begin{eqnarray}
\abs{\k[p_{\a}]}&\leq&\lnm{p_{\a}},\\
\abs{\t[S_{\a}]}&\leq&\lnnm{S_{\a}}.
\end{eqnarray}
Hence, we only need to estimate
\begin{eqnarray}
III=\t[\tilde\a]&=&\int_0^{\eta^+}\frac{\tilde\a\Big(\eta',\phi'(\eta,\phi;\eta')\Big)}{\sin\Big(\phi'(\eta,\phi;\eta')\Big)}
\exp(-G_{\eta^+,\eta'}-G_{\eta^+,\eta})\ud{\eta'}\\
&&+
\int_{\eta}^{\eta^+}\frac{\tilde\a\Big(\eta',\rr[\phi'(\eta,\phi;\eta')]\Big)}{\sin\Big(\phi'(\eta,\phi;\eta')\Big)}\exp(-G_{\eta',\eta})\ud{\eta'}\nonumber.
\end{eqnarray}
Note that $\abs{E(\eta,\phi)}\geq \ue^{-V(L)}$ implies
\begin{eqnarray}
\ue^{-V(\eta)}\cos\phi\geq \ue^{-V(L)}.
\end{eqnarray}
Hence, we can further deduce that
\begin{eqnarray}
\cos\phi&\geq&\ue^{V(\eta)-V(L)}\geq \ue^{V(0)-V(L)}\geq \bigg(1-\frac{\e^{\frac{1}{2}}}{\rk}\bigg).
\end{eqnarray}
Hence, we know
\begin{eqnarray}
\abs{\sin\phi}\leq\sqrt{1-\bigg(1-\frac{\e^{\frac{1}{2}}}{\rk}\bigg)^2}\leq \e^{\frac{1}{4}}.
\end{eqnarray}
Hence, when $\e$ is sufficiently small, we always have
\begin{eqnarray}
\abs{\sin\phi}\leq \e^{\frac{1}{4}}\leq \d_0.
\end{eqnarray}
This means we do not need to bother with the estimate of $\sin\phi\leq-\d_0$ as Step 1 in estimating $I$ and $II$. Also, it is not necessary to discuss the case $\abs{\cos\phi}<\d_0$.

Then the integral $\displaystyle\int_0^{\eta}(\cdots)$ is similar to the argument in Region I, and the integral $\displaystyle\int_{\eta}^{\eta^+}(\cdots)$ is similar to the argument in Region II.
Hence, combining the methods in Region I and Region II, we can show the desired result, i.e.
\begin{eqnarray}
\abs{III}&\leq&C\Big(1+\abs{\ln(\e)}\Big)\bigg(\lnm{p_{\a}}+\lnnm{S_{\a}}\bigg)\\
&&+C\Big(1+\abs{\ln(\e)}\Big)\bigg(\frac{1}{\d}+\frac{\e}{\d^2}\bigg)\bigg(\lnnm{\v}+\lnnm{S}\bigg)\no\\
&&+C\Big(1+\abs{\ln(\e)}\Big)\bigg(\d+\Big(1+\abs{\ln(\e)}\Big)\sqrt{\d_0}\bigg)\lnnm{\a}.\no
\end{eqnarray}

\subsubsection{Estimate of Normal Derivative}

Combining the analysis in these three regions, we have
\begin{eqnarray}
\abs{\a}&\leq&C\Big(1+\abs{\ln(\e)}\Big)\bigg(\lnm{p_{\a}}+\lnnm{S_{\a}}\bigg)\\
&&+\frac{C}{\d_0^2}\bigg(\lnm{p}+\lnm{\zeta\frac{\p p}{\p\phi}}+\lnnm{S}+\lnnm{\zeta\frac{\p S}{\p\eta}}+\lnnm{\v}\bigg)\no\\
&&+C\Big(1+\abs{\ln(\e)}\Big)\bigg(\frac{1}{\d}+\frac{\e}{\d^2}\bigg)\bigg(\lnnm{\v}+\lnnm{S}\bigg)\no\\
&&+C\Big(1+\abs{\ln(\e)}\Big)\bigg(\d+\Big(1+\abs{\ln(\e)}\Big)\sqrt{\d_0}\bigg)\lnnm{\a}.\no
\end{eqnarray}
Taking supremum over all $(\eta,\phi)$, we have
\begin{eqnarray}\label{pt 05}
\lnnm{\a}&\leq&C\Big(1+\abs{\ln(\e)}\Big)\bigg(\lnm{p_{\a}}+\lnnm{S_{\a}}\bigg)\\
&&+\frac{C}{\d_0^2}\bigg(\lnm{p}+\lnm{\zeta\frac{\p p}{\p\phi}}+\lnnm{S}+\lnnm{\zeta\frac{\p S}{\p\eta}}+\lnnm{\v}\bigg)\no\\
&&+C\Big(1+\abs{\ln(\e)}\Big)\bigg(\frac{1}{\d}+\frac{\e}{\d^2}\bigg)\bigg(\lnnm{\v}+\lnnm{S}\bigg)\no\\
&&+C\Big(1+\abs{\ln(\e)}\Big)\bigg(\d+\Big(1+\abs{\ln(\e)}\Big)\sqrt{\d_0}\bigg)\lnnm{\a}.\no
\end{eqnarray}
Then we choose these constants to perform absorbing argument. First we choose $\d=C_0\Big(1+\abs{\ln(\e)}\Big)^{-1}$ for $C_0>0$ sufficiently small such that
\begin{eqnarray}
C\d\leq\frac{1}{4}.
\end{eqnarray}
Then we take $\d_0=C_0\Big(1+\abs{\ln(\e)}\Big)^{-4}$ such that
\begin{eqnarray}
C\Big(1+\abs{\ln(\e)}\Big)^2\sqrt{\d_0}\leq \frac{1}{4}.
\end{eqnarray}
for $\e$ sufficiently small. Note that this mild decay of $\d_0$ with respect to $\e$ also justifies the assumption in Case III that
\begin{eqnarray}
\e^{\frac{1}{4}}\leq \frac{\d_0}{2},
\end{eqnarray}
for $\e$ sufficiently small. Hence, we can absorb all the term related to $\lnnm{\a}$ on the right-hand side of (\ref{pt 05}) to the left-hand side to obtain
\begin{eqnarray}
\lnnm{\a}&\leq&C\abs{\ln(\e)}\bigg(\lnm{p_{\a}}+\lnnm{S_{\a}}\bigg)\\
&&+C\abs{\ln(\e)}^8\bigg(\lnm{p}+\lnm{\zeta\dfrac{\p p}{\p\phi}}+\lnnm{S}+\lnnm{\zeta\dfrac{\p S}{\p\eta}}+\lnnm{\v}\bigg).\no
\end{eqnarray}

\subsection{A Priori Estimate of Derivatives}

In this subsection, we further estimate the normal and velocity derivatives.
\begin{theorem}\label{pt theorem 1}
We have
\begin{eqnarray}
&&\lnnm{\zeta\frac{\p\v}{\p\eta}}+\lnnm{F(\eta)\cos\phi\frac{\p\v}{\p\phi}}\\
&\leq&C\abs{\ln(\e)}^8\bigg(\lnm{p}+\lnm{(\e+\zeta)\dfrac{\p p}{\p\phi}}+\lnnm{S}+\lnnm{\zeta\dfrac{\p S}{\p\eta}}+\lnnm{\v}\bigg).\no
\end{eqnarray}
\end{theorem}
\begin{proof}
We have
\begin{eqnarray}\label{pt 06}
\lnnm{\a}&\leq&C\abs{\ln(\e)}\bigg(\lnm{p_{\a}}+\lnnm{S_{\a}}\bigg)\\
&&+C\abs{\ln(\e)}^8\bigg(\lnm{p}+\lnm{\zeta\dfrac{\p p}{\p\phi}}+\lnnm{S}+\lnnm{\zeta\dfrac{\p S}{\p\eta}}+\lnnm{\v}\bigg).\no
\end{eqnarray}
Taking derivatives on both sides of (\ref{Milne difference problem}) and multiplying $\zeta$, we have
\begin{eqnarray}
p_{\a}&=&\e\cos\phi\frac{\p p}{\p\phi}+p-\bar\v(0)-S(0,\phi),\\
S_{\a}&=&\frac{\p{F}}{\p{\eta}}\zeta\cos\phi\frac{\p\v}{\p\phi}+\zeta\frac{\p S}{\p\eta}.
\end{eqnarray}
Since $\abs{F(\eta)}\leq C\e$ and $\abs{\dfrac{\p F}{\p\eta}}\leq C\e F$, we may directly estimate
\begin{eqnarray}
\lnm{p_{\a}}&\leq&C\bigg(\lnm{p}+\e\lnm{\dfrac{\p p}{\p\phi}}+\lnnm{S}+\lnnm{\v}\bigg),\\
\lnnm{S_{\a}}&\leq&C\bigg(\e\lnnm{F(\eta)\cos\phi\frac{\p\v}{\p\phi}}+\lnnm{\zeta\dfrac{\p S}{\p\eta}}\bigg)
\end{eqnarray}
Then we derive
\begin{eqnarray}
\lnnm{\a}&\leq&C\e\lnnm{F(\eta)\cos\phi\frac{\p\v}{\p\phi}}\\
&&+C\abs{\ln(\e)}^8\bigg(\lnm{p}+\lnm{(\e+\zeta)\dfrac{\p p}{\p\phi}}+\lnnm{S}+\lnnm{\zeta\dfrac{\p S}{\p\eta}}+\lnnm{\v}\bigg).\no
\end{eqnarray}
We know
\begin{eqnarray}
\lnnm{\zeta\frac{\p\v}{\p\eta}}&\leq&C\e\lnnm{F(\eta)\cos\phi\frac{\p\v}{\p\phi}}\\
&&+C\abs{\ln(\e)}^8\bigg(\lnm{p}+\lnm{(\e+\zeta)\dfrac{\p p}{\p\phi}}+\lnnm{S}+\lnnm{\zeta\dfrac{\p S}{\p\eta}}+\lnnm{\v}\bigg).\no
\end{eqnarray}
Considering the equation (\ref{Milne difference problem}), since $\zeta(\eta,\phi)\geq\abs{\sin\phi}$, we have
\begin{eqnarray}
\lnnm{F(\eta)\cos\phi\frac{\p\v}{\p\phi}}&\leq&\lnnm{\sin\phi\frac{\p\v}{\p\eta}}+\lnnm{\v}+\lnnm{\bar\v}+\lnnm{S}\\
&\leq&C\e\lnnm{F(\eta)\cos\phi\frac{\p\v}{\p\phi}}\no\\
&&+C\abs{\ln(\e)}^8\bigg(\lnm{p}+\lnm{(\e+\zeta)\dfrac{\p p}{\p\phi}}+\lnnm{S}+\lnnm{\zeta\dfrac{\p S}{\p\eta}}+\lnnm{\v}\bigg).\no
\end{eqnarray}
Absorbing $\lnnm{F(\eta)\cos\phi\dfrac{\p\v}{\p\phi}}$ into the left-hand side, we obtain
\begin{eqnarray}
\\
\lnnm{F(\eta)\cos\phi\frac{\p\v}{\p\phi}}&\leq&C\abs{\ln(\e)}^8\bigg(\lnm{p}+\lnm{(\e+\zeta)\dfrac{\p p}{\p\phi}}+\lnnm{S}+\lnnm{\zeta\dfrac{\p S}{\p\eta}}+\lnnm{\v}\bigg).\no
\end{eqnarray}
Therefore, we further derive
\begin{eqnarray}
\\
\lnnm{\zeta\frac{\p\v}{\p\eta}}&\leq&C\abs{\ln(\e)}^8\bigg(\lnm{p}+\lnm{(\e+\zeta)\dfrac{\p p}{\p\phi}}+\lnnm{S}+\lnnm{\zeta\dfrac{\p S}{\p\eta}}+\lnnm{\v}\bigg).\no
\end{eqnarray}
\end{proof}

\begin{theorem}\label{pt theorem 2}
For $K_0>0$ sufficiently small, we have
\begin{eqnarray}
&&\lnnm{\ue^{K_0\eta}\zeta\frac{\p\v}{\p\eta}}+\lnnm{\ue^{K_0\eta}F(\eta)\cos\phi\frac{\p\v}{\p\phi}}\\
&\leq&C\abs{\ln(\e)}^8\bigg(\lnm{p}+\lnm{(\e+\zeta)\dfrac{\p p}{\p\phi}}+\lnnm{\ue^{K_0\eta}S}+\lnnm{\ue^{K_0\eta}\zeta\dfrac{\p S}{\p\eta}}+\lnnm{\ue^{K_0\eta}\v}\bigg).\no
\end{eqnarray}
\end{theorem}
\begin{proof}
This proof is almost identical to Theorem \ref{pt theorem 1}. The only difference is that $S_{\a}$ is added by $K_0\a\sin\phi$. When $K_0$ is sufficiently small, we can also absorb them into the left-hand side. Hence, this is obvious.
\end{proof}

\newpage

\section{Diffusive Limit}

\subsection{Analysis of Regular Boundary Layer}

In this subsection, we will justify that the regular boundary layers are all well-defined. We divide it into several steps:\\
\ \\
Step 1: Well-Posedness of $\ub_0$.\\
$\ub_0$ satisfies the $\e$-Milne problem with geometric correction
\begin{eqnarray}
\left\{
\begin{array}{l}
\sin\phi\dfrac{\p \ub_0 }{\p\eta}+F(\eta)\cos\phi\dfrac{\p
\ub_0 }{\p\phi}+\ub_0 -\bub_0 =0,\\\rule{0ex}{1.5em}
\ub_0 (0,\tau,\phi)=\gb(\tau,\phi)-\mathscr{F}_0(\tau)\ \ \text{for}\ \
\sin\phi>0,\\\rule{0ex}{1.5em}
\ub_0 (L,\tau,\phi)=\ub_0 (L,\tau,\rr[\phi]),
\end{array}
\right.
\end{eqnarray}
Therefore, since $\lnm{\gb}\leq C$, by Theorem \ref{Milne theorem 2}, we know
\begin{eqnarray}
\lnnm{\ue^{K_0\eta}\ub_0}\leq C.
\end{eqnarray}
\ \\
Step 2: Tangential Derivatives of $\ub_0$.\\
The $\tau$ derivative $W=\dfrac{\p\ub_0}{\p\tau}$ satisfies
\begin{eqnarray}
\left\{
\begin{array}{l}
\sin\phi\dfrac{\p W}{\p\eta}+F(\eta)\cos\phi\dfrac{\p W}{\p\phi}+W-\bar W=-\dfrac{\rk'}{\rk-\e\eta}F(\eta)\cos\phi\dfrac{\p \ub_0}{\p\phi},\\\rule{0ex}{1.5em}
W (0,\tau,\phi)=\dfrac{\p\gb}{\p\tau}(\tau,\phi)-\dfrac{\p\mathscr{F}_0}{\p\tau}(\tau)\ \ \text{for}\ \
\sin\phi>0,\\\rule{0ex}{1.5em}
W (L,\tau,\phi)=W (L,\tau,\rr[\phi]),
\end{array}
\right.
\end{eqnarray}
where $\rk'$ represents the $\theta$ derivative of $\rk$. Here we need the regularity estimates of $\ub_0$.

Based on Theorem \ref{pt theorem 2}, we know
\begin{eqnarray}
\\
\lnnm{\ue^{K_0\eta}F(\eta)\cos\phi\frac{\p\ub_0}{\p\phi}}
&\leq&C\abs{\ln(\e)}^8\bigg(\lnm{\gb}+\lnm{(\e+\zeta)\dfrac{\p\gb}{\p\phi}}+\lnnm{\ue^{K_0\eta}\ub_0}\bigg)\leq C\abs{\ln(\e)}^8.\no
\end{eqnarray}
Note that here although $\lnm{\dfrac{\p\gb}{\p\phi}}\leq C\e^{-\alpha}$, with the help of $\e+\zeta$, we can get rid of this negative power. Therefore, by Theorem \ref{Milne theorem 2}, we have
\begin{eqnarray}
\lnnm{\ue^{K_0\eta}W}\leq C\abs{\ln(\e)}^8.
\end{eqnarray}
\ \\
Step 3: Well-Posedness of $\ub_1$.\\
$\ub_1$ satisfies the $\e$-Milne problem with geometric correction
\begin{eqnarray}
\left\{
\begin{array}{l}
\sin\phi\dfrac{\p \ub_1 }{\p\eta}+F(\eta)\cos\phi\dfrac{\p
\ub_1 }{\p\phi}+\ub_1 -\bub_1 =\dfrac{W}{\rk-\e\eta}\cos\phi,\\\rule{0ex}{1.5em}
\ub_1 (0,\tau,\phi)=\vw\cdot\nx\u_0(\vx_0,\vw)-\mathscr{F} _{1,L}(\tau)\ \ \text{for}\ \
\sin\phi>0,\\\rule{0ex}{1.5em}
\ub_1 (L,\tau,\phi)=\ub_1 (L,\tau,\rr[\phi]).
\end{array}
\right.
\end{eqnarray}
Therefore, by Theorem \ref{Milne theorem 2}, we know
\begin{eqnarray}
\lnnm{\ue^{K_0\eta}\ub_1}\leq C\lnnm{\ue^{K_0\eta}W}\leq C\abs{\ln(\e)}^8.
\end{eqnarray}
\ \\
Step 4: Tangential Derivatives of $\ub_1$.\\
The $\tau$ derivative $V=\dfrac{\p\ub_1}{\p\tau}$ satisfies
\begin{eqnarray}
\left\{
\begin{array}{l}
\sin\phi\dfrac{\p V}{\p\eta}+F(\eta)\cos\phi\dfrac{\p V}{\p\phi}+V-\bar V=S_1+S_2+S_3,\\\rule{0ex}{1.5em}
V (0,\tau,\phi)=\dfrac{\p}{\p\tau}\bigg(\vw\cdot\nx\u_0(\vx_0,\vw)-\mathscr{F} _{1,L}(\tau)\bigg)\ \ \text{for}\ \
\sin\phi>0,\\\rule{0ex}{1.5em}
V (L,\tau,\phi)=V (L,\tau,\rr[\phi]),
\end{array}
\right.
\end{eqnarray}
where
\begin{eqnarray}
S_1&=&-\frac{\rk'}{\rk-\e\eta}F(\eta)\cos\phi\frac{\p \ub_1}{\p\phi},\\
S_2&=&-\dfrac{\rk'}{(\rk-\e\eta)^2}W\cos\phi,\\
S_3&=&\dfrac{1}{\rk-\e\eta}\cos\phi\dfrac{\p W}{\p\tau}.
\end{eqnarray}
Based on Theorem \ref{pt theorem 2}, we have
\begin{eqnarray}
\lnnm{\ue^{K_0\eta}S_1}&\leq&C\lnnm{\ue^{K_0\eta}F(\eta)\cos\phi\frac{\p \ub_1}{\p\phi}}\\
&\leq& C\bigg(\lnnm{\ue^{K_0\eta}\dfrac{W}{\rk-\e\eta}\cos\phi}+\lnnm{\ue^{K_0\eta}\zeta\frac{\p}{\p\eta}\bigg(\dfrac{W}{\rk-\e\eta}\cos\phi\bigg)}\bigg)\no\\
&\leq& C\bigg(\lnnm{\ue^{K_0\eta}W}+\lnnm{\ue^{K_0\eta}\zeta\frac{\p W}{\p\eta}}\bigg),\no\\
\lnnm{\ue^{K_0\eta}S_2}&\leq&C\lnnm{\ue^{K_0\eta}\dfrac{\rk'}{(\rk-\e\eta)^2}W\cos\phi}\\
&\leq& C\lnnm{\ue^{K_0\eta}W},\no\\
\lnnm{\ue^{K_0\eta}S_3}&\leq&C\lnnm{\ue^{K_0\eta}\frac{\p W}{\p\tau}}.
\end{eqnarray}
\ \\
Step 5: Tangential Derivatives of $W$.\\
The $\tau$ derivative $Z=\dfrac{\p W}{\p\tau}$ satisfies
\begin{eqnarray}
\left\{
\begin{array}{l}
\sin\phi\dfrac{\p Z}{\p\eta}+F(\eta)\cos\phi\dfrac{\p Z}{\p\phi}+Z-\bar Z=T_1+T_2,\\\rule{0ex}{1.5em}
Z (0,\tau,\phi)=\dfrac{\p^2\gb}{\p\tau^2}(\tau,\phi)-\dfrac{\p^2\mathscr{F}_0}{\p\tau^2}(\tau)\ \ \text{for}\ \
\sin\phi>0,\\\rule{0ex}{1.5em}
Z (L,\tau,\phi)=Z (L,\tau,\rr[\phi]),
\end{array}
\right.
\end{eqnarray}
where
\begin{eqnarray}
T_1&=&-\dfrac{\rk'}{\rk-\e\eta}F(\eta)\cos\phi\dfrac{\p W}{\p\phi},\\
T_2&=&-\dfrac{\p}{\p\tau}\bigg(\dfrac{\rk'}{\rk-\e\eta}\bigg)F(\eta)\cos\phi\dfrac{\p\ub_0}{\p\phi}.
\end{eqnarray}
Based on Theorem \ref{pt theorem 2}, we have
\begin{eqnarray}
\lnnm{\ue^{K_0\eta}T_1}&\leq&C\lnnm{\ue^{K_0\eta}F(\eta)\cos\phi\dfrac{\p W}{\p\phi}},\\
\lnnm{\ue^{K_0\eta}T_2}&\leq&C\lnnm{F(\eta)\cos\phi\dfrac{\p\ub_0}{\p\phi}}\leq C\abs{\ln(\e)}^8.
\end{eqnarray}
Therefore, we have
\begin{eqnarray}
\lnnm{\ue^{K_0\eta}S_3}&\leq& \lnnm{\ue^{K_0\eta}Z}\leq C\abs{\ln(\e)}^8+C\lnnm{\ue^{K_0\eta}F(\eta)\cos\phi\dfrac{\p W}{\p\phi}}.
\end{eqnarray}
In total, we have
\begin{eqnarray}
&&\lnnm{\ue^{K_0\eta}S_1}+\lnnm{\ue^{K_0\eta}S_2}+\lnnm{\ue^{K_0\eta}S_3}\\
&\leq& C\abs{\ln(\e)}^8+C\bigg(\lnnm{\ue^{K_0\eta}\zeta\frac{\p W}{\p\eta}}+\lnnm{\ue^{K_0\eta}F(\eta)\cos\phi\dfrac{\p W}{\p\phi}}\bigg).\no
\end{eqnarray}
Hence, we need the regularity estimate of $W$. However, this cannot be done directly. We will first study the normal derivative of $\ub_0$.\\
\ \\
Step 6: Regularity of Normal Derivative.\\
The normal derivative $A=\dfrac{\p\ub_0}{\p\eta}$ satisfies
\begin{eqnarray}
\left\{
\begin{array}{l}
\sin\phi\dfrac{\p A}{\p\eta}+F(\eta)\cos\phi\dfrac{\p A}{\p\phi}+A-\bar A=\dfrac{\e}{R-\e\eta}F(\eta)\cos\phi\dfrac{\p\ub_0}{\p\phi},\\\rule{0ex}{1.5em}
A (0,\tau,\phi)=\dfrac{1}{\sin\phi}\bigg(F(\eta)\cos\phi\dfrac{\p\gb}{\p\phi}(\tau,\phi)-\gb(0,\tau,\phi)+\bar\ub_0(0,\tau,\phi)\bigg)\ \ \text{for}\ \
\sin\phi>0,\\\rule{0ex}{1.5em}
A (L,\tau,\phi)=A (L,\tau,\rr[\phi]),
\end{array}
\right.
\end{eqnarray}
This is where the cut-off in $\gb$ plays a role.
Based on the construction of $\gb$, we know $\lnm{A(0,\phi,\tau)}\leq C\e^{-\alpha}$ and $\lnm{(\e+\zeta)\dfrac{\p A}{\p\phi}(0,\phi,\tau)}\leq C\e^{-\alpha}$. Therefore, using Theorem \ref{Milne theorem 2}, we have
\begin{eqnarray}
\lnnm{\ue^{K_0\eta}A}&\leq&C\bigg(\lnm{A(0,\phi,\tau)}+\lnnm{\ue^{K_0\eta}F(\eta)\cos\phi\dfrac{\p\ub_0}{\p\phi}}\bigg)\leq C\e^{-\alpha}.
\end{eqnarray}
By Theorem \ref{pt theorem 2}, we know
\begin{eqnarray}
&&\lnnm{\ue^{K_0\eta}\zeta\frac{\p A}{\p\eta}}+\lnnm{\ue^{K_0\eta}F(\eta)\cos\phi\frac{\p A}{\p\phi}}\\
&\leq& C\abs{\ln(\e)}^8\Bigg(\e^{-\alpha}+\lnnm{\ue^{K_0\eta}\frac{\e}{R-\e\eta}F(\eta)\cos\phi\frac{\p \ub_0}{\p\phi}}+\lnnm{\ue^{K_0\eta}\zeta\frac{\p}{\p{\eta}}\bigg(\frac{\e}{R-\e\eta}F(\eta)\cos\phi\frac{\p\ub_0}{\p\phi}\bigg)}\Bigg)\no\\
&\leq& C\abs{\ln(\e)}^8\Bigg(\e^{-\alpha}+\e\lnnm{\ue^{K_0\eta}F(\eta)\cos\phi\frac{\p \ub_0}{\p\phi}}+\e\lnnm{\ue^{K_0\eta}F(\eta)\cos\phi\frac{\p A}{\p\phi}}\Bigg)\no
\end{eqnarray}
Then we may absorb $\lnnm{\ue^{K_0\eta}F(\eta)\cos\phi\dfrac{\p A}{\p\phi}}$ into the left-hand side to obtain
\begin{eqnarray}
&&\lnnm{\ue^{K_0\eta}\zeta\frac{\p A}{\p\eta}}+\lnnm{\ue^{K_0\eta}F(\eta)\cos\phi\frac{\p A}{\p\phi}}
\leq C\e^{-\alpha}\abs{\ln(\e)}^8.
\end{eqnarray}
\ \\
Step 7: Regularity of Tangential Derivative.\\
We turn to the regularity of $W$. Based on Theorem \ref{pt theorem 2}, we have
\begin{eqnarray}
&&\lnnm{\ue^{K_0\eta}\zeta\frac{\p W}{\p\eta}}+\lnnm{\ue^{K_0\eta}F(\eta)\cos\phi\frac{\p W}{\p\phi}}\\
&\leq& C\abs{\ln(\e)}^8\Bigg(\e^{-\alpha}+\lnnm{\ue^{K_0\eta}\dfrac{\rk'}{\rk-\e\eta}F(\eta)\cos\phi\dfrac{\p \ub_0}{\p\phi}}+\lnnm{\ue^{K_0\eta}\zeta\frac{\p}{\p{\eta}}\bigg(\dfrac{\rk'}{\rk-\e\eta}F(\eta)\cos\phi\dfrac{\p \ub_0}{\p\phi}\bigg)}\Bigg)\no\\
&\leq& C\abs{\ln(\e)}^8\Bigg(\e^{-\alpha}+\lnnm{\ue^{K_0\eta}F(\eta)\cos\phi\dfrac{\p \ub_0}{\p\phi}}+\lnnm{\ue^{K_0\eta}F(\eta)\cos\phi\frac{\p A}{\p\phi}}\Bigg)\no\\
&\leq&C\e^{-\alpha}\abs{\ln(\e)}^{16}.\no
\end{eqnarray}
\ \\
Step 8: Synthesis.\\
Using above estimates, we actually have shown that
\begin{eqnarray}
\lnnm{\ue^{K_0\eta}V}\leq C\e^{-\alpha}\abs{\ln(\e)}^{16}.
\end{eqnarray}

\subsection{Analysis of Singular Boundary Layer}

In this subsection, we will justify that the singular boundary layers are all well-defined. We divide it into several steps:\\
\ \\
Step 1: Well-Posedness of $\uf_0$.\\
$\uf_0$ satisfies the $\e$-Milne problem with geometric correction
\begin{eqnarray}
\left\{
\begin{array}{l}
\sin\phi\dfrac{\p \uf_0 }{\p\eta}+F(\eta)\cos\phi\dfrac{\p
\uf_0 }{\p\phi}+\uf_0 -\buf_0 =0,\\\rule{0ex}{1.5em}
\uf_0 (0,\tau,\phi)=\gf(\tau,\phi)-\mathfrak{F} _{0,L}(\tau)\ \ \text{for}\ \
\sin\phi>0,\\\rule{0ex}{1.5em}
\uf_0 (L,\tau,\phi)=\uf_0 (L,\tau,\rr[\phi]).
\end{array}
\right.
\end{eqnarray}
Therefore, by Theorem \ref{Milne theorem 2}, we know
\begin{eqnarray}
\lnnm{\ue^{K_0\eta}\uf_0}\leq C.
\end{eqnarray}
However, this is not sufficient for future use and we need more detailed analysis. We will divide the domain $(\eta,\phi)\in[0,L]\times[-\pi,\pi)$ into two regions:
\begin{itemize}
\item
Region I $\chi_1$: $0\leq\zeta<2\e^{\alpha}$.
\item
Region II $\chi_2$: $2\e^{\alpha}\leq\zeta\leq1$.
\end{itemize}
Here we use $\chi_i$ to represent either the corresponding region or the indicator function. It is easy to see that $\gf=0$ in Region II. Similarly we decompose the solution $\uf_0=\chi_1\uf_0+\chi_2\uf_0=f_1+f_2$ in these two regions. In the following, the estimates for $f_i$ will be restricted to the region $\chi_i$ for $i=1,2$. Using Theorem \ref{Milne lemma 6}, we can easily show that
\begin{eqnarray}
\tnnm{\ue^{K_0\eta}\uf_0}&\leq&C\e^{\alpha}.
\end{eqnarray}
The key to $L^{\infty}$ estimates in Theorem \ref{Milne theorem 3} is Lemma \ref{Milne lemma 3} and Lemma \ref{Milne lemma 5}. Their proofs are basically tracking along the characteristics. Hence, we know
\begin{eqnarray}
\lnnm{\ue^{K_0\eta}\buf_0}&\leq& C\bigg(\e^{\alpha}\ltnm{\ue^{K_0\eta}f_1}+\ltnm{\ue^{K_0\eta}f_2}\bigg)\\
&\leq&C\bigg(\tnnm{\ue^{K_0\eta}\uf_0}+\d\e^{\alpha}\lnnm{\ue^{K_0\eta}f_1}+\d\lnnm{\ue^{K_0\eta}f_2}\bigg).\no
\end{eqnarray}
Thus, considering $\chi_1\gf=\gf$ and $\chi_2\gf=0$, we may directly obtain
\begin{eqnarray}
\lnnm{\ue^{K_0\eta}f_1}&\leq&C\bigg(\lnm{\chi_1\gf}+\lnnm{\ue^{K_0\eta}\buf_0}\bigg)\\
&\leq&C\bigg(\lnm{\chi_1\gf}+\tnnm{\ue^{K_0\eta}\uf_0}+\d\e^{\alpha}\lnnm{\ue^{K_0\eta}f_1}+\d\lnnm{\ue^{K_0\eta}f_2}\bigg)\no\\
&\leq&C\bigg(1+\d\e^{\alpha}\lnnm{\ue^{K_0\eta}f_1}+\d\lnnm{\ue^{K_0\eta}f_2}\bigg),\no\\
\lnnm{\ue^{K_0\eta}f_2}&\leq&C\bigg(\lnm{\chi_2\gf}+\lnnm{\ue^{K_0\eta}\buf_0}\bigg)\\
&\leq&C\bigg(\lnm{\chi_2\gf}+\tnnm{\ue^{K_0\eta}\uf_0}+\d\e^{\alpha}\lnnm{\ue^{K_0\eta}f_1}+\d\lnnm{\ue^{K_0\eta}f_2}\bigg)\no\\
&\leq&C\bigg(\e^{\alpha}+\d\e^{\alpha}\lnnm{\ue^{K_0\eta}f_1}+\d\lnnm{\ue^{K_0\eta}f_2}\bigg).\no
\end{eqnarray}
Letting $\d$ small, absorbing $\lnnm{\ue^{K_0\eta}f_1}$ and $\lnnm{\ue^{K_0\eta}f_2}$, we know
\begin{eqnarray}
\lnnm{\ue^{K_0\eta}f_1}&\leq&C\bigg(1+\d\lnnm{\ue^{K_0\eta}f_2}\bigg),\\
\lnnm{\ue^{K_0\eta}f_2}&\leq&C\bigg(\e^{\alpha}+\d\e^{\alpha}\lnnm{\ue^{K_0\eta}f_1}\bigg).
\end{eqnarray}
Combining them together, we can easily see that
\begin{eqnarray}
\lnnm{\ue^{K_0\eta}f_1}&\leq&C,\\
\lnnm{\ue^{K_0\eta}f_2}&\leq&C\e^{\alpha}.
\end{eqnarray}
In total, we can derive
\begin{eqnarray}
\lnnm{\ue^{K_0\eta}\buf_0}\leq C\e^{\alpha}.
\end{eqnarray}
\ \\
Step 2: Regularity of $\uf_0$.\\
This is very similar to the well-posedness proof, we will also consider the regularity of $\uf_0$ in two regions. Note that in the proof of Theorem \ref{pt theorem 2}, the $L^{\infty}$ estimates relies on two kinds of quantities:
\begin{itemize}
\item
$\abs{\zeta\dfrac{\p\uf_0}{\p\eta}}$ on the same characteristics.
\item
$\ds\int_{-\pi}^{\pi}\zeta\dfrac{\p\uf_0}{\p\eta}\ud{\phi}$ for some $\eta>0$.
\end{itemize}
Correspondingly, we may handle them separately: for the first case, since $\zeta$ is preserved along the characteristics, we can directly separate the estimate of $f_1$ and $f_2$; for the second case, we may use the simple domain decomposition
\begin{eqnarray}
\int_{-\pi}^{\pi}\zeta\dfrac{\p\uf_0}{\p\eta}(\eta,\phi)\ud{\phi}&=&\int_{\chi_1}\zeta\dfrac{\p f_1}{\p\eta}\ud{\phi}+\int_{\chi_2}\zeta\dfrac{\p f_2}{\p\eta}\ud{\phi}\leq C\bigg(\e^{\alpha}\ltnm{\zeta\dfrac{\p f_1}{\p\eta}}+\ltnm{\zeta\dfrac{\p f_2}{\p\eta}}\bigg).
\end{eqnarray}
Then following a similar absorbing argument as in above well-posedness proof, we have
\begin{eqnarray}
&&\lnnm{\ue^{K_0\eta}\zeta\frac{\p f_1}{\p\eta}}+\lnnm{\ue^{K_0\eta}F(\eta)\cos\phi\frac{\p f_1}{\p\phi}}\\
&\leq& C\abs{\ln(\e)}^8\bigg(\lnm{\gf}+\lnm{(\e+\zeta)\frac{\p\gf}{\p\phi}}+\lnnm{\ue^{K_0\eta}\uf_0}\bigg)\leq C\abs{\ln(\e)}^8,\no\\
&&\lnnm{\ue^{K_0\eta}\zeta\frac{\p f_2}{\p\eta}}+\lnnm{\ue^{K_0\eta}F(\eta)\cos\phi\frac{\p f_2}{\p\phi}}\\
&\leq& C\abs{\ln(\e)}^8\bigg(\lnnm{\ue^{K_0\eta}f_2}+\e^{\alpha}\lnnm{\ue^{K_0\eta}f_1}\bigg) \leq C\e^{\alpha}\abs{\ln(\e)}^8.\no
\end{eqnarray}
Note that although $\lnm{\dfrac{\p\gf}{\p\phi}}\leq C\e^{-\alpha}$, with the help of $\e+\zeta$, we can get rid of this negative power. \\
\ \\
Step 3: Tangential Derivatives of $\uf_0$.\\
The $\tau$ derivative $P=\dfrac{\p\uf_0}{\p\tau}$ satisfies
\begin{eqnarray}
\left\{
\begin{array}{l}
\sin\phi\dfrac{\p P}{\p\eta}+F(\eta)\cos\phi\dfrac{\p P}{\p\phi}+P-\bar P=-\dfrac{\rk'}{\rk-\e\eta}F(\eta)\cos\phi\dfrac{\p \uf_0}{\p\phi},\\\rule{0ex}{2.0em}
P (0,\tau,\phi)=\dfrac{\p\gf}{\p\tau}(\tau,\phi)-\dfrac{\p\mathfrak{F}_{0,L}}{\p\tau}(\tau)\ \ \text{for}\ \
\sin\phi>0,\\\rule{0ex}{2.0em}
P (L,\tau,\phi)=P (L,\tau,\rr[\phi]).
\end{array}
\right.
\end{eqnarray}
It is easy to check that
\begin{eqnarray}
\int_{-\pi}^{\pi}\cos\phi\dfrac{\p \uf_0}{\p\phi}\ud{\phi}=\int_{-\pi}^{\pi}\uf_0\sin\phi\ud{\phi}=0,
\end{eqnarray}
due to the orthogonal property. Hence, using Theorem \ref{Milne lemma 6} with $S_Q=0$, we have
\begin{eqnarray}
\tnnm{\ue^{K_0\eta}P}&\leq&C\e^{\alpha}\abs{\ln(\e)}^8,
\end{eqnarray}
which further implies
\begin{eqnarray}
\lnnm{\ue^{K_0\eta}P_1}&\leq&C\bigg(\lnnm{\frac{\p\gf}{\p\tau}}+\tnnm{\ue^{K_0\eta}P}+\lnnm{\ue^{K_0\eta}F(\eta)\cos\phi\frac{\p \uf_0}{\p\phi}}\bigg)\\
&\leq& C\abs{\ln(\e)}^8,\no\\
\\
\lnnm{\ue^{K_0\eta}P_2}&\leq&C\bigg(\ue^{K_0\eta}\tnnm{P}+\e^{\alpha}\lnnm{\ue^{K_0\eta}F(\eta)\cos\phi\frac{\p f_1}{\p\phi}}+\lnnm{\ue^{K_0\eta}F(\eta)\cos\phi\frac{\p f_2}{\p\phi}}\bigg)\no\\
&\leq& C\e^{\alpha}\abs{\ln(\e)}^8.\no
\end{eqnarray}
where $P_1=\dfrac{\p f_1}{\p\tau}$ and $P_2=\dfrac{\p f_2}{\p\tau}$.

\subsection{Analysis of Interior Solution}

In this subsection, we will justify that the interior solutions are all well-defined. We divide it into several steps:\\
\ \\
Step 1: Well-Posedness of $\u_0$.\\
$\u_0$ satisfies an elliptic equation
\begin{eqnarray}
\left\{
\begin{array}{l}
\u_0(\vx,\vw)=\bu_0(\vx) ,\\\rule{0ex}{1.5em} \Delta_x\bu_0(\vx)=0\ \ \text{in}\
\ \Omega,\\\rule{0ex}{1.5em}
\bu_0(\vx_0)=\mathscr{F}_{0,L}(\tau)+\mathfrak{F}_{0,L}(\tau)\ \ \text{on}\ \
\p\Omega.
\end{array}
\right.
\end{eqnarray}
Based on standard elliptic theory, we have
\begin{eqnarray}
\nm{\bu_0}_{H^3(\Omega)}\leq C\bigg(\nm{\mathscr{F}_{0,L}}_{H^{\frac{5}{2}}(\p\Omega)}+\nm{\mathfrak{F}_{0,L}}_{H^{\frac{5}{2}}(\p\Omega)}\bigg)\leq C.
\end{eqnarray}
\ \\
Step 2: Well-Posedness of $\u_1$.\\
$\u_1$ satisfies an elliptic equation
\begin{eqnarray}
\left\{
\begin{array}{rcl}
\u_1(\vx,\vw)&=&\bu_1(\vx)-\vw\cdot\nx\u_0(\vx,\vw),\\\rule{0ex}{1.5em}
\Delta_x\bu_1(\vx)&=&-\displaystyle\int_{\s^1}\Big(\vw\cdot\nx\u_{0}(\vx,\vw)\Big)\ud{\vw}\
\ \text{in}\ \ \Omega,\\\rule{0ex}{1.5em} \bu_1(\vx_0)&=&f _{1,L}(\tau)\ \ \text{on}\ \
\p\Omega.
\end{array}
\right.
\end{eqnarray}
Based on standard elliptic theory, we have
\begin{eqnarray}
\nm{\bu_1}_{H^3(\Omega)}\leq C\bigg(\nm{\mathscr{F}_{1,L}}_{H^{\frac{5}{2}}(\p\Omega)}+\nm{\u_0}_{H^{2}(\Omega)}\bigg)\leq C\abs{\ln(\e)}^8.
\end{eqnarray}
\ \\
Step 3: Well-Posedness of $\u_2$.\\
$\u_2$ satisfies an elliptic equation
\begin{eqnarray}
\left\{
\begin{array}{rcl}
\u_{2}(\vx,\vw)&=&\bu_{2}(\vx)-\vw\cdot\nx\u_{1}(\vx,\vw),\\\rule{0ex}{1.5em}
\Delta_x\bu_{2}(\vx)&=&-\displaystyle\int_{\s^1}\Big(\vw\cdot\nx\u_{1}(\vx,\vw)\Big)\ud{\vw}\
\ \text{in}\ \ \Omega,\\\rule{0ex}{1.5em} \bu_2(\vx_0)&=&0\ \ \text{on}\ \
\p\Omega.
\end{array}
\right.
\end{eqnarray}
Based on standard elliptic theory, we have
\begin{eqnarray}
\nm{\bu_2}_{H^3(\Omega)}\leq C\bigg(\nm{\bu_0}_{H^{3}(\Omega)}+\nm{\bu_1}_{H^{2}(\Omega)}\bigg)\leq C\abs{\ln(\e)}^8.
\end{eqnarray}

\subsection{Proof of Main Theorem}

\begin{theorem}\label{diffusive limit}
Assume $g(\vx_0,\vw)\in C^3(\Gamma^-)$. Then for the steady neutron
transport equation (\ref{transport}), there exists a unique solution
$u^{\e}(\vx,\vw)\in L^{\infty}(\Omega\times\s^1)$. Moreover, for any $0<\d<<1$, the solution obeys the estimate
\begin{eqnarray}
\im{u^{\e}-\u-\uu}{\Omega\times\s^1}\leq C(\d)\e^{\frac{1}{2}-\d},
\end{eqnarray}
where $\u(\vx)$ satisfies the Laplace equation with Dirichlet boundary condition
\begin{eqnarray}
\left\{
\begin{array}{l}
\Delta_x\u(\vx)=0\ \ \text{in}\
\ \Omega,\\\rule{0ex}{1.5em}
\u(\vx_0)=D(\vx_0)\ \ \text{on}\ \
\p\Omega,
\end{array}
\right.
\end{eqnarray}
and $\uu(\eta,\tau,\phi)$ satisfies the $\e$-Milne problem with geometric correction
\begin{eqnarray}
\left\{
\begin{array}{l}
\sin\phi\dfrac{\p \uu }{\p\eta}-\dfrac{\e}{\rk(\tau)-\e\eta}\cos\phi\dfrac{\p
\uu }{\p\phi}+\uu -\buu =0,\\\rule{0ex}{1.5em}
\uu (0,\tau,\phi)=g(\tau,\phi)-D(\tau)\ \ \text{for}\ \
\sin\phi>0,\\\rule{0ex}{1.5em}
\uu (L,\tau,\phi)=\uu (L,\tau,\rr[\phi]),
\end{array}
\right.
\end{eqnarray}
for $L=\e^{-\frac{1}{2}}$, $\rr[\phi]=-\phi$, $\eta$ the rescaled normal variable, $\tau$ the tangential variable, and $\phi$ the velocity variable.
\end{theorem}
\begin{proof}
Based on Theorem \ref{LI estimate}, we know there exists a unique $u^{\e}(\vx,\vw)\in L^{\infty}(\Omega\times\s^1)$, so we focus on the diffusive limit. We can divide the proof into several steps:\\
\ \\
Step 1: Remainder definitions.\\
We define the remainder as
\begin{eqnarray}\label{pf 1_}
R&=&u^{\e}-\sum_{k=0}^{2}\e^k\u_k-\sum_{k=0}^{1}\e^k\ub_k-\uf_0=u^{\e}-\q-\qb-\qf,
\end{eqnarray}
where
\begin{eqnarray}
\q&=&\u_0+\e\u_1+\e^2\u_2,\\
\qb&=&\ub_0+\e\ub_1,\\
\qf&=&\uf_0.
\end{eqnarray}
Noting the equation (\ref{transport temp}) is equivalent to the
equation (\ref{transport}), we write $\ll$ to denote the neutron
transport operator as follows:
\begin{eqnarray}
\ll[u]&=&\e\vw\cdot\nx u+ u-\bar u\\
&=&\sin\phi\frac{\p
u}{\p\eta}-\frac{\e}{R_{\kappa}-\e\eta}\cos\phi\bigg(\frac{\p
u}{\p\phi}+\frac{\p u}{\p\tau}\bigg)+ u-\bar u.\nonumber
\end{eqnarray}
\ \\
Step 2: Estimates of $\ll[\q]$.\\
The interior contribution can be estimated as
\begin{eqnarray}
\ll[\q]=\e\vw\cdot\nx \q+ \q-\bar
\q&=&\e^{3}\vw\cdot\nx \u_2.
\end{eqnarray}
Based on classical elliptic estimates, we have
\begin{eqnarray}
\im{\ll[\q]}{\Omega\times\s^1}&\leq&\im{\e^{3}\vw\cdot\nx \u_2}{\Omega\times\s^1}\leq C\e^{3}\im{\nx\u_2}{\Omega\times\s^1}\leq
C\e^{3}\abs{\ln(\e)}^8.
\end{eqnarray}
This implies
\begin{eqnarray}\label{pf 2_}
\tm{\ll[\q]}{\Omega\times\s^1}&\leq& C\e^{3}\abs{\ln(\e)}^8,\\
\nm{\ll[\q]}_{L^{\frac{2m}{2m-1}}(\Omega\times\s^1)}&\leq& C\e^{3}\abs{\ln(\e)}^8,\\
\im{\ll[\q]}{\Omega\times\s^1}&\leq& C\e^{3}\abs{\ln(\e)}^8.
\end{eqnarray}
\ \\
Step 3: Estimates of $\ll \qb$.\\
We need to estimate $\ub_0+\e\ub_1$. The boundary layer contribution can be
estimated as
\begin{eqnarray}\label{remainder temp 1}
\ll[\ub_0+\e\ub_1]&=&\sin\phi\frac{\p
(\ub_0+\e\ub_1)}{\p\eta}-\frac{\e}{R_{\kappa}-\e\eta}\cos\phi\bigg(\frac{\p
(\ub_0+\e\ub_1)}{\p\phi}+\frac{\p(\ub_0+\e\ub_1)}{\p\tau}\bigg)\\
&&+ (\ub_0+\e\ub_1)-
(\bub_0+\e\bub_1)\no\\
&=&-\e^2\frac{1}{R_{\kappa}-\e\eta}\cos\phi\frac{\p
\ub_1}{\p\tau}\nonumber.
\end{eqnarray}
By previous analysis, we have
\begin{eqnarray}
\im{-\e^2\frac{1}{R_{\kappa}-\e\eta}\cos\phi\frac{\p
\ub_1}{\p\tau}}{\Omega\times\s^1}&\leq&C\e^2\im{\frac{\p
\ub_1}{\p\tau}}{\Omega\times\s^1}\leq C\e^{2-\alpha}\abs{\ln(\e)}^8.
\end{eqnarray}
Also, the exponential decay of $\dfrac{\p\ub_1}{\p\tau}$ and the rescaling $\eta=\frac{\mu}{\e}$ implies
\begin{eqnarray}
\tm{-\e^2\frac{1}{R_{\kappa}-\e\eta}\cos\phi\frac{\p
\ub_1}{\p\tau}}{\Omega\times\s^1}&\leq& \e^2\tm{\frac{\p
\ub_1}{\p\tau}}{\Omega\times\s^1}\\
&\leq&\e^2\Bigg(\int_{-\pi}^{\pi}\int_0^{R_{\min}}(R_{\min}-\mu)\lnm{\frac{\p\ub_1}{\p\tau}(\mu,\tau)}^2\ud{\mu}\ud{\tau}\Bigg)^{\frac{1}{2}}\no\\
&\leq&\e^{\frac{5}{2}}\Bigg(\int_{-\pi}^{\pi}\int_0^{\frac{R_{\min}}{\e}}(R_{\min}-\e\eta)\lnm{\frac{\p\ub_1}{\p\tau}(\eta,\tau)}^2\ud{\eta}\ud{\tau}\Bigg)^{\frac{1}{2}}\no\\
&\leq&C\e^{\frac{5}{2}-\alpha}\abs{\ln(\e)}^8\Bigg(\int_{-\pi}^{\pi}\int_0^{\frac{R_{\min}}{\e}}\ue^{-2K_0\eta}\ud{\eta}\ud{\tau}\Bigg)^{\frac{1}{2}}\no\\
&\leq& C\e^{\frac{5}{2}-\alpha}\abs{\ln(\e)}^8.\no
\end{eqnarray}
Similarly, we have
\begin{eqnarray}
\nm{-\e^2\frac{1}{R_{\kappa}-\e\eta}\cos\phi\frac{\p
\ub_1}{\p\tau}}_{L^{\frac{2m}{2m-1}}(\Omega\times\s^1)}&\leq&C\e^{3-\frac{1}{2m}-\alpha}\abs{\ln(\e)}^8.
\end{eqnarray}
In total, we have
\begin{eqnarray}
\tm{\ll[\qb]}{\Omega\times\s^1}&\leq& C\e^{\frac{5}{2}-\alpha}\abs{\ln(\e)}^8,\\
\nm{\ll[\qb]}_{L^{\frac{2m}{2m-1}}(\Omega\times\s^1)}&\leq& C\e^{3-\frac{1}{2m}-\alpha}\abs{\ln(\e)}^8,\\
\im{\ll[\qb]}{\Omega\times\s^1}&\leq& C\e^{2-\alpha}\abs{\ln(\e)}^8.
\end{eqnarray}
\ \\
Step 4: Estimates of $\ll \qf$.\\
We need to estimate $\uf_0$. The boundary layer contribution can be
estimated as
\begin{eqnarray}\label{remainder temp 1.}
\ll[\uf_0]&=&\sin\phi\frac{\p
\uf_0}{\p\eta}-\frac{\e}{R_{\kappa}-\e\eta}\cos\phi\bigg(\frac{\p
\uf_0}{\p\phi}+\frac{\p\uf_0}{\p\tau}\bigg)+ \uf_0-
\buf_0\\
&=&-\e\frac{1}{R_{\kappa}-\e\eta}\cos\phi\frac{\p
\uf_0}{\p\tau}\nonumber.
\end{eqnarray}
By previous analysis, we have
\begin{eqnarray}
\im{-\e\frac{1}{R_{\kappa}-\e\eta}\cos\phi\frac{\p
\uf_0}{\p\tau}}{\Omega\times\s^1}&\leq&C\e\im{\frac{\p
\uf_0}{\p\tau}}{\Omega\times\s^1}\leq C\e\abs{\ln(\e)}^8.
\end{eqnarray}
Also, the exponential decay of $\dfrac{\p\uf_0}{\p\tau}$ and the rescaling $\eta=\dfrac{\mu}{\e}$ implies
\begin{eqnarray}
&&\tm{-\e\frac{1}{R_{\kappa}-\e\eta}\cos\phi\frac{\p
\uf_0}{\p\tau}}{\Omega\times\s^1}
\leq \e\tm{\frac{\p
\uf_0}{\p\tau}}{\Omega\times\s^1}\\
&\leq&\e\Bigg(\int_{-\pi}^{\pi}\int_0^{R_{\min}}\int_{-\pi}^{\pi}\chi_1(R_{\min}-\mu)\lnm{\frac{\p P_1}{\p\tau}(\mu,\tau)}^2\ud{\phi}\ud{\mu}\ud{\tau}\Bigg)^{\frac{1}{2}}\no\\
&&+\e\Bigg(\int_{-\pi}^{\pi}\int_0^{R_{\min}}\int_{-\pi}^{\pi}\chi_2(R_{\min}-\mu)\lnm{\frac{\p P_2}{\p\tau}(\mu,\tau)}^2\ud{\phi}\ud{\mu}\ud{\tau}\Bigg)^{\frac{1}{2}}\no\\
&\leq&\e^{\frac{3}{2}}\Bigg(\int_{-\pi}^{\pi}\int_0^{\frac{R_{\min}}{\e}}\int_{-\pi}^{\pi}\chi_1(R_{\min}-\e\eta)\lnm{\frac{\p P_1}{\p\tau}(\eta,\tau)}^2\ud{\phi}\ud{\eta}\ud{\tau}\Bigg)^{\frac{1}{2}}\no\\
&&+\e^{\frac{3}{2}}\Bigg(\int_{-\pi}^{\pi}\int_0^{\frac{R_{\min}}{\e}}\int_{-\pi}^{\pi}\chi_2(R_{\min}-\e\eta)\lnm{\frac{\p P_2}{\p\tau}(\eta,\tau)}^2\ud{\phi}\ud{\eta}\ud{\tau}\Bigg)^{\frac{1}{2}}\no\\
&\leq&C\e^{\frac{3}{2}+\frac{\alpha}{2}}\abs{\ln(\e)}^{8}\Bigg(\int_{-\pi}^{\pi}\int_0^{\frac{R_{\min}}{\e}}\ue^{-2K_0\eta}\ud{\eta}\ud{\tau}\Bigg)^{\frac{1}{2}}\no\\
&\leq& C\e^{\frac{3}{2}+\frac{\alpha}{2}}\abs{\ln(\e)}^8.\no
\end{eqnarray}
Similarly, we have
\begin{eqnarray}
\nm{-\e\frac{1}{R_{\kappa}-\e\eta}\cos\phi\frac{\p
\uf_0}{\p\tau}}_{L^{\frac{2m}{2m-1}}(\Omega\times\s^1)}&\leq&C\e^{2-\frac{1}{2m}+\alpha}\abs{\ln(\e)}^8.
\end{eqnarray}
In total, we have
\begin{eqnarray}
\tm{\ll[\qb]}{\Omega\times\s^1}&\leq& C\e^{\frac{3}{2}+\frac{\alpha}{2}}\abs{\ln(\e)}^8,\\
\nm{\ll[\qb]}_{L^{\frac{2m}{2m-1}}(\Omega\times\s^1)}&\leq& C\e^{2-\frac{1}{2m}+\frac{(2m-1)\alpha}{2m}}\abs{\ln(\e)}^8,\\
\im{\ll[\qb]}{\Omega\times\s^1}&\leq& C\e\abs{\ln(\e)}^8.
\end{eqnarray}
\ \\
Step 5: Source Term and Boundary Condition.\\
In summary, since $\ll[u^{\e}]=0$, collecting estimates in Step 2 to Step 4, we can prove
\begin{eqnarray}
\tm{\ll[R]}{\Omega\times\s^1}&\leq& C\bigg(\e^{\frac{5}{2}-\alpha}+\e^{\frac{3}{2}+\frac{\alpha}{2}}\bigg)\abs{\ln(\e)}^8,\\
\nm{\ll[R]}_{L^{\frac{2m}{2m-1}}(\Omega\times\s^1)}&\leq& C\bigg(\e^{3-\frac{1}{2m}-\alpha}+\e^{2-\frac{1}{2m}+\frac{(2m-1)\alpha}{2m}}\bigg)\abs{\ln(\e)}^8,\\
\im{\ll[R]}{\Omega\times\s^1}&\leq& C\bigg(\e^{2-\alpha}+\e\bigg)\abs{\ln(\e)}^8.
\end{eqnarray}
We can directly obtain that the boundary data is satisfied up to $O(\e)$, so we know
\begin{eqnarray}
\tm{R}{\Gamma^-}&\leq& C\e^2,\\
\nm{R}_{L^{m}(\Gamma^-)}&\leq&C\e^2,\\
\im{R}{\Gamma^-}&\leq& C\e^2
\end{eqnarray}
\ \\
Step 6: Diffusive Limit.\\
Hence, the remainder $R$ satisfies the equation
\begin{eqnarray}
\left\{
\begin{array}{l}
\e \vw\cdot\nabla_x R+R-\bar R=\ll[R]\ \ \text{for}\ \ \vx\in\Omega,\\
R=R\ \ \text{for}\ \ \vw\cdot\vn<0\ \ \text{and}\ \
\vx_0\in\p\Omega.
\end{array}
\right.
\end{eqnarray}
By Theorem \ref{LI estimate}, we have for $m$ sufficiently large,
\begin{eqnarray}
\im{R}{\Omega\times\s^1}
&\leq& C\bigg(\frac{1}{\e^{1+\frac{1}{m}}}\tm{\ll[R]}{\Omega\times\s^1}+
\frac{1}{\e^{2+\frac{1}{m}}}\nm{\ll[R]}_{L^{\frac{2m}{2m-1}}(\Omega\times\s^1)}+\im{\ll[R]}{\Omega\times\s^1}\\
&&+\frac{1}{\e^{\frac{1}{2}+\frac{1}{m}}}\nm{R}_{L^2(\Gamma^-)}+\frac{1}{\e^{\frac{1}{m}}}\nm{R}_{L^{m}(\Gamma^-)}+\im{R}{\Gamma^-}\bigg)\no,\\
&\leq& C\Bigg(\frac{1}{\e^{1+\frac{1}{m}}}\bigg(\e^{\frac{5}{2}-\alpha}+\e^{\frac{3}{2}+\frac{\alpha}{2}}\bigg)\abs{\ln(\e)}^8+
\frac{1}{\e^{2+\frac{1}{m}}}\bigg(\e^{3-\frac{1}{2m}-\alpha}+\e^{2-\frac{1}{2m}+\frac{(2m-1)\alpha}{2m}}\bigg)\abs{\ln(\e)}^8+(\e)\abs{\ln(\e)}^8\no\\
&&+\frac{1}{\e^{\frac{1}{2}+\frac{1}{m}}}(\e^2)+\frac{1}{\e^{\frac{1}{m}}}(\e^2)+(\e^2)\Bigg)\no\\
&\leq&C\bigg(\e^{1-\frac{3}{2m}-\alpha}+\e^{\frac{(2m-1)\alpha}{2m}-\frac{3}{2m}}\bigg)\abs{\ln(\e)}^8.\no
\end{eqnarray}
Here, we need
\begin{eqnarray}
1-\frac{3}{2m}-\alpha>0,\quad\frac{(2m-1)\alpha}{2m}-\frac{3}{2m}>0,
\end{eqnarray}
which means
\begin{eqnarray}
\frac{3}{2m-1}<\alpha<1-\frac{3}{2m}.
\end{eqnarray}
For $m>3$, this is always achievable. Also, we know
\begin{eqnarray}
\min_{\alpha}\left\{\e^{1-\frac{3}{2m}-\alpha}+\e^{\frac{(2m-1)\alpha}{2m}-\frac{3}{2m}}\right\}=2\e^{\frac{4m^2-14m+3}{8m^2-2m}}\leq C(\d)\e^{\frac{1}{2}-\delta}.
\end{eqnarray}
Note that the constant $C$ might depend on $m$ and thus depend on $\d$.
Since it is easy to see
\begin{eqnarray}
\im{\sum_{k=1}^{2}\e^k\u_k+\sum_{k=1}^{1}\e^k\ub_k}{\Omega\times\s^1}\leq C\e,
\end{eqnarray}
our result naturally follows. We simply take $\u=\u_0$ and $\uu=\ub_0+\uf_0$. It is obvious that $\uu$ satisfies the $\e$-Milne problem with geometric correction with the full boundary data $g(\phi,\tau)-\mathscr{F}_{0,L}(\tau)-\mathfrak{F}_{0,L}(\tau)$. This completes the proof of main theorem.
\end{proof}

\newpage


\bibliographystyle{siam}
\bibliography{Reference}

\end{document}